\documentclass[xcolor=x11names,reqno,12pt]{amsart}
 \usepackage{amsmath}
 \usepackage{amsthm}
 \usepackage{amssymb}
 \usepackage{latexsym,longtable}
 \usepackage{graphicx}
 \usepackage{multicol}
 \usepackage{mathrsfs}
 \usepackage{young}
 \usepackage[vcentermath,enableskew,stdtext]{youngtab}
 \usepackage[left=1.1in,right=1.1in,top=1.08in,bottom=1.1in]{geometry}
 \usepackage[all]{xy}
    \SelectTips{cm}{10}     
    \everyxy={<2.5em,0em>:} 
 \usepackage{fancyhdr}      
 \usepackage{multicol}
 \usepackage{multirow}
 \usepackage{enumitem}
 \usepackage{bm}
 \usepackage{stmaryrd}
 \usepackage{etex}
 \usepackage{tikz}
 \usepackage{float}
 \usepackage{array}
 \usepackage{colortbl}
 \usepackage{mathtools}
 \restylefloat{figure}
\numberwithin{equation}{section}
 \usetikzlibrary{snakes}
\usepackage{url}
 \usepackage{tabmac}
 \usepackage{ytableau}

\usepackage{amsfonts,epsfig,color}
\makeatletter
\newcommand*{\centerfloat}{%
  \parindent \z@
  \leftskip \z@ \@plus 1fil \@minus \textwidth
  \rightskip\leftskip
  \parfillskip \z@skip}
\makeatother
\newcounter{ctr}
\theoremstyle{plain}
\newtheorem{theorem}{Theorem}[section]
\newtheorem{lemma}[theorem]{Lemma}
\newtheorem{corollary}[theorem]{Corollary}
\newtheorem{proposition}[theorem]{Proposition}
\newtheorem{conjecture}[theorem]{Conjecture}

\theoremstyle{definition}
\newtheorem{definition}[theorem]{Definition}

\newtheorem{remark}[theorem]{Remark}
\newtheorem{example}[theorem]{Example}

\newcommand{\ignore}[1]{}

\newcommand{\bb}{\ensuremath{\mathbf{H}}}

\newcommand{\CC}{\ensuremath{\mathbb{C}}}

\newcommand{\End}{\text{\rm End}}

\newcommand{\HH}{\ensuremath{H}}

\renewcommand{\hom}{\text{\rm Hom}}

\DeclareMathOperator{\im}{Image}

\renewcommand{\L}{\ensuremath{\mathscr{L}}}

\renewcommand{\O}{\ensuremath{\mathcal{O}}}

\newcommand{\p}{\ensuremath{\mathfrak{p}}}

\newcommand{\QQ}{\ensuremath{\mathbb{Q}}}

\newcommand{\NN}{\ensuremath{\mathbb{N}}}

\newcommand{\sgn}{\text{\rm sgn}}

\newcommand{\fs}{\ensuremath{\mathfrak{s}}}
\newcommand{\fn}{\ensuremath{\mathfrak{n}}}

\newcommand{\ZZ}{\ensuremath{\mathbb{Z}}}

\newcommand{\be}{\begin{equation}}
\newcommand{\ee}{\end{equation}}

\renewcommand{\SS}{\ensuremath{\mathcal{S}}}
\newcommand{\tsr}{\ensuremath{\otimes}}

\newcommand{\eS}{\widehat{\SS}}

\newcommand{\gd}{\ensuremath{\trianglerighteq}} 

\newcommand{\sh}{\text{\rm shape}}

\newcommand{\kskew}{\ensuremath{k\text{\rm -skew}}}

\DeclareMathOperator{\inside}{inside}
\DeclareMathOperator{\outside}{outside}

\DeclareMathOperator{\upp}{up}
\DeclareMathOperator{\down}{down}
\DeclareMathOperator{\chainup}{top}

\DeclareMathOperator{\chaindown}{bot}
\DeclareMathOperator{\downpath}{downpath}
\DeclareMathOperator{\uppath}{uppath}
\DeclareMathOperator{\bpath}{path}

\DeclareMathOperator{\cvr}{cover}
\DeclareMathOperator{\cover}{\textsc{cover}}

\DeclareMathOperator{\Par}{Par}
\DeclareMathOperator{\tpar}{\widetilde{Par}}
\DeclareMathOperator{\spin}{spin}
\DeclareMathOperator{\Span}{span}
\DeclareMathOperator{\bounce}{bounce}
\DeclareMathOperator{\markz}{marks}

\DeclareMathOperator{\Gr}{Gr}
\DeclareMathOperator{\chr}{ch}

\newcommand{\eroot}[1]{\ensuremath{\varepsilon_{#1}}}

\newcommand{\mynone}{~}

\newcommand{\charge}{\text{charge}}

\DeclareMathOperator{\sort}{sort}

\def\Tiny{\fontsize{6pt}{6pt}\selectfont}

\newcommand{\crc}[1]{#1\star}

\newcommand{\core}{\ensuremath{\mathfrak{c}}}

\newcommand{\tL}{\ensuremath{\tilde{L}}}
\newcommand{\beperp}{\ensuremath{L}}

\newcommand{\VSMT}{\text{\rm VSMT}}
\newcommand{\SMT}{\text{\rm SMT}}

\newcommand{\RI}{\ensuremath{(\Psi,\gamma)}}

\newlength{\mycellsize}
\newcommand\mytbl[1]{
\vcenter{
\let\\=\cr
\baselineskip=-16000pt \lineskiplimit=16000pt \lineskip=0pt
\halign{&\mytblcell{##}\cr#1\crcr}}}

\newcommand{\mytblcell}[1]{{%
\def \arg{#1}\def \void{}%
\ifx \void \arg
\vbox to \mycellsize{\vfil \hrule width \mycellsize height 0pt}%
\else \unitlength=\mycellsize
\begin{picture}(1,1)
\put(0,0){\makebox(1,1){$#1\vphantom{\crc{#1}}$}}
\put(0,0){\line(1,0){1}}
\put(0,1){\line(1,0){1}}
\put(0,0){\line(0,1){1}}
\put(1,0){\line(0,1){1}}
\end{picture}%
\fi}}

\newcommand{\skl}[3]{\ensuremath{#1 \xrightarrow{#2} #3}}

\newlength{\cellsize}
\newcommand\mytableau[1]{
\vcenter{
\let\\=\cr
\baselineskip=-16000pt \lineskiplimit=16000pt \lineskip=0pt
\halign{&\mytableaucell{##}\cr#1\crcr}}}

\newcommand{\mytableaucell}[1]{{%
\def \arg{#1}\def \void{}%
\ifx \void \arg
\vbox to \cellsize{\vfil \hrule width \cellsize height 0pt}%
\else \unitlength=\cellsize
\begin{picture}(1,1)
\put(0,0){\makebox(1,1){$#1\vphantom{\crc{#1}}$}}
\put(0,0){\line(1,0){1}}
\put(0,1){\line(1,0){1}}
\put(0,0){\line(0,1){1}}
\put(1,0){\line(0,1){1}}
\end{picture}%
\fi}}

\newcommand\boldtableau[1]{
\vcenter{
\let\\=\cr
\baselineskip=-16000pt \lineskiplimit=16000pt \lineskip=0pt
\halign{&\boldtableaucell{##}\cr#1\crcr}}}

\newcommand{\boldtableaucell}[1]{{%
\def \arg{#1}\def \void{}%
\ifx \void \arg
\vbox to \cellsize{\vfil \hrule width \cellsize height 0pt}%
\else \unitlength=\cellsize
\begin{picture}(1,1)
\put(0,0){\makebox(1,1){$\mathbf{#1\vphantom{\crc{#1}}}$}}
\put(0,0){\line(1,0){1}}
\put(0,1){\line(1,0){1}}
\put(0,0){\line(0,1){1}}
\put(1,0){\line(0,1){1}}
\end{picture}%
\fi}}

\title{Catalan functions and $k$-Schur positivity}


\begin{document}

\author{Jonah Blasiak}
\address{Department of Mathematics, Drexel University, Philadelphia, PA 19104}
\email{jblasiak@gmail.com}

\author{Jennifer Morse}
\address{Department of Math, University of Virginia, Charlottesville, VA 22904}
\email{morsej@virginia.edu}

\author{Anna Pun}
\address{Department of Mathematics, Drexel University, Philadelphia, PA 19104}
\email{annapunying@gmail.com}

\author{Daniel Summers}
\address{Department of Mathematics, Drexel University, Philadelphia, PA 19104}
\email{danielsummers72@gmail.com}

\thanks{Authors were supported by NSF Grants DMS-1600391 (J.~B.)
and DMS-1833333  (J.~M.). }

\begin{abstract}
We prove that graded $k$-Schur functions are
 $G$-equivariant Euler characteristics of vector bundles
on the flag variety, settling a conjecture of Chen-Haiman.
We expose a new miraculous shift invariance property of the graded $k$-Schur functions and resolve
the Schur positivity and $k$-branching conjectures
in the strongest possible terms by providing direct combinatorial formulas using strong marked tableaux.
\end{abstract}
\maketitle

\vspace{-2mm}
\section{Introduction}

We resolve conjectures made in~\cite{LLM}, \cite{ChenThesis}, and \cite{LLMSMemoirs2} which were inspired by
problems on Macdonald polynomials.
These remarkable polynomials
form a basis for the ring of symmetric functions over the field  $\QQ(q,t)$.
Their study over the last three decades has generated
an impressive body of research, a prominent focus being the
Macdonald positivity conjecture:
the Schur expansion coefficients
of the (Garsia)
modified Macdonald polynomials $H_\mu(\mathbf{x};q,t)$ lie in  $\mathbb{N}[q,t]$.
This was proved by Haiman  \cite{Haiman1} using geometry of Hilbert schemes, yet
many questions arising in this study remain unanswered.

Lapointe, Lascoux, and Morse \cite{LLM} considerably strengthened the Macdonald positivity conjecture.
They constructed a family of functions and conjectured (i) they form
a basis for the space
$\Lambda^k = \Span_{\QQ(q,t)} \{H_\mu(\mathbf {x};q,t)\}_{\mu_1\leq k}$,
(ii) they are Schur positive,
 and (iii)
the expansion of $H_\mu(\mathbf {x};q,t)\in\Lambda^k$
in this basis has coefficients in $\NN[q,t]$.
The problem of Schur expanding Macdonald polynomials
thus factors into the two positivity problems (ii) and (iii).
Because the intricate construction of these functions lacked in mechanism for proof,
many conjecturally equivalent candidates have since been proposed.
Informally, all these candidates are now called {\it $k$-Schur functions}.

At the forefront of $k$-Schur investigations is
the conjectured {\it  branching property}\,:
\begin{equation}
\label{kbranching}
\text{ the $k+1$-Schur expansion of a $k$-Schur function has coefficients in $\NN[t]$\,.}
\end{equation}
Since a  $k$-Schur function reduces to a Schur function for large  $k$, the iteration of branching implies (ii).
However, every effort to prove that even one of the $k$-Schur candidates satisfies (i) and \eqref{kbranching}, or even (i) and (ii),
over the last decades has failed.

In contrast, the {\it ungraded} case has been more tractable.
The $k$-Schur concept is interesting even when $t=1$ despite needing
generic $t$ for Macdonald polynomial applications.
The existence of an ungraded family of $k$-Schur functions satisfying (the  $t=1$ versions of) (i), (ii), and \eqref{kbranching}
was established for functions $s_\lambda^{(k)}({\bf x})$
defined in~\cite{LMktableau} using chains in weak order of the affine
symmetric group  $\eS_{k+1}$.
It was proven that
they form a basis \break for $\Lambda^k|_{t =1}$ with the Gromov-Witten invariants
for the quantum cohomology of the Grassmannian included in their structure constants~\cite{LMquantum}.
Lam established that the basis $s_\lambda^{(k)}(\mathbf{x})$  represents
Schubert classes in the homology of the affine Grassmannian  $\Gr_{SL_{k+1}}$
of $SL_{k+1}$~\cite{LamSchubert} and gave a geometric proof of \eqref{kbranching} at $t=1$:
the branching property
reflects that the image of a Schubert class
is a positive sum of Schubert classes under an
inclusion $H_*(\Gr_{SL_{k+1}}) \to H_*(\Gr_{SL_{k+2}})$
\cite{LamBulletin}.
Additionally, an algorithm for computing the branching coefficients
using an intricate equivalence relation is worked out in~\cite{LLMSMemoirs1,LLMSMemoirs2}.

Building off the work on the ungraded case,
a  candidate $\{s_\lambda^{(k)}(\mathbf{x};t)\}$ for  $k$-Schur functions was
proposed in \cite{LLMSMemoirs1} by attaching a nonnegative integer called {spin} to {\it strong marked tableaux},
certain chains in the strong (Bruhat) order of $\eS_{k+1}$.
It was conjectured that they satisfy the desired properties (i)--(iii)
and shown \cite{LLMSMemoirs1} that they are equal to the weak order $k$-Schur functions  $s_\lambda^{(k)}(\mathbf{x})$ at $t=1$.

In a different vein, Li-Chung Chen and Mark Haiman~\cite{ChenThesis}
conjectured that $k$-Schur functions are a subclass of a family of symmetric
functions indexed by pairs  $(\Psi, \gamma)$ consisting of an upper order ideal  $\Psi$ of positive roots (of which there are Catalan many) and a weight  $\gamma \in \ZZ^{\ell}$.
These {\it Catalan (symmetric) functions}
can be defined by a Demazure-operator formula, and
are equal
to $GL_\ell$-equivariant Euler characteristics of vector bundles
on the flag variety
by the Borel-Weil-Bott theorem.
Chen-Haiman \cite{ChenThesis} investigated their Schur expansions and
conjectured a positive combinatorial formula when $\gamma_1\geq\gamma_2\geq\cdots$.
Panyushev \cite{Panyushev} studied similar questions and proved a cohomological vanishing theorem to establish
Schur positivity of a large subclass of Catalan functions.

Catalan functions are the modified Hall-Littlewood polynomials
when the ideal of roots $\Psi$ consists of all the positive roots and  $\gamma$ is a partition.
Here their Schur expansion coefficients are
the  \emph{Kostka-Foulkes polynomials} (Lusztig's  $t$-analog of weight multiplicities in type~A),
which have been extensively studied from algebraic, geometric, and combinatorial perspectives
(see, e.g., \cite{Macbook,H3,DLT}).
In the case that $\Psi$ consists of the roots above a block diagonal matrix, the Schur expansion coefficients
are the \emph{generalized Kostka polynomials} investigated by
Broer, Shimozono-Weyman, and others
\cite{BroerNormality, SW, Scyclageposet, SchillingWarnaar, Shimozonoaffine, KSS}. 

Chen-Haiman constructed an ideal of roots associated to each partition  $\lambda$ with  $\lambda_1 \le k$ \break
and conjectured that the associated Catalan functions $\{\fs_\lambda^{(k)}(\mathbf x;t)\}$
are $k$-Schur functions~\cite{ChenThesis}.
A key discovery in our work is an elegant set of ideals of roots which we
show yields the same family $\{\fs_\lambda^{(k)}(\mathbf x;t)\}$.
From there, we prove

\vspace{-1mm}
\begin{itemize}[leftmargin=.8cm]
\item{(Chen-Haiman conjecture)} The polynomials $\fs_\lambda^{(k)}(\mathbf x;t)$
are the $k$-Schur functions $s_\lambda^{(k)}(\mathbf x;t)$.
\item{($k$-Schur branching)} The coefficients of \eqref{kbranching}
are $\sum t^{\spin(T)}$ over certain skew strong tableaux $T$;
this settles Conjecture 1.1 of \cite{LLMSMemoirs2} in the strongest possible terms.
\item{(Schur positive basis)} The functions $s_\lambda^{(k)}(\mathbf x;t)$ are a Schur positive basis of~$\Lambda^k$.
This resolves step (ii) in the route for finding the Schur expansion of Macdonald polynomials and
finally establishes that a  $k$-Schur candidate satisfies (i), (ii), and~\eqref{kbranching}.
\item{(Dual Pieri rule)} The polynomials $\fs_\lambda^{(k)}({\bf x};t)$ satisfy a vertical-strip defining rule.
\item{(Shift invariance)}
$\fs_\lambda^{(k)}(\mathbf{x};t)=e_\ell^\perp \fs_{\lambda+1^\ell}^{(k+1)}(\mathbf{x};t)$.
This powerful new discovery follows easily from our definition of $\fs_\lambda^{(k)}(\mathbf x;t)$
and, together with the dual Pieri rule, it immediately yields our $k$-Schur branching rule.
\end{itemize}

\pagebreak[3]

\section{Main results}
\label{s main results}
We work in the ring $\Lambda = \QQ(t)[h_1,h_2,\dots]$ of symmetric functions in infinitely many variables
$\mathbf{x} = (x_1, x_2, \dots)$,
where  $h_d = h_d(\mathbf{x})=\sum_{i_1\leq\cdots \leq i_d} x_{i_1}\cdots x_{i_d}$.
The Schur functions $s_\lambda$ indexed by partitions $\lambda$ form a basis for $\Lambda$.
Schur functions may be defined more generally for any $\gamma \in \ZZ^\ell$ by the following version of the Jacobi-Trudi formula:
\begin{align}
\label{ed s gamma}
s_\gamma = s_\gamma(\mathbf{x}) = \det( h_{\gamma_i + j-i}(\mathbf{x}) )_{1 \le i, j \le \ell}
\ \in\, \Lambda,
\end{align}
where by convention $h_0(\mathbf{x}) = 1$  and $h_d(\mathbf{x}) = 0$ for  $d < 0$.

Central to our work is a family of symmetric functions investigated by Chen-Haiman \cite{ChenThesis} and
Panyushev \cite{Panyushev}; special cases include all $s_\gamma$, as well as the modified Hall-Littlewood polynomials
and generalizations thereof studied by Broer and Shimozono-Weyman~\cite{BroerNormality, SW}.
They can be described geometrically in terms of cohomology of vector bundles on the flag variety (see Theorem \ref{t cohomology}), or algebraically as follows:

\begin{definition}
\label{d HH gamma Psi}
Fix a positive integer  $\ell$.
A \emph{root ideal} is an upper order ideal
of the poset  $\Delta^+_\ell = \Delta^+  := \{(i,j) \mid 1 \le i < j \le \ell \big\}$
with partial order given by $(a,b) \leq (c,d)$ when $a\geq c$ and $b\leq d$.
Given a root ideal $\Psi \subset \Delta^+_\ell$ and $\gamma\in\mathbb Z^\ell$, the associated
\emph{Catalan function} is
\begin{align}
\label{e d HH gamma Psi}
H(\Psi;\gamma)(\mathbf{x};t) := \prod_{(i,j) \in \Psi} \big(1-tR_{i j}\big)^{-1} s_\gamma(\mathbf{x}) \ \in \, \Lambda,
\end{align}
where the raising operator $R_{i j}$ acts on the subscripts of the $s_\gamma$ by $R_{i j} s_\gamma = s_{\gamma + \epsilon_i - \epsilon_j}$
\break (a discussion of raising operators is given in \S\ref{ss Catalan functions}).
\end{definition}

Let $\Par^k_\ell = \{(\mu_1, \dots, \mu_\ell) \in \ZZ^\ell :  k \ge \mu_1 \ge \cdots \ge  \mu_\ell \ge 0 \}$
denote the set of partitions contained in the $\ell \times k$-rectangle and  $\Par^k$ the set of partitions  $\mu$ with  $\mu_1 \le k$.
The trailing zeros are a useful bookkeeping device for elements of  $\Par^k_\ell$,
whereas for $\Par^k$ we adopt the more common convention:
each  $\mu\in \Par^k$ is identified with $(\mu_1, \dots, \mu_{\ell(\mu)}, 0^i)$ for any  $i \ge 0$, where the \emph{length} $\ell(\mu)$ is
the number of nonzero parts of  $\mu$.

\begin{definition}
\label{d0 catalan kschur}
For  $\mu \in \Par^k_\ell$, define the root ideal
\begin{align}
\Delta^k(\mu) = \{(i,j) \in \Delta^+_\ell \mid  k-\mu_i + i < j\},
\end{align}
and the Catalan function
\begin{align}
\label{d catalan kschur}
&\fs^{(k)}_\mu(\mathbf{x};t) := H(\Delta^k(\mu);\mu)=
\prod_{i=1}^\ell\,\prod_{j=k+1-\mu_i+i}^\ell \big(1-tR_{i j}\big)^{-1} s_\mu(\mathbf{x})
\,.
\end{align}
\end{definition}

We will soon see (Theorem~\ref{c kschur eq catalan}) that the  $\fs^{(k)}_\mu(\mathbf{x};t)$ are the $k$-Schur functions,
which proves a conjecture of Chen-Haiman.  This is a consequence of four fundamental properties of these Catalan functions
described in the next theorem.
Their statement requires the definition of strong marked tableaux, which are certain saturated chains in the strong Bruhat order
for the affine symmetric group $\eS_{k+1}$, together with some extra data.
The precise definition is most readily given in terms of
partitions arising in modular representation theory called $k+1$-cores.
Combinatorial examples are provided in \S\ref{ss example strong tab}.

The \emph{diagram} of a partition $\lambda$ is the subset of boxes
$\{(r,c) \in \ZZ_{\geq 1} \times \ZZ_{\geq 1} \mid c \le \lambda_r\}$
in the plane, drawn in English (matrix-style) notation so that rows (resp. columns)
are increasing from north to south (resp. west to east).
Each box has a \emph{hook length} which counts the
number of boxes below it in its column and weakly to its right in its row.
A {\it $k+1$-core} is a partition with no box of hook length $k+1$.
There is a bijection $\p$ \cite{LMtaboncores}
from the set of $k+1$-cores to $\Par^k$
mapping a $k+1$-core $\kappa$ to the partition  $\lambda$ whose $r$-th row  $\lambda_r$ is
the number of boxes in the  $r$-th row of $\kappa$ having hook length $\le k$.

A \emph{strong cover} $\tau \Rightarrow \kappa$ is a pair of $k+1$-cores such that $\tau \subset \kappa$ and $|\p(\tau)| + 1 = |\p(\kappa)|$.
A \emph{strong marked cover} $\tau \xRightarrow{~~r~~} \kappa$ is a strong cover  $\tau \Rightarrow \kappa$
together with a positive integer~$r$ which is allowed to be the smallest row index of
any connected component of the skew shape $\kappa/\tau$.
Let $\eta=(\eta_1,\eta_2,\ldots)\in \ZZ^\infty_{\ge 0}$ with  $m = |\eta| := \sum_i \eta_i$ finite.
 A \emph{strong marked tableau}
$T$ of \emph{weight} $\eta$ is a sequence of strong marked covers
\[\kappa^{(0)} \xRightarrow{~~r_1~~} \kappa^{(1)} \xRightarrow{~~r_2~~} \cdots  \xRightarrow{~~r_m~~} \kappa^{(m)}\]
such that
$r_{v_i+1} \ge r_{v_i+2} \ge \dots \ge r_{v_i+\eta_i}$  for all $i \ge 1$, where $v_i := \eta_1 + \cdots + \eta_{i-1}$.
A \emph{vertical strong marked tableau} is defined the same way except we require each subsequence $r_{v_i+1} < r_{v_i+2} < \dots < r_{v_i+\eta_i}$ to be strictly increasing rather than weakly decreasing.
We write $\inside(T) = \p(\kappa^{(0)})$ and $\outside(T) = \p(\kappa^{(m)})$.
The set of strong marked tableaux (resp. vertical strong marked tableaux) $T$ of weight $\eta$ with $\outside(T) = \mu$ is denoted $\SMT^k_\eta(\mu)$
(resp. $\VSMT^k_\eta(\mu)$).

The \emph{spin} of a strong marked cover $\tau \xRightarrow{~~r~~} \kappa$ is defined to be $c\cdot (h-1) + N$, where $c$ is the number of connected components of  the skew shape $\kappa/\tau$, $h$ is the height (number of rows) of each component, and $N$ is the number of components entirely contained in rows $>r$.
For a (vertical) strong marked tableau $T$, $\spin(T)$ is defined to be the sum of the spins of the strong marked covers comprising $T$.

For  $f \in \Lambda$,
let  $f^\perp$ be the linear operator on  $\Lambda$ that is adjoint to multiplication by $f$ with respect to the Hall inner product, i.e.,
$\langle f^\perp (g), h\rangle = \langle g, f h \rangle$ for all  $g,h \in \Lambda$.

\begin{theorem}
\label{t three properties}
Fix a positive integer $\ell$.
The Catalan functions $\{\fs^{(k)}_\mu\}_{k \ge 1, \, \mu \in \Par^k_\ell}$,
satisfy the following properties:
\begin{align}
\label{et three properties 1}
&\text{\emph{(horizontal dual Pieri rule)}} &&\text{$h_d^\perp \fs^{(k)}_\mu = \!\! \sum_{T \in \SMT^k_{(d)}(\mu)} t^{\spin(T)} \fs^{(k)}_{\inside(T)}$ \ \  for all $d \ge 0$}\,; \\
\label{et three properties 2}\
&\text{\emph{(vertical dual Pieri rule)}} &&\text{$e_d^\perp \fs^{(k)}_\mu = \!\! \sum_{T \in \VSMT^k_{(d)}(\mu)} t^{\spin(T)} \fs^{(k)}_{\inside(T)}$ \  for all $d  \ge 0$}\,;\\
\label{et three properties 3}
&\text{\emph{(shift invariance)}} 
&&\text{$\fs^{(k)}_{\mu} = e_\ell^\perp \fs^{(k+1)}_{\mu+1^\ell}$}\,;\\[2mm]
\label{et three properties 4}
&\text{\emph{(Schur function stability)}} &&\text{if $k \ge |\mu|$, then $\fs^{(k)}_\mu = s_\mu$}. 
\end{align}
\end{theorem}


Theorem~\ref{t three properties}  has several powerful consequences.
Foremost is that the $\fs^{(k)}_\mu$ are $k$-Schur functions.
More precisely, we adopt the following definition of the  $k$-Schur functions
from \cite[\S9.3]{LLMSMemoirs1} (see also~\cite{AssafBilley,LLMSMemoirs2}):
for  $\mu \in \Par^k$, let
\begin{align}
\label{ec strong marked def kschur 1}
s^{(k)}_\mu({\bf x};t) = \sum_{\eta \in \ZZ^\infty_{\geq 0}, \,  |\eta| = |\mu|} \ \,  \sum_{T \in \SMT^k_{\eta}(\mu)} t^{\spin(T)} \mathbf{x}^\eta.
\end{align}
Among the several conjecturally equivalent definitions,
this one has the advantage that
its $t=1$ specializations
$\{s^{(k)}_\mu(\mathbf{x};1)\}$
are known to agree \cite{LLMSMemoirs1} with a quite different looking combinatorial definition
using weak tableaux from \cite{LMktableau},  and are Schubert classes in  the homology of the
affine Grassmannian  $\Gr_{SL_{k+1}}$ of  $SL_{k+1}$ \cite{LamSchubert}.

\begin{theorem}
\label{c kschur eq catalan}
For $\mu\in \Par^k_\ell$, the $k$-Schur function $s_\mu^{(k)}({\bf x};t)$ is
the Catalan function $\fs^{(k)}_\mu(\mathbf{x};t)$.
\end{theorem}

\begin{proof}
The homogeneous symmetric function basis for $\Lambda$ consists of $h_\lambda=\prod_i h_{\lambda_i}$,
as $\lambda$ ranges over partitions.  For any partition $\lambda$ of size $|\mu|$,
\eqref{et three properties 1} implies
\[\langle \fs^{(k)}_\mu, h_\lambda \rangle  = \langle h_\lambda^\perp \fs^{(k)}_\mu, 1 \rangle
= \sum_{T \in \SMT^k_{\lambda}(\mu)} t^{\spin(T)},
\]
where we have also used \eqref{et three properties 4} to obtain $\fs^{(k)}_{\inside(T)} = \fs^{(k)}_{0^\ell} = 1$ for every $T$ in the sum.
The basis of monomial symmetric functions $\{m_\nu(\mathbf{x})\}$ is dual to
the homogenous basis and thus
\begin{align}
\label{ec strong marked def kschur 2}
\fs^{(k)}_\mu(\mathbf{x};t)\, &= \sum_{\text{partitions } \lambda \text{ of } |\mu|} \  \sum_{T \in \SMT^k_{\lambda}(\mu)} t^{\spin(T)}
m_\lambda(\mathbf{x})
\,.
\end{align}
Let $\eta \in \ZZ^\infty_{\ge 0}$ with  $|\eta| = |\mu|$ and let  $\lambda$ be the partition obtained by sorting $\eta$.
Since the  $h_d^\perp$ pairwise commute,  again using \eqref{et three properties 1} and \eqref{et three properties 4} we obtain
\begin{align}
\label{ec strong marked commute 1}
\sum_{T \in \SMT^k_{\eta}(\mu)} t^{\spin(T)} = (h_{\eta_1}^\perp h_{\eta_2}^\perp \cdots) \fs^{(k)}_\mu = (h_{\lambda_1}^\perp h_{\lambda_2}^\perp \cdots) \fs^{(k)}_\mu =  \sum_{T \in \SMT^k_{\lambda}(\mu)} t^{\spin(T)}\,.
\end{align}
Thus the right sides of \eqref{ec strong marked def kschur 2} and \eqref{ec strong marked def kschur 1} agree, as desired.
\end{proof}

Theorem \ref{c kschur eq catalan} has
several important consequences; we give informal statements now, deferring their precise versions for later.

\begin{corollary}
\label{cor kschur eq catalan}
\
\begin{list}{\emph{(\arabic{ctr})}} {\usecounter{ctr} \setlength{\itemsep}{1pt} \setlength{\topsep}{2pt}}
\item The $k$-Schur functions $s^{(k)}_\mu(\mathbf{x};t)$ defined by \eqref{ec strong marked def kschur 1} are symmetric functions.
\item The  $k$-Schur functions $s^{(k)}_\mu(\mathbf{x};t)$ are equal to Catalan functions defined by Chen-Haiman
via a different root ideal than our  $\Delta^k(\mu)$ (see Section \ref{s Chen root set}). 
\item For $\mu\in\Par^k_\ell$, $s^{(k)}_\mu(\mathbf{x};t)$ is the $GL_\ell$-equivariant
Euler characteristic of a vector bundle  on the flag variety  determined by $\mu$ and $k$.
\item Homology Schubert classes of $\Gr_{SL_{k+1}}$ are equal to the ungraded ($t=1$)  version of this Euler characteristic.
\end{list}
%
\end{corollary}
\begin{proof}
Theorem~\ref{c kschur eq catalan} allows us to work interchangeably with $\fs^{(k)}_\mu(\mathbf{x}; t)$
and $s^{(k)}_\mu(\mathbf x;t)$ and we do so from now on without further mention.
Symmetry follows directly from the definition of the Catalan functions.
We match $\fs^{(k)}_\mu$ to the Catalan functions
studied by Chen-Haiman in Theorem~\ref{t Chen roots} to settle (2).
Statement (3) is proved in Theorem~\ref{t cohomology}, and this in turn implies
(4) using~\cite[Theorem 4.11]{LLMSMemoirs1} and~\cite[Theorem 7.1]{LamSchubert}.
\end{proof}

The realization of the $k$-Schur functions as the subclass~\eqref{d catalan kschur} of Catalan functions
is highly lucrative.  One of the most striking outcomes is Property~\eqref{et three properties 3}; overlooked in prior investigations of $k$-Schur functions, it fully resolves $k$-branching:

\vspace{1.6mm}
\parbox{15cm}{\emph{The coefficients in the $k+1$-Schur expansion of a  $k$-Schur function are none other than
the structure coefficients appearing in the vertical dual Pieri rule~\eqref{et three properties 2}.}}

\begin{theorem}[$k$-Schur branching rule]
\label{t branching}
For $\mu \in \Par^k_\ell$,
the expansion of the  $k$-Schur function $s^{(k)}_\mu$ into $k+1$-Schur functions is given by
\begin{align}
\label{et branching}
s^{(k)}_\mu = \sum_{T \in \VSMT^{k+1}_{(\ell)}(\mu + 1^\ell)} t^{\spin(T)} s^{(k+1)}_{\inside(T)}.
\end{align}
\end{theorem}

The Schur expansion of a $k$-Schur function can then be achieved by incrementally iterating~\eqref{et branching}
until $k$ is large enough to apply~\eqref{et three properties 4}.  Interestingly, a more elegant formula can be derived by
a different combination of~\eqref{et three properties 2},~\eqref{et three properties 3}, and~\eqref{et three properties 4}.

\begin{theorem}[$k$-Schur into Schur]
\label{t k Schur to Schur}
Let $\mu \in \Par^k_\ell$ and set $m = \max(|\mu|-k, 0)$.
The Schur expansion of the $k$-Schur function $s^{(k)}_\mu$ is given by
\begin{align}
\label{et k Schur to Schur}
s^{(k)}_\mu = \sum_{T \in \VSMT^{k+m}_{(\ell^{m})}(\mu + m^\ell)} t^{\spin(T)} s_{\inside(T)}.
\end{align}
\end{theorem}
\begin{proof}
Apply~\eqref{et three properties 3} $m$ times to obtain
\begin{align}
\fs^{(k)}_\mu = (e_\ell^\perp)^{m} \, \fs^{(k+m)}_{\mu+m^\ell}\,.
\end{align}
The vertical dual Pieri rule~\eqref{et three properties 2} then gives the
$(k+m)$-Schur function decomposition and~\eqref{et three properties 4}
ensures this is the Schur function decomposition by the careful choice of $m$.
\end{proof}

See Example \ref{ex branching} for Theorem \ref{t branching} and Example \ref{ex k Schur to Schur} and Figure \ref{f branching to Schur} for  Theorem \ref{t k Schur to Schur}.

Recall from the introduction that the original  $k$-Schur candidate of \cite{LLM} (as well as subsequent candidates) conjecturally satisfies (i)--(ii), i.e., forms
a Schur positive basis for the space
$\Lambda^k = \Span_{\QQ(q,t)} \{H_\mu(\mathbf{x};q,t)\}_{\mu_1\leq k}$.
We settle this conjecture for the $k$-Schur functions of~\eqref{ec strong marked def kschur 1}, thereby giving the first proof that a  $k$-Schur candidate satisfies (i)--(ii).
This follows from Theorem \ref{t k Schur to Schur} together with the next result which
additionally refines statement (i) to give bases for subspaces  $\Lambda^k_\ell$
which depend on  $\ell$ as well as  $k$.

\begin{theorem}
\label{t kschurs are a basis}
For any positive integers $\ell$ and $k$,
the $k$-Schur functions $\{s_\mu^{(k)}(\mathbf x;t)\}_{\mu\in\Par_\ell^k}$ form a basis for the space
$\Lambda^k_\ell =  \Span_{\QQ(t)}\{H_{\mu}(\mathbf{x};t) \mid \mu \in \Par_\ell^k \} \subset \Lambda$.
\end{theorem}

Here $H_{\mu}(\mathbf{x};t) = H_{\mu}(\mathbf{x};0, t)$ are the modified Hall-Littlewood polynomials
(denoted $Q'_\mu(\mathbf{x};t)$ in \cite[p. 234]{Macbook}).
Note that $\QQ(q,t) \otimes_{\QQ(t)} (\sum_\ell \Lambda^k_\ell) = \Lambda^k$ since  $\Lambda^k$
has the alternative description $\Lambda^k = \Span_{\QQ(q,t)}\{H_{\mu}(\mathbf{x}; t) \mid \mu \in \Par^k\}$; 
this follows from the fact that the Schur functions and Macdonald's integral forms $J_\mu(\mathbf{x};q,t) $ 
are related by a triangular change of basis---see, e.g., Equations 6.6, 2.18, and 2.20 of \cite{LMksplit}.

We also give three intrinsic descriptions of the $k$-Schur functions.
Studies of ungraded $k$-Schur functions have been served well by the characterization of $s^{(k)}_\mu({\bf x};1)$ as the unique symmetric functions satisfying~\eqref{et three properties 1}
and~\eqref{et three properties 4} at $t=1$.
We can now show that the $k$-Schur functions  $s^{(k)}_\mu({\bf x};t)$ of~\eqref{ec strong marked def kschur 1}
have this characterization (without the $t=1$ specialization), as well as a  new one coming from the shift invariance
property \eqref{et three properties 3}.

\begin{corollary}
\label{c several descriptions k schur}
The $k$-Schur functions $\{s^{(k)}_\mu\}_{k \ge 1, \, \mu \in \Par^k_\ell}$ are the unique family of symmetric functions
satisfying the following subsets of properties \eqref{et three properties 1}--\eqref{et three properties 4}.
\begin{list}{\emph{(\arabic{ctr})}} {\usecounter{ctr} \setlength{\itemsep}{1pt} \setlength{\topsep}{2pt}}
\item \eqref{et three properties 1} and \eqref{et three properties 4};
\item \eqref{et three properties 2} and \eqref{et three properties 4};
\item \eqref{et three properties 2} for  $d = \ell$, \eqref{et three properties 3}, and \eqref{et three properties 4}.
\end{list}
\end{corollary}
\begin{proof}
The proof of Theorem~\ref{c kschur eq catalan} establishes (1).
By a similar argument, \eqref{et three properties 2} and \eqref{et three properties 4} imply
\begin{align}
\fs^{(k)}_\mu(\mathbf{x};t)\, &= \sum_{\text{partitions } \lambda \text{ of } |\mu|} \  \sum_{T \in \VSMT^k_{\lambda}(\mu)} t^{\spin(T)}
\omega(m_\lambda(\mathbf{x}))
\,,
\end{align}
where $\omega$ is the involution on $\Lambda$ defined by $\omega(e_d) = h_d$ for $d \ge 0$.
This establishes (2).
The properties in (3) determine a unique family of symmetric functions by Theorem~\ref{t k Schur to Schur}.
\end{proof}

\subsection{Outline}
The bulk of our paper is devoted to developing machinery to prove the vertical dual Pieri rule \eqref{et three properties 2}.
The remaining results are fairly straightforward;
we prove  \eqref{et three properties 3}, \eqref{et three properties 4}, and Theorem \ref{t kschurs are a basis} in Section \ref{s:defCat},
and \eqref{et three properties 1} in \S\ref{ss horizontal dual Pieri} as a corollary to \eqref{et three properties 2}.

Here are some highlights from the proof of the vertical dual Pieri rule, many of which
we feel will have further applications beyond this paper.  This also serves to give a rough outline of the proof.
\begin{itemize}
\item 
    We prove that $e_d^\perp (H(\Psi;\gamma)) = \sum_{S \subset [\ell], \, |S| = d} H(\Psi; \gamma - \epsilon_S)$ (Lemma~\ref{l ed perp HH}), which is
    our starting point for evaluating the left side of \eqref{et three properties 2}.
\item To handle the terms $H(\Psi; \gamma - \epsilon_S)$ in this sum, we prove a \emph{$k$-Schur straightening rule} (Theorem \ref{t k schur straightening many}) which shows that analogs of the $\fs^{(k)}_\mu$ indexed by nonpartitions
    are equal to 0 or to a power of  $t$ times $\fs^{(k)}_\nu$ for partition $\nu$.
\item Miraculously, the combinatorics arising in this rule exactly matches that of strong covers (Proposition \ref{p covers agree}). 
\item  To prove the $k$-Schur straightening rule, we develop several tools for working with Catalan functions including
a recurrence which expresses a Catalan function as the sum of two such polynomials with similar root ideals (Proposition \ref{p inductive computation atom}).
 \item To prove \eqref{et three properties 2} by induction on  $d$, we must prove a stronger statement in which the right side of
\eqref{et three properties 2} is replaced by a sum over tableaux which are marked only in rows $\le m$, and the left side is an algebraically
defined generalization of $e_d^\perp \fs^{(k)}_\mu$.
\end{itemize}

\subsection{Combinatorial examples}
\label{ss example strong tab}

\begin{example}
The $5$-core $\kappa = 53221$ and its image $\p(\kappa) = 32221 \in \Par^4$
are
\[
\kappa=
\tiny\tableau[scY]{9&7&4 &2 &1 \\6&4 &1 \\4 &2 \\3 &1 \\1 }
\quad
\mapsto
\quad
\p(\kappa)=
\tiny\tableau[scY]{&&\\&\\&\\&\\\\}\,,
\]
where the boxes of $\kappa$ are labeled by their hook lengths.
\end{example}

A (vertical) strong marked tableau  $T = (\kappa^{(0)} \xRightarrow{~~r_1~~} \kappa^{(1)} \xRightarrow{~~r_2~~} \cdots  \xRightarrow{~~r_m~~} \kappa^{(m)})$ is drawn by
filling each skew shape $\kappa^{(i)}/\kappa^{(i-1)}$ with the entry $i$ and starring the entry in position $(r_i,\kappa^{(i)}_{r_i})$, for all $i \in [m]$.
(This is really the standardization of $T$, but suffices for the examples in this paper as we will always specify the weight  $\eta$ separately.)
Strong marked covers are drawn this way too,
regarding them as strong marked tableaux of weight  $(1)$.

\begin{example}
\label{ex strong cover and spin}
Let $k = 4$. For $\tau = 663331111$ and $\kappa = 665443221$,
$\p(\tau) = 332221111$ and $\p(\kappa) = 222222221$.
Thus $\tau \Rightarrow \kappa$ is a strong cover and it has two distinct markings:
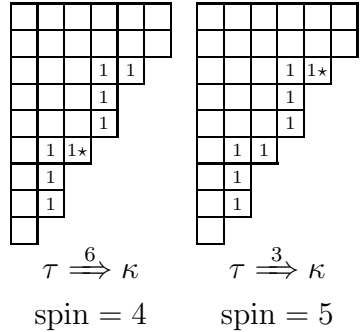
\[
\begin{array}{cc}
{\fontsize{7pt}{5pt}\selectfont \text{\tableau{
~&~&~&~&~&~\\
~&~&~&~&~&~\\
~&~&~&1&1\\
~&~&~&1\\
~&~&~&1\\
~&1&\crc{1}\\
~&1\\
~&1\\
~}}} &
{\fontsize{7pt}{5pt}\selectfont \text{\tableau{
~&~&~&~&~&~\\
~&~&~&~&~&~\\
~&~&~&1&\crc{1}\\
~&~&~&1\\
~&~&~&1\\
~&1&1\\
~&1\\
~&1\\
~}}}\\[3.5mm]
\tau \xRightarrow{~~6~~} \kappa & \tau \xRightarrow{~~3~~} \kappa\\[1mm]
\spin = 4 & \spin = 5
\end{array}\]
\end{example}

\begin{remark}
Although strong marked covers are typically marked by the content of the northeastmost box
of a connected component of $\kappa/\tau$,
it is equivalent (and more natural for us) to use row indices.
\end{remark}

Given a  $k+1$-core  $\kappa$ and $\lambda =\p(\kappa)$,
the  \emph{$k$-skew diagram} of  $\lambda$ denoted $\kskew(\lambda)$,
is the subdiagram of  $\kappa$ consisting of boxes with hook length
$\le k$.  Hence the row lengths  of  $\kskew(\lambda)$ are given  by $\lambda$ itself.

\begin{example}
\label{ex branching}
According to Theorem~\ref{t branching}, the expansion of $s^{(3)}_{22221}$ into 4-Schur functions is obtained by summing $t^{\spin(T)}s^{(4)}_{\inside(T)}$ over the set $\VSMT^4_{(5)}(33332)$ of vertical strong marked tableaux given below.  Note that $86532 = \p^{-1}(33332)$ is the outer shape of each diagram on the first line.
\[
\begin{array}{rllll}
T\quad\text{   }&
{\fontsize{7pt}{5pt}\selectfont
\text{\tableau{~&~&~&~&~&\crc{1}&3&5\\~&~&2&2&\crc{2}&4\\~&~&2&\crc{3}&5\\~&~&\crc{4}\\3&\crc{5}\\}}
}&
{\fontsize{7pt}{5pt}\selectfont
\text{\tableau{~&~&~&~&~&\crc{1}&3&5\\~&~&~&~&\crc{2}&4\\~&~&1&\crc{3}&5\\~&2&\crc{4}\\3&\crc{5}\\}}
}&
{\fontsize{7pt}{5pt}\selectfont
\text{\tableau{~&~&~&~&\crc{1}&3&3&5\\~&~&~&~&\crc{2}&4\\~&1&3&\crc{3}&5\\~&2&\crc{4}\\~&\crc{5}\\}}
}&
{\fontsize{7pt}{5pt}\selectfont
\text{\tableau{~&~&~&~&\crc{1}&3&3&5\\~&~&~&2&\crc{2}&4\\~&~&3&\crc{3}&5\\~&~&\crc{4}\\~&\crc{5}\\}}
}\\[12mm]
\kskew(\inside(T))&
{\fontsize{7pt}{5pt}\selectfont
\text{\tableau{\bl&\bl&~&~&~&\bl&\bl&\bl\\~&~\\~&~\\~&~\\\bl\\\bl\\}}
}&
{\fontsize{7pt}{5pt}\selectfont
\text{\tableau{\bl&\bl&~&~&~&\bl&\bl&\bl\\ \bl&~&~&~\\~&~\\~\\\bl\\\bl\\}}
}&
{\fontsize{7pt}{5pt}\selectfont
\text{\tableau{\bl&~&~&~&\bl&\bl&\bl&\bl\\ \bl&~&~&~\\~\\~\\~\\\bl\\}}
}&
{\fontsize{7pt}{5pt}\selectfont
\text{\tableau{\bl& \bl &~&~&\bl&\bl&\bl&\bl\\ \bl&~&~\\~&~\\~&~\\~\\\bl\\}}
}\\[9mm]
\inside(T)& \ 3222 \quad &\ 3321 \quad &\ 33111 \quad & \ 22221 \quad \\[4mm]
\spin(T) & \ \ 2 \quad & \ \ 2 \quad & \ \ 2 \quad & \ \ 0 \quad \\[4mm]
s^{(3)}_{22221} \ \, =  &  t^2s^{(4)}_{3222} \quad \quad +& t^2s^{(4)}_{3321}\quad \quad +& t^2s^{(4)}_{33111}\quad \quad + & s^{(4)}_{22221}.
\end{array}
\]
\end{example}

%

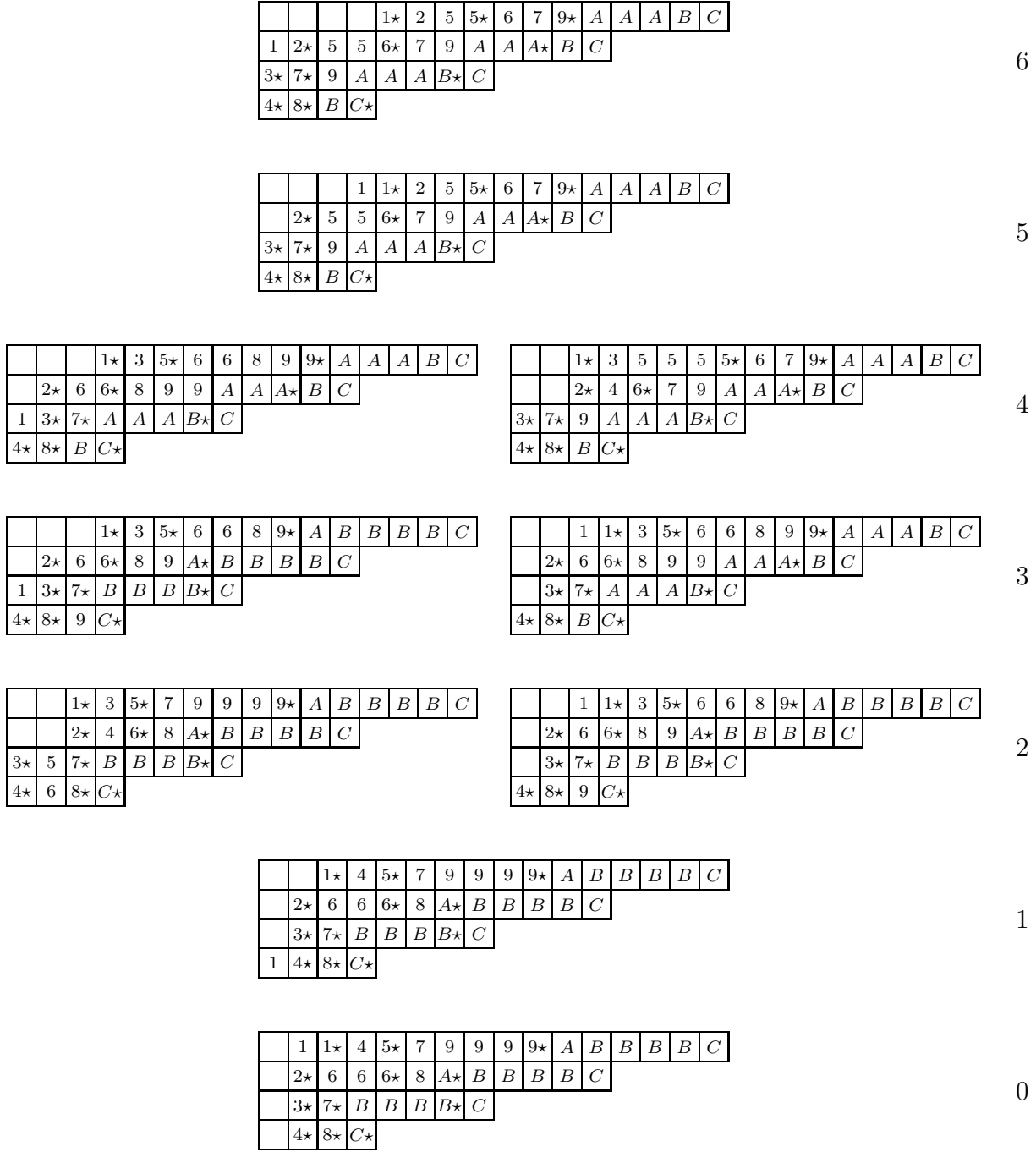
\begin{figure}[H]
\centerfloat
\begin{tikzpicture}[xscale = .43,yscale = 2.7]
\tikzstyle{vertex}=[inner sep=3pt, outer sep=4pt]
\tikzstyle{framedvertex}=[inner sep=3pt, outer sep=4pt]
\tikzstyle{aedge} = [draw, thin, ->,black]
\tikzstyle{edge} = [draw, thick, -,black]
\tikzstyle{doubleedge} = [draw, thick, double distance=1pt, -,black]
\tikzstyle{hiddenedge} = [draw=none, thick, double distance=1pt, -,black]
\tikzstyle{dashededge} = [draw, very thick, dashed, black]
\tikzstyle{LabelStyleH} = [text=black, anchor=south]
\tikzstyle{LabelStyleHn} = [text=black, anchor=north]
\tikzstyle{LabelStyleV} = [text=black, anchor=east]
\setlength{\cellsize}{1.6ex}
\begin{scope}[xscale = 2.3][xshift = 540pt]
\node[vertex] at (0,5) {\fontsize{8pt}{6.5pt}\selectfont $\tableau{
~&~&~&~&\crc{1}&2&5&\crc{5}&6&7&\crc{9}&A&A&A&B&C\\1&\crc{2}&5&5&\crc{6}&7&9&A&A&\crc{A}&B&C\\\crc{3}&\crc{7}&9&A&A&A&\crc{B}&C\\\crc{4}&\crc{8}&B&\crc{C}\\
}$};
\node at (0,4) {\fontsize{8pt}{6.5pt}\selectfont $\tableau{
~&~&~&1&\crc{1}&2&5&\crc{5}&6&7&\crc{9}&A&A&A&B&C\\~&\crc{2}&5&5&\crc{6}&7&9&A&A&\crc{A}&B&C\\\crc{3}&\crc{7}&9&A&A&A&\crc{B}&C\\\crc{4}&\crc{8}&B&\crc{C}\\
}$};
\node at (4,3) {\fontsize{8pt}{6.5pt}\selectfont $\tableau{
~&~&\crc{1}&3&5&5&5&\crc{5}&6&7&\crc{9}&A&A&A&B&C\\~&~&\crc{2}&4&\crc{6}&7&9&A&A&\crc{A}&B&C\\\crc{3}&\crc{7}&9&A&A&A&\crc{B}&C\\\crc{4}&\crc{8}&B&\crc{C}\\
}$};
\node at (-4,3) {\fontsize{8pt}{6.5pt}\selectfont $\tableau{
~&~&~&\crc{1}&3&\crc{5}&6&6&8&9&\crc{9}&A&A&A&B&C\\~&\crc{2}&6&\crc{6}&8&9&9&A&A&\crc{A}&B&C\\1&\crc{3}&\crc{7}&A&A&A&\crc{B}&C\\\crc{4}&\crc{8}&B&\crc{C}\\
}$};
\node at (4,2) {\fontsize{8pt}{6.5pt}\selectfont $\tableau{
~&~&1&\crc{1}&3&\crc{5}&6&6&8&9&\crc{9}&A&A&A&B&C\\~&\crc{2}&6&\crc{6}&8&9&9&A&A&\crc{A}&B&C\\~&\crc{3}&\crc{7}&A&A&A&\crc{B}&C\\\crc{4}&\crc{8}&B&\crc{C}\\
}$};
\node at (-4,2) {\fontsize{8pt}{6.5pt}\selectfont $\tableau{
~&~&~&\crc{1}&3&\crc{5}&6&6&8&\crc{9}&A&B&B&B&B&C\\~&\crc{2}&6&\crc{6}&8&9&\crc{A}&B&B&B&B&C\\1&\crc{3}&\crc{7}&B&B&B&\crc{B}&C\\\crc{4}&\crc{8}&9&\crc{C}\\
}$};
\node at (4,1) {\fontsize{8pt}{6.5pt}\selectfont $\tableau{
~&~&1&\crc{1}&3&\crc{5}&6&6&8&\crc{9}&A&B&B&B&B&C\\~&\crc{2}&6&\crc{6}&8&9&\crc{A}&B&B&B&B&C\\~&\crc{3}&\crc{7}&B&B&B&\crc{B}&C\\\crc{4}&\crc{8}&9&\crc{C}\\
}$};
\node at (-4,1) {\fontsize{8pt}{6.5pt}\selectfont $\tableau{
~&~&\crc{1}&3&\crc{5}&7&9&9&9&\crc{9}&A&B&B&B&B&C\\~&~&\crc{2}&4&\crc{6}&8&\crc{A}&B&B&B&B&C\\\crc{3}&5&\crc{7}&B&B&B&\crc{B}&C\\\crc{4}&6&\crc{8}&\crc{C}\\
}$};
\node at (0,0) {\fontsize{8pt}{6.5pt}\selectfont $\tableau{
~&~&\crc{1}&4&\crc{5}&7&9&9&9&\crc{9}&A&B&B&B&B&C\\~&\crc{2}&6&6&\crc{6}&8&\crc{A}&B&B&B&B&C\\~&\crc{3}&\crc{7}&B&B&B&\crc{B}&C\\1&\crc{4}&\crc{8}&\crc{C}\\
}$};
\node at (0,-1) {\fontsize{8pt}{6.5pt}\selectfont $\tableau{
~&1&\crc{1}&4&\crc{5}&7&9&9&9&\crc{9}&A&B&B&B&B&C\\~&\crc{2}&6&6&\crc{6}&8&\crc{A}&B&B&B&B&C\\~&\crc{3}&\crc{7}&B&B&B&\crc{B}&C\\~&\crc{4}&\crc{8}&\crc{C}\\
}$};
\end{scope}
\begin{scope}[xshift = 550pt]
\node at (0,5.5) {$\text{Spin}$};
\node at (0,5) {6};
\node at (0,4) {5};
\node at (0,3) {4};
\node at (0,2) {3};
\node at (0,1) {2};
\node at (0,0) {1};
\node at (0,-1) {0};
\end{scope}

\vspace{4mm}

\end{tikzpicture}
\caption{\label{f branching to Schur}
According to Theorem \ref{t k Schur to Schur} with  $k=1$, $\ell =4$,
the Schur expansion of the $1$-Schur function $s^{(1)}_{1111}$
(also equal to the modified Hall-Littlewood polynomial $H_{1111}$) is obtained by
summing
$t^{\spin(T)} s_{\inside(T)}$ over the set $\VSMT^4_{(4,4,4)}(4,4,4,4)$ of vertical strong marked tableaux  $T$ given above.
We have written  $A,B,C$ in place  $10,11,12$.
Note that $T$ having weight  $\ell^m = (4,4,4)$ means that the skew shapes (strong covers) labeled by  $1,2,\dots,\ell$ have stars in rows  $1,2, \dots, \ell$, as do the next  $\ell$ skew shapes, and so on.
}
\end{figure}

\begin{example}
\label{ex k Schur to Schur}
We compute the Schur expansion of $s^{(4)}_{3321}$ using the (proof of) Theorem~\ref{t k Schur to Schur}:
it is given by the sum $t^{\spin(T)}s_{\inside(T)}$ over the set $\VSMT^7_{(4,4,4)}(6654)$ of vertical strong marked tableaux given below.
Note that the $m$ in Theorem \ref{t k Schur to Schur} is a convenient choice, but often a smaller $m$ suffices;
the $7$-Schur expansion of $s^{(4)}_{3321}$ is already the Schur expansion, so $m = 3$ (hence $k+m = 7$) suffices.
\[\begin{tikzpicture}[xscale = .42,yscale = 2.3]
\tikzstyle{vertex}=[inner sep=3pt, outer sep=4pt]
\tikzstyle{framedvertex}=[inner sep=3pt, outer sep=4pt]
\tikzstyle{aedge} = [draw, thin, ->,black]
\tikzstyle{edge} = [draw, thick, -,black]
\tikzstyle{doubleedge} = [draw, thick, double distance=1pt, -,black]
\tikzstyle{hiddenedge} = [draw=none, thick, double distance=1pt, -,black]
\tikzstyle{dashededge} = [draw, very thick, dashed, black]
\tikzstyle{LabelStyleH} = [text=black, anchor=south]
\tikzstyle{LabelStyleHn} = [text=black, anchor=north]
\tikzstyle{LabelStyleV} = [text=black, anchor=east]
\setlength{\cellsize}{1.6ex}
\begin{scope}[xscale = 1.8][xshift = 540pt]
\node[vertex] at (0,5) {\fontsize{7.5pt}{6pt}\selectfont $\text{ \tableau{
~&~&~&~&~&\crc{1}&3&\crc{5}&7&\crc{9}&B\\~&~&~&~&\crc{2}&6&\crc{6}&8&\crc{A}&C\\\crc{3}&5&\crc{7}&9&\crc{B}\\\crc{4}&\crc{8}&A&\crc{C}\\}
}$};
\node[vertex] at (-4,4) {\fontsize{7.5pt}{6pt}\selectfont $\text{
\tableau{
~&~&~&~&\crc{1}&4&5&\crc{5}&7&\crc{9}&B\\~&~&~&~&\crc{2}&6&\crc{6}&8&\crc{A}&C\\~&\crc{3}&\crc{7}&9&\crc{B}\\\crc{4}&\crc{8}&A&\crc{C}\\}
}$};
\node[vertex] at (4,4) {\fontsize{7.5pt}{6pt}\selectfont $\text{
\tableau{
~&~&~&~&~&\crc{1}&4&\crc{5}&7&\crc{9}&B\\~&~&~&\crc{2}&\crc{6}&A&A&A&\crc{A}&C\\~&\crc{3}&\crc{7}&9&\crc{B}\\1&\crc{4}&\crc{8}&\crc{C}\\}
}$};
\node[vertex] at (4,3) {\fontsize{7.5pt}{6pt}\selectfont $\text{
\tableau{
~&~&~&~&1&\crc{1}&4&\crc{5}&7&\crc{9}&B\\~&~&~&\crc{2}&\crc{6}&A&A&A&\crc{A}&C\\~&\crc{3}&\crc{7}&9&\crc{B}\\~&\crc{4}&\crc{8}&\crc{C}\\}
}$};
\node[vertex] at (-4,3) {\fontsize{7.5pt}{6pt}\selectfont $\text{
\tableau{
~&~&~&~&\crc{1}&4&\crc{5}&8&9&\crc{9}&B\\~&~&~&\crc{2}&\crc{6}&A&A&A&\crc{A}&C\\~&~&\crc{3}&\crc{7}&\crc{B}\\\crc{4}&5&\crc{8}&\crc{C}\\}
}$};
\node[vertex] at (0,2) {\fontsize{7.5pt}{6pt}\selectfont $\text{
\tableau{
~&~&~&\crc{1}&5&5&\crc{5}&8&9&\crc{9}&B\\~&~&~&\crc{2}&\crc{6}&A&A&A&\crc{A}&C\\~&~&\crc{3}&\crc{7}&\crc{B}\\~&\crc{4}&\crc{8}&\crc{C}\\}
}$};
\end{scope}
\begin{scope}[xshift = 520pt]
\node at (0,5.5) {$\text{Spin}$};
\node at (0,5) {3};
\node at (0,4) {2};
\node at (0,3) {1};
\node at (0,2) {0};
\end{scope}
\end{tikzpicture}\]

\[s^{(4)}_{3321} = t^3s_{54} + t^2(s_{441}+s_{531}) + t(s_{432}+ s_{4311})
 + s_{3321}.
\]
\end{example}

\section{Catalan functions as  $G$-equivariant Euler characteristics}
\label{s cat history}

We first review the geometric description of Catalan functions from \cite{ChenThesis, Panyushev}, and then
summarize prior work on these polynomials.
This serves to provide context for our results but is not necessary for the remainder of the paper.

Let  $G = GL_\ell(\CC)$, $B \subset G$ the standard lower triangular Borel subgroup, and $H \subset B$ the subgroup of
diagonal matrices.
The character group of  $H$ (integral weights) are identified with $\ZZ^\ell$
via the correspondence sending $\gamma \in \ZZ^\ell$ to  the character  $H \to \CC^\times$ given by $\text{diag}(z_1, \dots, z_\ell) \mapsto z_1^{\gamma_1} \cdots z_\ell^{\gamma_\ell}$.

The \emph{character} of a $G$-module $M$ is defined by
\begin{align}
\label{ed character}
\chr(M) = \sum_{\substack{\text{partitions }\lambda \\ \ell(\lambda) \le \ell}}
 \dim \hom(V_\lambda, M) s_\lambda  \in \Lambda,
 \end{align}
where $V_\lambda$ denotes the irreducible $G$-module of highest weight $\lambda$.

Given a $B$-module $N$, let $G\times_B N$ denote the homogeneous $G$-vector bundle on $G/B$ with fiber $N$ above $B \in G/B$,
and let $\L_{G/B}(N)$ denote the locally free $\O_{G/B}$-module of its sections.
For $\gamma \in \ZZ^\ell$, let $\CC_\gamma$ denote the one-dimensional $B$-module of weight $\gamma$.
Consider the adjoint action of $B$ on  the Lie algebra $\fn$ of strictly lower triangular matrices.
The $B$-stable (or ``ad-nilpotent'')
ideals of  $\fn$ are in bijection with root ideals via
the map sending the root ideal $\Psi$ to the $B$-submodule, call it $N_\Psi$, of $\fn$ with weights  $\{\epsilon_j-\epsilon_i \mid (i,j)\in \Psi\}$.
Note that the dual $N_\Psi^*$ has weights given by  $\Psi$.

The Catalan functions appear naturally as certain  $G$-equivariant Euler characteristics, as the following result shows.

\begin{theorem}[\cite{ChenThesis, Panyushev}]
\label{t cohomology}
Let $\RI$ be an indexed root ideal.
Let 
$S^j N_\Psi^*$ denote the $j$-th symmetric power of the  $B$-module  $N_\Psi^*$.
Then
\begin{align}
\label{et cohomology}
H(\Psi;\gamma) =  \sum_{i, j \ge 0} (-1)^i t^j \chr \Big( H^i\big(G/B, \L_{G/B}(S^j N_\Psi^*  \tsr \CC_\gamma) \big) \Big).
\end{align}
\end{theorem}
\begin{proof}
This is a consequence of the  Borel-Weil-Bott theorem and
follows from a straightforward extension of the argument going from Equation 2.1 to
Equation 2.4 in \cite{SW}.
\end{proof}

\begin{remark}
\
\begin{itemize}[leftmargin=1cm]
\item[(1)] A result essentially the same as this one is proved in {\cite[Theorem 3.8]{Panyushev}} by adapting a proof  \cite{Hs1} for the case $\Psi = \Delta^+$.
\item[(2)]
For the precise statement here, we have followed the conventions of \cite{SW}, which conveniently handles a duality in Borel-Weil-Bott.
Note that had we chosen to use the upper triangular Borel subgroup $B'$ instead,
Borel-Weil-Bott implies that
$\chi_{G/B'}(\gamma)  := \sum_{i\ge 0} (-1)^i \chr \big(H^i(G/B', \L_{G/B'}(\CC_{\gamma}) )\big) = s_{(\gamma_\ell, \dots, \gamma_1)}$,
while  $\chi_{G/B}(\gamma) = H(\varnothing; \gamma) = s_\gamma$.
See \cite[\S3]{KamnitzerLectures} for a nice explanation of this duality.
\item[(3)] A version of \eqref{et cohomology} in fact holds for any  $B$-module  $N$, with the left side replaced by a raising operator formula over
the multiset of weights of $N$. However, restricting to the $N_\Psi^*$ is natural from the geometric perspective;
\cite{Panyushev} allows more general (but not arbitrary) $B$-modules  $N$ than the $N_\Psi^*$ considered here.
\item[(4)]
By our definition \eqref{ed s gamma} of  $s_\lambda$, $s_\lambda = 0$ for weakly decreasing  $\lambda \in \ZZ^\ell$ with  $\lambda_\ell < 0$,
so both sides of \eqref{et cohomology} record only the polynomial representations.
This differs from \cite{SW, Panyushev}, which give versions of \eqref{et cohomology} without the polynomial truncation.
\end{itemize}
\end{remark}

\begin{conjecture}[{\cite[Conj.~5.4.3]{ChenThesis}}]
\label{cj chen vanishing}
For any partition  $\gamma$ and root ideal $\Psi$, the cohomology in \eqref{et cohomology} vanishes for $i > 0$ and hence
$H(\Psi;\gamma)$ is a Schur positive symmetric function.
\end{conjecture}

As previously mentioned, the Catalan functions  $H(\Delta^+; \mu)$  for partition $\mu$
are the modified Hall-Littlewood polynomials (see Proposition \ref{p HL Schur and Jing}).  In this case, the Schur expansion coefficients are
the Kostka-Foulkes polynomials, which have been extensively studied
(see, e.g., \cite{Macbook,H3,DLT}).
Another well-studied class of Catalan functions are the
\emph{parabolic Hall-Littlewood polynomials}---the case
$\Psi = \Delta(\eta)$ for some $\eta\in \ZZ_{\geq 0}^r$, where
$$
\Delta(\eta):=\big\{ \text{$\alpha \in \Delta^+_{|\eta|}$ above the block diagonal with block sizes $\eta_1,\ldots,\eta_r$}\big\}\,.
$$
For example,
\[
\Delta(1,3,2) =
\ytableausetup{mathmode, boxsize=.8em,centertableaux}
{\tiny
\begin{ytableau}
*(white)     &*(red)  &*(red)   &*(red)  &*(red)  &*(red) \\
\mynone & *(white) & *(white) & *(white) & *(red)  &*(red)  \\
\mynone &*(white)  & *(white) & *(white) & *(red)  &*(red)  \\
\mynone &*(white)  & *(white)  & *(white) & *(red) &*(red) \\
\mynone &\mynone  &\mynone  &\mynone  & *(white) & *(white) \\
\mynone &\mynone  &\mynone  &\mynone  &*(white)  & *(white)
\end{ytableau}
}.
\]

Broer proved that in the case $\Psi = \Delta^+$,
the cohomology in \eqref{et cohomology} vanishes for  $i > 0$ if and only if $\gamma_i - \gamma_j \ge -1 \text{ for } i < j$
(see \cite[Theorem 2.4]{Broer} and \cite[Proposition 2(iii)]{Broer97}).
Broer posed Conjecture \ref{cj chen vanishing} in the parabolic case  $\Psi = \Delta(\eta)$ (see, e.g., \cite[Conjecture~5]{SW}) and proved it in
the subcase $\gamma$ is a \emph{dominant sequence of rectangles},
meaning that $\gamma=(a_1^{\eta_1},a_2^{\eta_2},\ldots,a_r^{\eta_r})$
for some  $a_1 \ge a_2 \ge \cdots \ge a_r$ \cite[Theorem 2.2]{BroerNormality}.

Panyushev proved that the cohomology in \eqref{et cohomology} vanishes for  $i > 0$ when the weight
$\gamma - \rho + \sum_{(i,j) \in \Delta^+ \setminus \Psi} \epsilon_i- \epsilon_j$ is weakly decreasing, where $\rho = (\ell-1, \ell-2, \dots, 0)$
(see \cite[Theorem~3.2]{Panyushev} for the full statement of which this is a special case);
this includes the case where $\gamma$ is a partition with distinct parts and $\Psi$ is any root ideal as well as many instances where $\gamma$ is not a partition.

A natural combinatorial problem arising here is to find a positive combinatorial formula for the Schur expansion of  $H(\Psi; \gamma)$ when  $\gamma$ is a partition.
Shimozono-Weyman posed a (still open) conjecture that the Schur expansion of $H(\Delta(\eta); \gamma)$ when  $\gamma$ is a partition can be described using an intricate
combinatorial procedure called katabolism~\cite{SW}.
Progress has been made in the case  $\gamma$ is a dominant sequence of rectangles:
the Schur expansion was described by
Schilling-Warnaar \cite{SchillingWarnaar} and Shimozono~\cite{Scyclageposet} (independently)
using a cocyclage poset on Littlewood-Richardson tableaux;
by Shimozono \cite{Shimozonoaffine} using affine Demazure crystals;
and by A. N. Kirillov-Schilling-Shimozono \cite{KSS} using rigged configurations.
Chen-Haiman conjectured a generalization of the Shimozono-Weyman katabolism formula to any root ideal  $\Psi$ and  partition $\gamma$ \cite[Conjecture~5.4.3]{ChenThesis}.

\section{Catalan functions}
\label{s:defCat}

We elaborate on the definition of the Catalan functions $H(\Psi;\gamma)$ 
and give another description of these polynomials using Hall-Littlewood vertex operators. 
We then prove
 \eqref{et three properties 3}, \eqref{et three properties 4}, and~Theorem \ref{t kschurs are a basis},
along the way establishing some basic facts about Catalan functions
which
may have further applications beyond this paper.

\subsection{Notation}
\label{ss notation}
Throughout the paper we use the following notation/conventions:
We write $[a,b]$ for the interval $\{i \in \ZZ \mid a \le i \le b\}$ and $[n] := [1,n]$.
For  $i\in [\ell]$, we write  $\epsilon_i \in  \ZZ^\ell$ for the weight with a 1 in position  $i$ and 0's elsewhere,
and for a set $S\subset  [\ell]$,  denote $\epsilon_S = \sum_{i \in S} \epsilon_i$.
We often omit set braces on singleton sets to avoid clutter.

\subsection{Indexed root ideals}
We regard the set $\Delta^+_\ell=
\Delta^+ := \big\{(i,j) \mid 1 \le i < j \le \ell \big\}$ as labels for the set of positive roots of the root system of type $A_{\ell-1}$ (though for brevity we refer to elements of $\Delta^+$ as roots as well).
For  $\alpha = (i,j) \in \Delta^+_\ell$, we write $\eroot{\alpha} = \epsilon_i - \epsilon_j \in \ZZ^\ell$
for the corresponding positive root (not to be confused with  $\epsilon_{\{i,j\}} = \epsilon_i + \epsilon_j$).
As previously mentioned, a
\emph{root ideal} is an upper order ideal of the poset $\Delta^+$ with partial order given by $(a,b) \leq (c,d)$ when $a\geq c$ and $b\leq d$.
We also work with the complement $\Delta^+ \setminus \Psi$, a lower order ideal of  $\Delta^+$.

\begin{figure}[h]
\[\ytableausetup{mathmode, boxsize=1.7em,centertableaux}
\begin{array}{cc}
\tiny
\begin{ytableau}
 ~ & ~ & *(red!80) 1,\!3 & *(red!80) 1,\!4 & *(red!80) 1,\!5 \\
 \mynone &  ~ & ~ & *(red!80) 2,\!4& *(red!80) 2,\!5 \\
 \mynone & \mynone  &  ~ & *(red!80) 3,\!4& *(red!80) 3,\!5\\
 \mynone & \mynone  & \mynone  &  ~ & ~ \\
 \mynone & \mynone  & \mynone  & \mynone  &  ~
\end{ytableau}
&\tiny
\begin{ytableau}
 ~ & *(blue!48) 1,\!2 & *(red!10) & *(red!10) & *(red!10) \\
 \mynone &  ~ & *(blue!48) 2,\!3 & *(red!10)  & *(red!10)\\
 \mynone & \mynone  &  ~ & *(red!10)  & *(red!10)\\
 \mynone & \mynone  & \mynone  &  ~ & *(blue!48) 4,\!5  \\
 \mynone & \mynone  & \mynone  & \mynone  &  ~
\end{ytableau}\\[14mm]
\text{\small $\Psi = \{(1,3),\!(1,4),(1,5), (2,4), (2,5), (3,4), (3,5)\}$ }&
\text{\small $\Delta^+ \setminus \Psi = \{(1,2),(2,3),(4,5)\} $}
\end{array}
\]
\caption{\label{f root set} For  $\ell =5$, a root ideal  $\Psi$ and its complement  $\Delta^+ \setminus \Psi$:}
\end{figure}

An {\it indexed root ideal of length $\ell$} is
a pair $(\Psi, \gamma)$
consisting of a root ideal $\Psi\subset\Delta_\ell^+$ and a weight $\gamma\in\ZZ^\ell$.

Given an indexed root ideal  $(\Psi,\gamma)$ of length  $\ell$,
we represent the Catalan function  $H(\Psi; \gamma)$
by the $\ell\times \ell$ grid of boxes (labeled by matrix-style coordinates as in Figure~\ref{f root set}),
with the boxes of $\Psi$ shaded and 
the entries of $\gamma$ written along the diagonal.
For example, with $\Psi$ as in Figure~\ref{f root set} and $\gamma = 33411$,
\[\ytableausetup{mathmode, boxsize=1.1em,centertableaux}
H(\Psi; \gamma) =
\tiny
\begin{ytableau}
 3 & ~ & *(red)   & *(red)   & *(red)   \\
 \mynone &  3 & ~ & *(red) & *(red)  \\
 \mynone & \mynone  &  4 & *(red) & *(red) \\
 \mynone & \mynone  & \mynone  &  1 & ~  \\
 \mynone & \mynone  & \mynone  & \mynone  &  1
\end{ytableau}
.
\]
This is the most convenient way we have found to compute with these polynomials.

\subsection{Catalan functions}
\label{ss Catalan functions}

Recall from \eqref{ed s gamma} that for any $\gamma \in \ZZ^\ell$, the Schur function $s_\gamma$ is defined to be $\det( h_{\gamma_i + j-i} )$.
The  $s_\gamma$ are expressed in the basis of Schur functions indexed by partitions by the following \emph{straightening rule}:
\begin{proposition}[Schur function straightening]
\label{p schur straighten}
For any $\gamma \in \ZZ^\ell$,
\[s_\gamma(\mathbf{x}) =
\begin{cases}
\sgn(\gamma+\rho) s_{\sort(\gamma+\rho) -\rho}(\mathbf{x}) & \text{if $\gamma + \rho$ has distinct nonnegative parts,}\\
0                                                          & \text{otherwise,}
\end{cases}
\]
where $\rho=(\ell-1,\ell-2,\dots,0)$, $\sort(\beta)$ denotes the weakly decreasing sequence obtained by sorting $\beta$,
and $\sgn(\beta)$ denotes sign of the shortest permutation taking $\beta$ to $\sort(\beta)$.
\end{proposition}
\begin{proof}
This is a consequence of the Jacobi-Trudi formula $s_\lambda = \det( h_{\lambda_i + j-i} )$
for partitions $\lambda$ and the row interchange property of the determinant.
\end{proof}

\begin{corollary}
\label{c poly truncation}
For a fixed degree $d \in \ZZ$, the Schur function $s_\gamma$ is nonzero for only finitely many of the weights  $\{\gamma \in \ZZ^\ell : |\gamma| = d\}$.
\end{corollary}

\begin{example}
For $\ell = 4$ and $\gamma = (3,1,2,5)$, $s_{\gamma}(\mathbf{x}) = 0$ since
$\gamma + \rho = (6,3,3,5)$ does not have distinct parts.
For $\gamma = (4,7,1,6)$,
$s_{4716}(\mathbf{x}) = s_{6552}(\mathbf{x})$ since
$\gamma + \rho = (7,9,2,6)$,  $\sgn(\gamma + \rho) = +1$, and
$\sort(\gamma + \rho)  - \rho  = (9,7,6,2)-(3,2,1,0) = (6,5,5,2)$.
\end{example}


Let  $(\Psi, \gamma)$ be an indexed root ideal.
In \eqref{e d HH gamma Psi}, we defined the Catalan function  $H(\Psi;\gamma)$ using raising operators  $R_{ij}$.
Raising operators do not give rise to well-defined operators on  $\Lambda$ (for example  $R_{12} \, s_{12} = s_{21} \ne 0$, but $s_{12} = 0$),
so they should be thought of as acting on the subscripts  $\gamma$ rather than the  $s_\gamma$ themselves.
A precise way to interpret \eqref{e d HH gamma Psi} is
\begin{align}
\label{e d HH gamma Psi formal}
H(\Psi;\gamma)(\mathbf{x};t) = \tilde{\pi} \bigg( \prod_{(i,j) \in \Psi} \big(1-tz_i/z_j\big)^{-1} \mathbf{z}^\gamma \bigg),
\end{align}
where the map $\tilde{\pi}$ is defined by first letting
$\pi: \QQ[z_1^{\pm}, \dots, z_\ell^{\pm}]  \to \QQ[h_1,h_2,\dots]$ be the linear map determined by $\mathbf{z}^\gamma \mapsto s_{\gamma}$,
and then $\tilde{\pi}: (\QQ[z_1^{\pm}, \dots, z_\ell^{\pm}] )[[t]] \to (\QQ[h_1,h_2,\dots])[[t]]$ is its natural extension, given by $\sum_{i\ge 0} f_i t^i \mapsto \sum_{i \ge 0} \pi(f_i) t^i$ for any  $f_i \in \QQ[z_1^{\pm}, \dots, z_\ell^{\pm}]$.
It follows from Corollary~\ref{c poly truncation} that the right side of \eqref{e d HH gamma Psi formal} actually lies in the degree $|\gamma|$ part of the ring of symmetric functions $\Lambda = \QQ(t)[h_1,h_2,\dots]$.

\begin{remark}
The  map $\pi$ is essentially the Demazure operator corresponding to the longest element of  $\SS_\ell$, though the Demazure operator is typically defined as a map from  $\QQ[z_1^{\pm}, \dots, z_\ell^{\pm}] \to \QQ[z_1^{\pm}, \dots, z_\ell^{\pm}]^{\SS_\ell}$ (see, e.g., \cite[\S2.1]{SW}).
\end{remark}

\begin{example}
\label{ex 3321}
With $\ell=4$, $\mu = 3321$, and $\Psi = \{(1,3),(2,4),(1,4)\}$, we have
\begin{align*}
H(\Psi; \mu)
&= (1-tR_{1 3})^{-1}(1-tR_{2 4})^{-1}(1-tR_{1 4})^{-1}s_{3321} \\
&= s_{3321}  +
t(s_{3420}+ s_{4311} + s_{4320})
+ t^2(s_{4410}+ s_{5301} + s_{5310}) \\
&\quad + t^3(s_{63-11}+ s_{5400}+ s_{6300})
+ t^4(s_{64-10}+s_{73-10})\\
&= s_{3321}  +
t(s_{4320}+ s_{4311})
+ t^2(s_{4410}+s_{5310})
+ t^3s_{5400}.
\end{align*}
Proposition~\ref{p schur straighten} is used
to truncate the series to terms  $s_\alpha$ with  $\alpha + \rho \in \ZZ_{\ge 0}^\ell$
for the second equality and it is used again for the third to give
$s_{3420} = s_{5301} = s_{64-10} = s_{73-10} = 0$ and  $s_{63-11} = - s_{6300}$.

This demonstrates the direct way to obtain the Schur expansion from the definition of Catalan functions. In contrast,
Example \ref{ex k Schur to Schur} illustrates
the Schur expansion of $H(\Psi;\mu) = \fs^{(4)}_{3321}$ obtained via Theorem \ref{t k Schur to Schur};
the former involves cancellation whereas the latter is manifestly positive.
%
\end{example}

\subsection{Compositional Hall-Littlewood polynomials}
\label{ss compositional HL}

A useful  alternative description of the Catalan functions involves Garsia's version \cite{Garsiaop} of Jing's Hall-Littlewood vertex operators \cite{Jing}.
These are the symmetric function operators defined for any $m\in \ZZ$ by
\begin{align}
\label{ed Jing}
 \bb_m = \sum_{i,j \ge 0}  (-1)^i t^j h_{m+i+j}(\mathbf{x}) e_i^\perp h_j^\perp \, \in \End(\Lambda)
\end{align}
(see \cite[Definition 2.5.2]{ZabrockiThesis} for this formula for $\bb_m$).
Now for $\gamma \in \ZZ^\ell$, define
\begin{align}
\label{e bb gamma}
\bb_\gamma &=  \bb_{\gamma_1}\bb_{\gamma_2}\cdots\bb_{\gamma_\ell}\qquad
\text{and}\qquad
H_\gamma(\mathbf{x};t)= \bb_\gamma \cdot 1.
\end{align}
When $\mu$ is a partition, $H_\mu(\mathbf{x};t)$ is the modified Hall-Littlewood polynomial  \cite[Theorem~2.1]{Garsiaop}.
For general  $\gamma$, the $H_\gamma(\mathbf{x};t)$ are known as {\it  compositional Hall-Littlewood polynomials}.
They played a key role in the recent proof~\cite{CarMel} of the Shuffle Conjecture~\cite{HHLRU}.
As we now show, any Catalan function can be conveniently expressed in terms of compositional Hall-Littlewood polynomials, making
them central to our work as well.


\begin{proposition}
\label{p HL Schur and Jing}
The compositional Hall-Littlewood polynomials are Catalan functions for the root ideal $\Delta^+$.
That is, for any $\gamma\in\ZZ^\ell$,
\[H(\Delta^+;\gamma) = \prod_{(i,j) \in \Delta^+} \big(1-t R_{ij}\big)^{-1} s_\gamma =  \bb_\gamma \cdot 1 = H_\gamma.
\]
\end{proposition}
\begin{proof}
This follows from Proposition 7 and Remark 2 of \cite{SZ}.
\end{proof}

\begin{proposition}
\label{p H definition CHL}
For any indexed root ideal  $(\Psi, \gamma)$,
\begin{align}
\label{ep H definition CHL}
H(\Psi;\gamma)(\mathbf{x};t) &= \prod_{(i,j) \in \Delta^+ \setminus \Psi} (1-t\mathbf{R}_{i j})H_\gamma(\mathbf{x};t),
\end{align}
where the raising operator $\mathbf{R}_{i j}$ acts on the subscripts of the $H_\gamma$ by $\mathbf{R}_{i j}H_\gamma = H_{\gamma + \epsilon_i - \epsilon_j}$.
\end{proposition}

The right side of \eqref{ep H definition CHL} can be expressed more formally as
\begin{align}
\label{ep H definition CHL formal}
\Phi \bigg( \prod_{(i,j) \in  \Delta^+ \setminus \Psi} \big(1-tz_i/z_j\big) \mathbf{z}^\gamma \bigg),
\end{align}
where $\Phi$ is the linear map $\QQ[z_1^{\pm}, \dots, z_\ell^{\pm}][t] \to \Lambda$ determined by
$t^i \mathbf{z}^\gamma \mapsto t^i H_\gamma$.

\begin{proof}
It follows from Proposition \ref{p HL Schur and Jing} that $\Phi$ is equal to the composition  $\tilde{\pi} \circ D$,
where  $D : \QQ[z_1^{\pm}, \dots, z_\ell^{\pm}][t] \to \QQ[z_1^{\pm}, \dots, z_\ell^{\pm}][[t]]$ is the linear map given by left multiplication by
$\prod_{(i,j) \in \Delta^+} \big(1-t \;\! z_i/z_j\big)^{-1}$.
Using the description \eqref{e d HH gamma Psi formal} of the Catalan function $H(\Psi; \gamma)$, we have
\begin{align}
H(\Psi; \gamma) &=
 \tilde{\pi}\bigg( \prod_{(i,j) \in \Psi} (1-t \;\! z_i/ z_j)^{-1}  \mathbf{z}^\gamma \bigg) \notag\\
&= \tilde{\pi}\bigg( \prod_{(i,j) \in \Delta^+} (1-t \;\! z_i/ z_j)^{-1} \prod_{(i,j) \in \Delta^+ \setminus \Psi} (1-t \;\! z_i/ z_j)  \mathbf{z}^\gamma \bigg) \notag\\
&= \tilde{\pi} \circ D \bigg( \prod_{(i,j) \in \Delta^+ \setminus \Psi} (1-t \;\! z_i/ z_j)  \mathbf{z}^\gamma \bigg), \notag
\end{align}
which agrees with \eqref{ep H definition CHL formal}.
\end{proof}

The Garsia-Jing operator $\bb_m$ at $t=1$ reduces to multiplication
by $h_m(\mathbf{x})$ (this is equivalent to the identity \eqref{e E(-y)H(y) 3} arising in a proof later on).
Hence for any $\gamma\in \ZZ^\ell$,
\begin{align*}
H_{\gamma}(\mathbf{x};1)  = h_\gamma(\mathbf{x}) = h_{\gamma_1}(\mathbf{x})h_{\gamma_2}(\mathbf{x})\cdots h_{\gamma_\ell}(\mathbf{x}),
\end{align*}
where  $h_m(\mathbf{x})=0$ if $m<0$.  This yields the following explicit description of the Catalan functions at  $t=1$ in terms
of homogeneous symmetric functions.

\begin{corollary}
\label{c H gamma t1}
For any indexed root ideal  $(\Psi, \gamma)$,
\begin{align*}
H(\Psi,\gamma)(\mathbf{x};1)  =
\prod_{(i,j) \in \Delta^+ \setminus \Psi} (1-\mathbf{R}_{i j})h_\gamma(\mathbf{x}).
\end{align*}
\end{corollary}

\begin{example}
The Catalan function from Example~\ref{ex 3321} at  $t=1$
is computed using Corollary~\ref{c H gamma t1} as follows:
for $\mu = 3321$ and $\Delta^+ \setminus \Psi = \{(1,2),(2,3),(3,4)\}$,
\begin{align*}
H(\Psi;\mu)(\mathbf{x};1)  &=
 (1 - \mathbf{R}_{1 2})(1 - \mathbf{R}_{2 3})(1-\mathbf{R}_{3 4}) h_{3321}\\
 &= h_{3321} - h_{3330} - h_{3411} - h_{4221} + h_{3420} + h_{4230} + h_{4311} - h_{4320}\\
&= h_{3321} - h_{3330} \hphantom{- 1} \hphantom{h_{3411}} - h_{4221} \hphantom{+ 1} \hphantom{h_{3420}} + h_{4230}\\
&= s_{3321} + s_{4320}+ s_{4311} + s_{4410}+s_{5310} + s_{5400}.
\end{align*}
\end{example}

\subsection{Proof of Property \eqref{et three properties 3}}

We first give relations satisfied by the Garsia-Jing and $e_d^\perp$ operators and
then deduce a general result about the action of $e_d^\perp$ on Catalan functions.
\begin{lemma}
\label{l ed perp Jing}
Let $d \in \ZZ_{\ge 0}$,  $m \in \ZZ$, and $\gamma \in \ZZ^\ell$.  Then
\begin{align}
e_d^\perp \bb_m &= \bb_m e_d^\perp  + \bb_{m-1} e_{d-1}^\perp \,,\\
\label{el ed perp Jing 1}
e_d^\perp \bb_\gamma &= \sum_{S\subset  [\ell], \, |S| \le d} \bb_{\gamma - \epsilon_S} e_{d-|S|}^\perp \,.
\end{align}
\end{lemma}
\begin{proof}
The first relation is
\cite[Equation 5.37]{GP} and the second
 follows from the first by a straightforward induction on $\ell$.
\end{proof}

\begin{lemma}
\label{l ed perp HH}
For $d \in \ZZ_{\ge 0}$ and $\RI$ any indexed root ideal of length $\ell$,
\begin{align}
\label{ec ed perp HH}
e_d^\perp (H(\Psi;\gamma)) = \sum_{S \subset [\ell], \, |S| = d} H(\Psi; \gamma - \epsilon_S).
\end{align}
\end{lemma}
\begin{proof}
Proposition~\ref{p H definition CHL} allows us to express any Catalan function in terms of
compositional Hall-Littlewood polynomials.
The result then follows from Lemma \ref{l ed perp Jing}, \eqref{e bb gamma}, and the fact that $e_i^\perp(1) = 0$ for $i > 0$.
\end{proof}

In particular, letting $d=\ell$ in the lemma gives
$e_\ell^\perp H(\Psi;\gamma) = H(\Psi; \gamma-1^\ell)$.
In turn,  we have
$e_\ell^\perp \fs^{(k+1)}_{\mu+1^\ell} = \fs^{(k)}_{\mu}$
since the root ideals $\Delta^{k}(\mu)$ and $\Delta^{k+1}(\mu + 1^\ell)$ are equal.

\subsection{Proof of Property \eqref{et three properties 4}}

The Catalan functions of length  $\ell$ contain those of length  $\ell-1$ in a natural way.
\begin{proposition}
\label{p trailing zeros}
For an indexed root ideal $(\Psi,\gamma)$ of length  $\ell$ with  $\gamma_\ell = 0$,
\[\quad \qquad H(\Psi;\gamma) = H(\hat{\Psi};\hat{\gamma}),
\quad \text{where $\hat{\Psi} = \{(i, j) \in \Psi \mid j < \ell\}$  and $\hat{\gamma} = (\gamma_1, \ldots, \gamma_{\ell-1})$.}\]
\end{proposition}
\begin{proof}
Using the description of the Catalan functions from Proposition \ref{p H definition CHL}, we have
\[
H(\Psi;\gamma) = \!\! \prod_{(i,j) \in \Delta^+ \setminus \Psi} \!\! (1-t\mathbf{R}_{i j})H_\gamma =
\!\! \prod_{(h,\ell) \in \Delta^+_\ell \setminus \Psi} \!\! (1-t\mathbf{R}_{h \ell})  \!\! \prod_{(i,j) \in \Delta^+_{\ell-1} \setminus \hat{\Psi}} \!\! (1-t\mathbf{R}_{i j}) H_\gamma  =
H(\hat{\Psi};\hat{\gamma}).
\]
The last equality uses that $\bb_0 \cdot 1 = 1$ and  $\bb_m \cdot 1 =0$ for  $m < 0$ to conclude that for any  $\alpha \in \ZZ^\ell$ with  $\alpha_\ell = 0$, we have
$\prod_{(h,\ell) \in \Delta^+_\ell \setminus \Psi} (1-t\mathbf{R}_{h \ell}) H_\alpha = H_{(\alpha_1, \ldots, \alpha_{\ell-1})}$.
\end{proof}

Property \eqref{et three properties 4} now follows easily:
let $\mu \in \Par^k_\ell$ with $|\mu| \le k$.  We need to show  $\fs^{(k)}_\mu = s_\mu$.
Proposition \ref{p trailing zeros} allows us to reduce to the case $\mu_\ell > 0$.
Then we have $k \ge |\mu| \ge \mu_i + \ell -i$ for all  $i \in [\ell]$.
Hence the root ideal $\Delta^k(\mu) = \{(i,j) \in \Delta^+ \mid k-\mu_i + i < j\}$ is empty and $\fs^{(k)}_\mu = H(\Delta^k(\mu); \mu)=  s_\mu$ follows.


\subsection{Proof of Theorem \ref{t kschurs are a basis}}


Define the \emph{dominance partial order}  $\gd$ on  $\ZZ^\ell$ by  $\gamma \gd \delta$ if  $\gamma_1 + \dots + \gamma_i \ge \delta_1 + \dots + \delta_i$ for all  $i \in [\ell]$.
\begin{lemma}
\label{l B dominance}
For $\gamma \in \ZZ^\ell$ and $k = \max(\gamma)$,
$H_\gamma \in \Span_{\QQ(t)}\!\big\{H_\lambda \mid \lambda \in \Par^k_\ell \text{ and } \lambda \gd \gamma \big\}.$
\end{lemma}
\begin{proof}
Since $\bb_m \cdot 1 = 0$ for  $m < 0$, the result is a consequence of the stronger claim:
\[\bb_\gamma \in \Span_{\QQ(t)}\!\big\{\bb_\lambda \mid \lambda \in \ZZ^\ell_{\le k} \text{ is weakly decreasing and } \lambda \gd \gamma \big\}. \]
This follows from repeated application of the identity \cite[Theorem~2.2]{Garsiaop}
\begin{equation}
\label{e bb commute}
\bb_m \bb_n = t \, \bb_{m+1}\bb_{n-1}+ t \, \bb_n\bb_m-\bb_{n-1}\bb_{m+1} \,.
\end{equation}
with  $m < n$,
noting that for  $n = m +1$, we must rearrange to obtain  $\bb_m \bb_{m+1} = t\bb_{m+1}\bb_{m}$, rather than apply \eqref{e bb commute} directly.
\end{proof}

The proof of  Theorem \ref{t kschurs are a basis} now goes as follows:
by
Proposition~\ref{p H definition CHL}, for any $\mu\in \Par_\ell^k$,
\begin{equation}
\label{kschur in Bop}
\fs^{(k)}_\mu = \prod_{i=1}^\ell\,\prod_{j=i+1}^{k-\mu_i+i}
(1-t\mathbf{R}_{ij}) H_\mu \,.
\end{equation}
Consider $H_\gamma$ arising from such a successive application of raising operators to $H_\mu$.
Then  $\gamma_i$ is obtained by adding some amount not exceeding $k-\mu_i$
to $\mu_i$.  Hence  $\gamma \in \ZZ^\ell_{\le k}$, implying
$H_\gamma \in \Span_{\QQ(t)}\!\big\{H_\lambda \mid  \lambda \in \Par^k_\ell \text{ and } \lambda \gd \gamma \big\}$
by Lemma \ref{l B dominance}.
Since each application of a  raising operator strictly increases dominance order, $\gamma \gd \mu$.
It follows that
 $\fs^{(k)}_\mu \in \Lambda^k_\ell$ and
the transition matrix expressing   $\{\fs^{(k)}_\mu\}_{\mu \in \Par^k_\ell}$ in terms of  $\{H_\lambda\}_{\lambda \in \Par^k_\ell}$ is
upper unitriangular with respect to dominance order, implying that the former is a basis for  $\Lambda^k_\ell$.

\section{Recurrences for the Catalan functions}

Computations with Catalan functions are facilitated by recurrences which express a Catalan function as the
sum of two Catalan functions
with similar indexed root ideals.
The bounce graph of a root ideal, defined below, is the natural combinatorial object arising in these computations.

\subsection{Bounce graphs}

We say that $\alpha\in \Psi$ is a \emph{removable root of  $\Psi$} if  $\Psi \setminus \alpha$ is a root ideal
and a root $\beta \in \Delta^+ \setminus \Psi$ is \emph{addable to $\Psi$} if $\Psi \cup \beta$ is a root ideal.

\begin{definition}
Fix a root ideal $\Psi\in\Delta^+_\ell$ and $x\in[\ell]$.
If there is a removable root  $(x,j)$ of  $\Psi$, then define $\down_\Psi(x) = j$; otherwise, $\down_\Psi(x)$ is undefined.
Similarly, if there is a removable root $(i,x)$ of  $\Psi$, then define $\upp_\Psi(x) = i$; otherwise, $\upp_\Psi(x)$ is undefined.
\end{definition}

\begin{definition}
\label{d row chain graph}
The \emph{bounce graph} of a root ideal  $\Psi \subset \Delta^+_\ell$ is the graph on the vertex set $[\ell]$
with edges $(r, \down_\Psi(r))$ for each $r\in [\ell]$ such that $\down_\Psi(r)$ is defined.
The bounce graph of  $\Psi$ is a disjoint union of paths called \emph{bounce paths of  $\Psi$}.

For each vertex $r \in [\ell]$,
distinguish $\chaindown_\Psi(r)$ (resp. $\chainup_\Psi(r)$) to be
the maximum (resp. minimum) element of the bounce path of  $\Psi$ containing  $r$.
For  $a,b \in [\ell]$ in the same bounce path of  $\Psi$ with  $a\le b$, we define
\[ \bpath_\Psi(a,b) = (a, \down_\Psi(a), \down^2_\Psi(a), \dots, b), \]
i.e., the list of indices in this path lying between  $a$ and  $b$.
We also set  $\downpath_\Psi(r) = \bpath(r, \chaindown_\Psi(r))$
and $\uppath_\Psi(r)$ to be the reverse of $\bpath(\chainup_\Psi(r),r)$ for any $r \in [\ell]$.
By a slight abuse of notation, we also write  $\bpath_\Psi(a,b)$, $\downpath_\Psi(r)$, and  $\uppath_\Psi(r)$ for the corresponding sets of indices.
For $b = \down_\Psi^{m}(a)$, the \emph{bounce} from  $a$ to  $b$ is
\[ B_\Psi(a,b) := |\bpath_\Psi(a,b)|-1 = m.\]
\end{definition}

\begin{example}
Examples of $\downpath$, $\uppath$, and bounce for the root ideal $\Psi$ below:
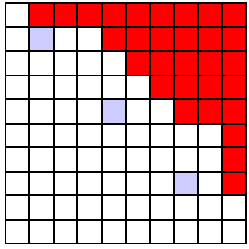
\ytableausetup{mathmode, boxsize=1.03em,centertableaux}
\[\begin{array}{cccccc}
{\tiny \begin{ytableau}
~ & *(red) & *(red)& *(red) & *(red)&*(red)&*(red)&*(red)&*(red)&*(red)\\
~ & *(blue!20) & & & *(red)&*(red)&*(red)&*(red)&*(red)&*(red)\\
~ & & & & &*(red)& *(red) & *(red) &*(red)&*(red)\\
~ & & & & & &*(red)&*(red)&*(red)&*(red)\\
~ & & & & *(blue!20) & & &*(red)&*(red)&*(red)\\
~ & & & & & & & & &*(red)\\
~ & & & & & & & & &*(red) \\
~ & & & & & & & *(blue!20) & & *(red)\\
~ & & & & & & & & & \\
~ & & & & & & & & &
\end{ytableau} } & & &
{\tiny \begin{ytableau}
~ & *(red) & *(red)& *(red) & *(red)&*(red)&*(red)&*(red)&*(red)&*(red)\\
~ & & & & *(red)&*(red)&*(red)&*(red)&*(red)&*(red)\\
~ & & *(blue!20)& & &*(red)& *(red) & *(red) &*(red)&*(red)\\
~ & & & & & &*(red)&*(red)&*(red)&*(red)\\
~ & & & &  & & &*(red)&*(red)&*(red)\\
~ & & & & &*(blue!20) & & & &*(red)\\
~ & & & & & & & & &*(red) \\
~ & & & & & & &  & & *(red)\\
~ & & & & & & & & & \\
~ & & & & & & & & &
\end{ytableau} } & &
{\tiny \begin{ytableau}
*(blue!20) & *(red) & *(red)& *(red) & *(red)&*(red)&*(red)&*(red)&*(red)&*(red)\\
~ & *(blue!20) & & & *(red)&*(red)&*(red)&*(red)&*(red)&*(red)\\
~ & & & & &*(red)& *(red) & *(red) &*(red)&*(red)\\
~ & & & & & &*(red)&*(red)&*(red)&*(red)\\
~ & & & & *(blue!20) & & &*(red)&*(red)&*(red)\\
~ & & & & & & & & &*(red)\\
~ & & & & & & & & &*(red) \\
~ & & & & & & & *(blue!20) & & *(red)\\
~ & & & & & & & & & \\
~ & & & & & & & & & *(blue!20)
\end{ytableau} }\\[16.4mm]
\bpath_{\Psi}(2,8) = \color{blue}{2,5, 8} & & & \downpath_{\Psi}(3) = \color{blue}{3,6} & &  \uppath_{\Psi}(10) = \color{blue}{10, 8, 5, 2,1}
\end{array}\]
\vspace{-1mm}
\[\text{$B_\Psi(2, 8) = 2$, $B_\Psi(1,10) = 4$, $B_\Psi(3,6) = 1$, and $B_\Psi(3,3) = 0$}.\]
\end{example}

\begin{definition}
A root ideal $\Psi$ is said to have

\emph{a wall in rows $r,r+1$} if rows $r$ and $r+1$ of $\Psi$ have the same length,

\emph{a ceiling in columns $c,c+1$} if columns $c$ and $c+1$ of $\Psi$ have the same length, and

\emph{a mirror in rows} $r,r+1$ if $\Psi$ has removable roots $(r,c)$, $(r+1,c+1)$ for some $c > r+1$.
\end{definition}

\begin{example}
The root ideal $\Psi$ in the previous example has
a ceiling in columns  $2,3$, in columns $3,4$,  and in columns  $8,9$,
a wall in rows  $6,7$, in rows  $7,8$, and in rows  $9, 10$, and  a mirror in rows  $2,3$, in rows  $3, 4$, and in rows  $4, 5$.
\end{example}

\subsection{Recurrences for the Catalan functions}

\begin{proposition}
\label{p inductive computation atom}
Let $(\Psi, \mu)$ be an indexed root ideal.
For any root $\beta$ addable to $\Psi$,
\begin{align}\label{e Rlambda recurrence nonroot HH}
\HH(\Psi; \mu) = \HH(\Psi \cup \beta; \mu) - t\;\! \HH(\Psi\cup \beta; \mu+ \eroot{\beta}).
\end{align}
For any removable root $\alpha$ of $\Psi$,
\begin{align}
\label{e Rlambda recurrence HH}
\HH(\Psi;\mu) = \HH(\Psi \setminus \alpha; \mu) + t\;\! \HH(\Psi; \mu+ \eroot{\alpha}). \qquad
\end{align}
\end{proposition}
\begin{proof}
The first identity \eqref{e Rlambda recurrence nonroot HH} follows directly from Proposition \ref{p H definition CHL} :
\begin{align*}
\HH(\Psi; \mu) &= \prod_{(i,j) \in \Delta^+ \setminus \Psi} (1-t\mathbf{R}_{i j}) H_\mu \\
&=  (1-t\mathbf{R}_\beta )\prod_{(i,j) \in \Delta^+ \setminus (\Psi \cup \beta)} (1-t\mathbf{R}_{i j}) H_\mu \\
&=  \prod_{(i,j) \in \Delta^+ \setminus (\Psi \cup \beta)} \!\! (1-t\mathbf{R}_{i j}) H_\mu \, - \, t \prod_{(i,j) \in \Delta^+ \setminus (\Psi \cup \beta)} \!\! (1-t\mathbf{R}_{i j}) H_{\mu+ \eroot{\beta}}.
\end{align*}
The second identity \eqref{e Rlambda recurrence HH} is then obtained by applying \eqref{e Rlambda recurrence nonroot HH} with  $\Psi = \Psi \setminus \alpha$ and  $\beta = \alpha$.
\end{proof}

We also record a convenient application of the recurrence \eqref{e Rlambda recurrence HH} obtained by iterating it along a downpath.

\begin{corollary}
\label{c inductive computation atom down}
Let $(\Psi,\mu)$ be an indexed root ideal of length  $\ell$
and  $m \in [\ell]$.
Then
\begin{align}
\HH(\Psi;\mu) = \sum_{z \in \downpath_\Psi(m)} t^{B_\Psi(m,z)}\HH(\Psi^z;\mu+ \epsilon_m - \epsilon_z),
\label{ec inductive computation atom down}
\end{align}
where $\Psi^z := \Psi \setminus \{(z, \down_\Psi(z))\}$ for $z \neq \chaindown_\Psi(m)$ and $\Psi^{\chaindown_\Psi(m)} := \Psi$.
\end{corollary}
\begin{proof}
The proof is by induction on $|\downpath_\Psi(m)|$.
The base case $|\downpath_\Psi(m)|=1$ is clear.
Now suppose $|\downpath_\Psi(m)| > 1$ and set  $m' = \down_\Psi(m)$.
The desired result is obtained by expanding  $\HH(\Psi;\mu)$ on the root $(m,m')$ using the recurrence \eqref{e Rlambda recurrence HH} and then
applying the inductive hypothesis:
\begin{align*}
\HH(\Psi;\mu)
&=  \HH(\Psi^m; \mu) + t\, \HH(\Psi; \mu+\epsilon_{m} - \epsilon_{m'}) \\
&= \HH(\Psi^m;\mu) + t \sum_{z \in \downpath_\Psi(m')} t^{B_\Psi(m',z)}\HH(\Psi^z;\mu+ \epsilon_{m} - \epsilon_{m'} + \epsilon_{m'} - \epsilon_z) \\
&= \sum_{z \in \downpath_\Psi(m)} t^{B_\Psi(m,z)}\HH(\Psi^z;\mu+ \epsilon_m - \epsilon_z). \qedhere
\end{align*}
\end{proof}

\section{Mirror lemmas}

We first give a natural generalization of Schur function straightening to Catalan functions (Lemma \ref{l little sl2 lemma})
and then deduce two \emph{Mirror Lemmas}.  The first gives sufficient conditions for a Catalan function to be zero and
the second gives sufficient conditions for two Catalan functions to be equal.
We further show that these lemmas often ``commute'' with certain generalizations of the operator $e_d^\perp$ called subset lowering operators.

The symmetric group $\SS_\ell$ acts on the ring $\QQ[z_1^{\pm 1}, z_2^{\pm 1}, \dots, z_\ell^{\pm 1}]$ by permuting variables.
The simple reflections $\tau_1, \dots, \tau_{\ell-1} \in \SS_\ell$ act on the basis of Laurent monomials  $\{\mathbf{z}^\gamma\}_{\gamma \in \ZZ^\ell}$
by $\tau_i \mathbf{z}^\gamma = \mathbf{z}^{\tau_i \gamma}$,
where $\tau_i \gamma = (\gamma_1, \dots,\gamma_{i-1}, \gamma_{i+1}, \gamma_{i}, \gamma_{i+2},\dots)$.
This action extends in the natural way to an action on  $\QQ[z_1^{\pm 1}, z_2^{\pm 1}, \dots, z_\ell^{\pm 1}][[t]]$.
We also consider the action of $\SS_\ell$ on subsets  $\Psi \subset [\ell] \times [\ell]$ given by
$\tau_i\Psi = \{(\tau_i(a), \tau_i(b)) \mid (a,b) \in \Psi\}$.

\begin{lemma}
\label{l little sl2 lemma}
Let $\Psi \subset \Delta^+_\ell$ be a root ideal such that $\tau_i \Psi = \Psi$.
Then for any $\gamma \in \ZZ^\ell$,
\[H(\Psi;\gamma) + H(\Psi; \epsilon_{i+1} - \epsilon_i + \tau_i\gamma) = 0.\]
\end{lemma}
\begin{proof}
Set  $f^\Psi = \prod_{(i,j) \in \Psi} \big(1-t \, z_i/ z_j \big)^{-1}$.
Recall from  \eqref{e d HH gamma Psi formal} that
the Catalan functions may be defined in terms of the linear map  $\tilde{\pi}$ by $H(\Psi;\gamma) = \tilde{\pi} (f^\Psi\, \mathbf{z}^\gamma)$.
Thus
\begin{align*}
H(\Psi;\gamma) + H(\Psi; \epsilon_{i+1} - \epsilon_i + \tau_i\gamma)
= \tilde{\pi} \big(f^\Psi\, \mathbf{z}^\gamma + f^\Psi\, \mathbf{z}^{\epsilon_{i+1} - \epsilon_i + \tau_i\gamma} \big).
\end{align*}
We have $\tau_i\Psi = \Psi$ implies  $\tau_i f^\Psi = f^\Psi$ (note that $\Psi \subset \Delta^+$ and $\tau_i \Psi =\Psi$ imply $(i, i+1) \notin \Psi$).  Hence we obtain
\begin{align*}
\tilde{\pi} \big(f^\Psi\, \mathbf{z}^\gamma + f^\Psi\, \mathbf{z}^{\epsilon_{i+1} - \epsilon_i + \tau_i\gamma} \big)
= \tilde{\pi} \circ (1+z_{i+1}/z_i \tau_i)\big(f^\Psi\, \mathbf{z}^\gamma \big),
\end{align*}
where $1+ z_{i+1}/z_i \tau_i$ is regarded as an operator on $\QQ[z_1^{\pm 1}, z_2^{\pm 1}, \dots, z_\ell^{\pm 1}][[t]]$.
We now claim that the operator
$\tilde{\pi} \circ (1+z_{i+1}/z_i \tau_i): \QQ[z_1^{\pm 1}, z_2^{\pm 1}, \dots, z_\ell^{\pm 1}][[t]] \to \QQ[h_1,h_2,\dots][[t]]$
is identically 0, which will complete the proof.
It suffices to show that  $\pi \circ (1 + z_{i+1}/z_i \tau_i)(\mathbf{z}^\delta) =0$ for any $\delta \in \ZZ^\ell$,
where  $\pi$ is the map used to define $\tilde{\pi}$ (see \eqref{e d HH gamma Psi formal}).  We have
\[\pi \circ (1 + z_{i+1}/z_i \tau_i)(\mathbf{z}^\delta)
= \pi \big(\mathbf{z}^\delta + \mathbf{z}^{\epsilon_{i+1} - \epsilon_i + \tau_i\delta}\big)
= s_{\delta} + s_{\epsilon_{i+1} - \epsilon_i + \tau_i\delta}
= 0,\]
where the last equality is by the Schur function straightening rule (Proposition \ref{p schur straighten}).
\end{proof}

\begin{lemma}
\label{l toggle lemma zero}
Let $(\Psi,\mu)$ be an indexed root ideal of length $\ell$ and $z \in [\ell-1]$, and  suppose
\begin{align}
\label{el toggle lemma zero 1}
&\text{$\Psi$ has a ceiling in columns $z, z+1$;} \qquad \\[-.7mm]
\label{el toggle lemma zero 2}
&\text{$\Psi$ has a wall in rows $z,z+1$;}  \qquad \\[-.7mm]
\label{el toggle lemma zero 3}
&\text{$\mu_{z} = \mu_{z+1}-1$.} \qquad
\end{align}
Then $\HH(\Psi;\mu)=0$.
\end{lemma}
\begin{proof}
Conditions \eqref{el toggle lemma zero 1}--\eqref{el toggle lemma zero 2} are just another way of saying $\tau_z\Psi = \Psi$.
By \eqref{el toggle lemma zero 3}, \break $\epsilon_{z+1} - \epsilon_z + \tau_z \mu = \mu$.
Hence the result follows from Lemma~\ref{l little sl2 lemma}.
\end{proof}

\begin{example}
By Lemma~\ref{l toggle lemma zero} with  $z=2$, the following Catalan function is zero:
\ytableausetup{mathmode, boxsize=1.1em,centertableaux}
\[{\scriptsize
\begin{ytableau}
   3    &  *(red) & *(red) &*(red)  & *(red)   \\
        &    1   &         &        & *(red)   \\
        &        &     2   &        & *(red) \\
        &        &         &    1   &          \\
        &        &         &        &    1
\end{ytableau}}
=0.
\]
\smallskip
\end{example}

\begin{lemma}[Mirror Lemma I]
\label{l little lemma many zero 2}
Let $(\Psi, \mu)$ be an indexed root ideal of length  $\ell$, and let $y,z,w$ be indices in the same bounce path of $\Psi$ with $1 \le y \le z \le w < \ell$, satisfying
\begin{align}
\label{el little lemma many zero 2 1}
&\text{$\Psi$ has a ceiling in columns $y, y+1$;}  \\[-.7mm]
\label{el little lemma many zero 2 2}
&\text{$\Psi$ has a mirror in rows $x, x+1$ for all $x \in \bpath_\Psi(y, \upp_\Psi(w)) $;}  \\[-.7mm]
\label{el little lemma many zero 2 3}
&\text{$\Psi$ has a wall in rows $w,w+1$;}  \\[-.7mm]
\label{el little lemma many zero 2 4}
&\text{$\mu_{x} = \mu_{x+1}$ for all $x \in \bpath_\Psi(y,w) \setminus \{z\}$;}  \\[-.7mm]
\label{el little lemma many zero 2 5}
&\text{$\mu_{z} = \mu_{z+1} - 1$}.
\end{align}
Then $\HH(\Psi;\mu) = 0$.
\end{lemma}
\begin{proof}
The proof is by induction on $w-y$.
The base case $y=w$ is Lemma~\ref{l toggle lemma zero}.
Now assume  $y < w$.
By \eqref{el little lemma many zero 2 2}, the root $\beta = (\upp_\Psi(w+1),w)$ is addable to $\Psi$.
So we can expand $\HH(\Psi;\mu)$ using \eqref{e Rlambda recurrence nonroot HH} to obtain
\[\HH(\Psi;\mu) = \HH(\Psi \cup \beta;\mu) - t\;\! \HH(\Psi\cup \beta; \mu+ \eroot{\beta}).
\]
The root ideal $\Psi \cup \beta$ has a wall in rows $\upp_\Psi(w), \upp_{\Psi}(w) + 1$ and a ceiling in columns $w,w+1$.
Hence, if  $z \ne w$, we have $\HH(\Psi \cup \beta;\mu) = 0$ by the inductive hypothesis and
$\HH(\Psi\cup \beta; \mu+ \eroot{\beta}) = 0$ by Lemma~\ref{l toggle lemma zero} (\eqref{el toggle lemma zero 3} holds by $(\mu+ \eroot{\beta})_{w} = (\mu + \eroot{\beta})_{w+1} - 1$).
If  $z = w$, then we have $\HH(\Psi \cup \beta;\mu) = 0$ by Lemma~\ref{l toggle lemma zero} and
$\HH(\Psi\cup \beta; \mu+ \eroot{\beta}) = 0$ by the inductive hypothesis
(\eqref{el little lemma many zero 2 5} holds with $\mu+ \eroot{\beta}$ in place of  $\mu$ and $\upp_\Psi(w)$ in place of  $z$).
\end{proof}

Here is another useful variant:
\begin{lemma}[Mirror Lemma II]
\label{l cascading toggle lemma}
Let $(\Psi, \mu)$ be an indexed root ideal of length $\ell$,
and let $y,w$ be indices in the same bounce path of $\Psi$ with $1 \le y \le w < \ell$,
satisfying \eqref{el little lemma many zero 2 1}--\eqref{el little lemma many zero 2 3}~and
\begin{align}
\label{el cascading toggle lemma 4}
&\text{$\mu_{x} = \mu_{x+1}$ for all $x \in \bpath_\Psi(y,w)$.}
\end{align}
If $\Psi$ has a removable root  $\alpha$ in column $y$, then  $\HH(\Psi;\mu) = \HH(\Psi \setminus \alpha ; \mu)$.
Similarly, \break if $\Psi$ has a removable root  $\beta$ in row $w+1$, then  $\HH(\Psi;\mu) = \HH(\Psi \setminus \beta ; \mu)$.
\end{lemma}
\begin{proof}
Apply \eqref{e Rlambda recurrence HH} with the removable root $\alpha$ to obtain
\[ \HH(\Psi;\mu) = \HH(\Psi \setminus \alpha;\mu) + t\;\! \HH(\Psi;\mu+ \eroot{\alpha}) =
\HH(\Psi \setminus \alpha;\mu),\]
where the second equality is by Lemma~\ref{l little lemma many zero 2} applied with indexed root ideal $(\Psi,\mu + \eroot{\alpha})$ and  $z = y$
(\eqref{el little lemma many zero 2 5} holds since $(\mu+ \eroot{\alpha})_{y} = (\mu+ \eroot{\alpha})_{y+1}-1$).
A similar argument with $\beta$ in place of $\alpha$ gives
$ \HH(\Psi;\mu) =  \HH(\Psi \setminus \beta;\mu)$.
\end{proof}

\begin{example}
By Lemma~\ref{l cascading toggle lemma} with $y=2$, $w = 4$, we have
\ytableausetup{mathmode, boxsize=1.1em,centertableaux}
\[\scriptsize
\begin{ytableau}
   3    &  \bm{\alpha} & *(red!80) &*(red!80)  & *(red!80)   &*(red!80)\\
        &    2   &         &  *(red!80)& *(red!80)   &*(red!80)\\
        &        &     2   &        & *(red!80)   &*(red!80)\\
        &        &         &    1   &          &*(red!80)\\
        &        &         &        &    1     &*(red!80) \bm{\beta}\\
        &        &         &        &          &  1
\end{ytableau}
\hspace{2mm} = \hspace{2mm}
\begin{ytableau}
   3    &  *(red!80) \bm{\alpha} & *(red!80) &*(red!80)  & *(red!80)   &*(red!80)\\
        &    2   &         &  *(red!80)& *(red!80)   &*(red!80)\\
        &        &     2   &        & *(red!80)   &*(red!80)\\
        &        &         &    1   &          &*(red!80)\\
        &        &         &        &    1     &*(red!80) \bm{\beta}\\
        &        &         &        &          &  1
\end{ytableau}
\hspace{2mm} = \hspace{2mm}
\begin{ytableau}
   3    &  *(red!80) \bm{\alpha} & *(red!80) &*(red!80)  & *(red!80)   &*(red!80)\\
        &    2   &         &  *(red!80)& *(red!80)   &*(red!80)\\
        &        &     2   &        & *(red!80)   &*(red!80)\\
        &        &         &    1   &          &*(red!80)\\
        &        &         &        &    1     &\bm{\beta}\\
        &        &         &        &          &  1
\end{ytableau} \, .
\]
\end{example}

The vertical dual Pieri rule gives a combinatorial description of  $e_d^\perp \fs^{(k)}_\mu$.
The proof of this rule requires extending it
to a combinatorial description of a more general operator on $\fs^{(k)}_\mu$, which we now define.

\begin{definition}
\label{d subset lowering}
For $d \in \ZZ_{\ge 0}$ and $V \subset [\ell]$, the \emph{subset lowering operator} $\beperp_{d,V}$ is given by
\begin{align*}
\beperp_{d,V} \HH(\Psi;\mu) = \sum_{S \subset V, \, |S| = d} \HH(\Psi;\mu-\epsilon_S),
\end{align*}
where $(\Psi, \mu)$ is any indexed root ideal  of length $\ell$.
\end{definition}
With this notation, Lemma~\ref{l ed perp HH} says that $\beperp_{d, [\ell]} H(\Psi;\mu) = e_d^\perp H(\Psi;\mu)$ for any  $d \ge 0$.

\begin{remark}
Just as for raising operators, the subset lowering operators 
should be thought of as acting on the input $\mu$ in  $\HH(\Psi; \mu)$ rather than on the polynomials themselves.
They are not in general well-defined operators on symmetric functions:  for instance,
$\HH(\varnothing;12) = 0$ but $\beperp_{1,\{1\}} \HH(\varnothing;12) := \HH(\varnothing; 02) = -s_{11} \neq 0$.
Also see Example~\ref{ex toggle lemma zero ed}.
\end{remark}

Despite the fact that $\beperp_{d,V}$ is not a well-defined operator on symmetric functions,
it commutes with raising operators and the recurrences of Proposition \ref{p inductive computation atom}.
Moreover, it
commutes with the Mirror Lemmas under some mild assumptions.
This means that if we have any computation involving Catalan functions that only uses the recurrences and
Mirror Lemmas in
a controlled way, then we can commute  $\beperp_{d,V}$ through this entire computation.
This is a powerful technique and is crucial to the proof of the vertical dual Pieri rule.

\begin{proposition}
\label{p Rlambda recurrence ed}
Let $(\Psi,\mu)$ be an indexed root ideal.
For any root $\beta$ addable to $\Psi$,
\begin{align}\label{e Rlambda recurrence nonroot HH ed}
\beperp_{d,V} \HH(\Psi;\mu) = \beperp_{d,V}\HH(\Psi \cup \beta;\mu) - t \, \beperp_{d,V}\HH(\Psi \cup \beta;\mu+ \eroot{\beta}).
\end{align}
For any removable root $\alpha$ of $\Psi$,
\begin{align}
\label{e Rlambda recurrence HH ed}
\beperp_{d,V} \HH(\Psi;\mu) = \beperp_{d,V}\HH(\Psi \setminus \alpha;\mu) + t \, \beperp_{d,V}\HH(\Psi;\mu+ \eroot{\alpha}).
\end{align}
\end{proposition}
\begin{proof}
This is immediate from the definition of  $\beperp_{d,V}$ and Proposition \ref{p inductive computation atom}.
\end{proof}


\begin{lemma}
\label{l toggle lemma zero ed}
Let $(\Psi, \mu)$ be an indexed root ideal of length $\ell$, let $z \in [\ell-1]$, and $V \subset [\ell]$.  Suppose this data satisfies
\eqref{el toggle lemma zero 1}--\eqref{el toggle lemma zero 3} together with
\begin{align}
\label{el toggle lemma zero ed 4}
&\text{$V$ contains both or neither of $z,z+1$.}
\end{align}
Then $\beperp_{d,V} \HH(\Psi;\mu)=0$ for any  $d \ge 0$.
\end{lemma}
\begin{proof}
By definition,
\[\beperp_{d,V} \HH(\Psi;\mu) = \sum_{S \subset V,\, |S| = d} \HH(\Psi;\mu-\epsilon_S).\]
The terms in the sum such that $S$ contains both or neither of $z,z+1$ are zero by Lemma~\ref{l toggle lemma zero}.
If $V \cap \{z,z+1\} = \varnothing$, then this accounts for all the terms, and we are done.
If $\{z,z+1\} \subset V$, then the remaining terms come in pairs:
\begin{align}
\label{e S subset V}
\sum_{S \subset V, \, |S| = d} \! \HH(\Psi;\mu-\epsilon_S) =
\!
\sum_{\substack{S' \subset V \setminus \{z,z+1\} \\ |S'| = d-1}} \!\Big( \HH(\Psi;\mu - \epsilon_{S'} - \epsilon_{z}) + \HH(\Psi;\mu - \epsilon_{S'} - \epsilon_{z+1}) \Big).
\end{align}
By Lemma~\ref{l little sl2 lemma} with $\gamma = \mu-\epsilon_{S'} - \epsilon_{z+1}$ (using that $t_z \gamma = \gamma$ and  $\mu - \epsilon_{S'} - \epsilon_{z} = \epsilon_{z+1} - \epsilon_{z} +  t_z \gamma $),
$\HH(\Psi;\mu - \epsilon_{S'} - \epsilon_{z}) + \HH(\Psi;\mu - \epsilon_{S'} - \epsilon_{z+1}) = 0$.
Hence the right side of \eqref{e S subset V} is zero.
\end{proof}

\begin{example}
\label{ex toggle lemma zero ed}
By Lemma~\ref{l toggle lemma zero ed} with  $z = 2$,
\ytableausetup{mathmode, boxsize=1.1em,centertableaux}
\[\scriptsize
\beperp_{1,\{2,3\}}
\begin{ytableau}
   3    &  *(red)& *(red) \\
        &    1   &        \\
        &        &     2
\end{ytableau}
:=
\begin{ytableau}
   3    &  *(red)& *(red) \\
        &    0   &        \\
        &        &     2
\end{ytableau}
+
\begin{ytableau}
   3    &  *(red)& *(red) \\
        &    1   &        \\
        &        &    1
\end{ytableau}
=0.
\]
For comparison here is an example in which \eqref{el toggle lemma zero ed 4} is not satisfied:
\[{\scriptsize
\beperp_{1,\{1,2\}}
\begin{ytableau}
   3    &  *(red)& *(red) \\
        &    1   &        \\
        &        &     2
\end{ytableau}
:=
\begin{ytableau}
   2    &  *(red)& *(red) \\
        &    1   &        \\
        &        &     2
\end{ytableau}
+
\begin{ytableau}
   3    &  *(red)& *(red) \\
        &    0   &        \\
        &        &    2
\end{ytableau}
= -
\begin{ytableau}
   3    &  *(red)& *(red) \\
        &    1   &        \\
        &        &    1
\end{ytableau}
= -s_{311}-t s_{410} \ne 0.}
\]
\end{example}

With Lemma \ref{l toggle lemma zero ed} in hand, we easily obtain generalizations of Lemmas \ref{l little lemma many zero 2} and \ref{l cascading toggle lemma} to
the  $\beperp_{d,V} \HH(\Psi;\mu)$.
\begin{lemma}
\label{l little lemma many zero 2 ed}
Let $(\Psi, \mu)$ be an indexed root ideal of length  $\ell$, let $y,z,w$ be indices in the same bounce path of $\Psi$ with $1 \le y \le z \le w < \ell$, and let $V \subset [\ell]$.
Suppose that this data satisfies \eqref{el little lemma many zero 2 1}--\eqref{el little lemma many zero 2 5} together with
\begin{align}
\label{el little lemma many zero 2 ed 6}
\text{$V$ contains both or neither of $x, x+1$ for all $x \in \bpath_\Psi(y,w)$.}
\end{align}
Then $\beperp_{d,V} \HH(\Psi;\mu) = 0$  for any  $d \ge 0$.
\end{lemma}
\begin{proof}
Repeat the proof of Lemma \ref{l little lemma many zero 2} using Lemma \ref{l toggle lemma zero ed} in place of Lemma~\ref{l toggle lemma zero} and
\eqref{e Rlambda recurrence nonroot HH ed} in place of \eqref{e Rlambda recurrence nonroot HH}.
\end{proof}

\begin{lemma}
\label{l cascading toggle lemma ed}
Let $(\Psi, \mu)$ be an indexed root ideal of length $\ell$,
let $y,w$ be indices in the same bounce path of $\Psi$ with $1 \le y \le w < \ell$,
let  $d \ge 0$,  and let $V \subset [\ell]$
satisfy \eqref{el little lemma many zero 2 1}--\eqref{el little lemma many zero 2 3}, \eqref{el cascading toggle lemma 4}, and \eqref{el little lemma many zero 2 ed 6}.
If $\Psi$ has a removable root  $\alpha$ in column $y$, then  $\beperp_{d,V}\HH(\Psi;\mu) = \beperp_{d,V}\HH(\Psi \setminus \alpha ; \mu)$.
Similarly,  if $\Psi$ has a removable root  $\beta$ in row $w+1$, then  $\beperp_{d,V}\HH(\Psi;\mu) = \beperp_{d,V}\HH(\Psi \setminus \beta ; \mu)$.
\end{lemma}
\begin{proof}
Repeat the proof of Lemma \ref{l cascading toggle lemma} using
\eqref{e Rlambda recurrence HH ed} in place of \eqref{e Rlambda recurrence HH} and
Lemma \ref{l little lemma many zero 2 ed} in place of Lemma~\ref{l little lemma many zero 2}.
\end{proof}

\begin{example}
By Lemma~\ref{l cascading toggle lemma ed} with $y=2$, $w = 4$, for any  $d \ge 0$ and $V \subset \{1,2, 3, 4,5, 6\}$ such that  $V \cap \{2,3\}$ has size 0 or 2 and $V \cap \{4,5\}$ has size 0 or 2, there holds
\ytableausetup{mathmode, boxsize=1.1em,centertableaux}
\[\scriptsize
\beperp_{d,V}
\begin{ytableau}
   3    &  \bm{\alpha} & *(red!80) &*(red!80)  & *(red!80)   &*(red!80)\\
        &    2   &         &  *(red!80)& *(red!80)   &*(red!80)\\
        &        &     2   &        & *(red!80)   &*(red!80)\\
        &        &         &    1   &          &*(red!80)\\
        &        &         &        &    1     &*(red!80) \bm{\beta}\\
        &        &         &        &          &  1
\end{ytableau}
\hspace{2mm} = \hspace{2mm}
\beperp_{d,V}
\begin{ytableau}
   3    &  *(red!80) \bm{\alpha} & *(red!80) &*(red!80)  & *(red!80)   &*(red!80)\\
        &    2   &         &  *(red!80)& *(red!80)   &*(red!80)\\
        &        &     2   &        & *(red!80)   &*(red!80)\\
        &        &         &    1   &          &*(red!80)\\
        &        &         &        &    1     &*(red!80) \bm{\beta}\\
        &        &         &        &          &  1
\end{ytableau}
\hspace{2mm} = \hspace{2mm}
\beperp_{d,V}
\begin{ytableau}
   3    &  *(red!80) \bm{\alpha} & *(red!80) &*(red!80)  & *(red!80)   &*(red!80)\\
        &    2   &         &  *(red!80)& *(red!80)   &*(red!80)\\
        &        &     2   &        & *(red!80)   &*(red!80)\\
        &        &         &    1   &          &*(red!80)\\
        &        &         &        &    1     &\bm{\beta}\\
        &        &         &        &          &  1
\end{ytableau} \, .
\]
\end{example}

\section{$k$-Schur straightening}
\label{s kSchur straightening}

One beautiful consequence of identifying $k$-Schur functions as a subclass of
Catalan functions is that we obtain a natural generalization of $k$-Schur functions
to a class of Catalan functions
indexed by the set of weights
\begin{align}
\label{e pseudopartition}
\tpar^k_\ell := \big\{\mu \in \ZZ_{\le k}^\ell \mid \mu_1 + \ell-1 \ge \mu_2 + \ell-2 \ge \dots \ge \mu_\ell \big\},
\end{align}
which contains  $\Par^k_\ell$ but many nonpartitions as well.
Here we show
that these Catalan functions obey a  $k$-Schur straightening rule similar to that of
ordinary Schur functions,
and this plays a crucial role in the proof of the dual Pieri rules.



\subsection{The $k$-Schur root ideal}
\label{s kschur root ideal}
The definition of  $\fs^{(k)}_\mu$ from \eqref{d catalan kschur} carries over unchanged to this more general setting, which we record for convenience:

\begin{definition}
\label{d root ideal kschur}
For  $\mu \in \tpar^k_\ell$,  define the root ideal
\begin{align}
\label{ed root ideal kschur}
\Delta^k(\mu) = \{(i,j) \in \Delta^+_\ell \mid  k-\mu_i + i < j \},
\end{align}
and the associated Catalan function
\begin{align}
\fs^{(k)}_\mu = H(\Delta^k(\mu);\mu) =
\prod_{i=1}^\ell\,\prod_{j=k+1-\mu_i+i}^\ell \big(1-tR_{i j}\big)^{-1} s_\mu
\, .
\end{align}
\end{definition}

\begin{example}
Here are some examples of the Catalan functions $\fs^{(k)}_\mu$ for $k=4$:
\[\fs^{(4)}_{3321} =
\ytableausetup{mathmode, boxsize=1.44em,centertableaux}
{\Tiny
\begin{ytableau}
3 & *(white) & *(red) &*(red)  \\
\mynone & 3 & *(white) & *(red)  \\
\mynone & \mynone  & 2 & *(white) \\
\mynone  & \mynone  & \mynone  & 1
\end{ytableau}}, \ \
\fs^{(4)}_{4432320} =
{\Tiny \begin{ytableau}
4         & *(red)   & *(red)   &*(red)   &*(red)   &*(red) &*(red)   \\
\mynone  & 4        & *(red)    &  *(red) &*(red)   &*(red) &*(red)    \\
\mynone  & \mynone  & 3        & ~       &*(red)    &*(red) &*(red)   \\
\mynone  & \mynone  & \mynone  & 2       & ~        &~       &*(red)   \\
\mynone  & \mynone  & \mynone  & \mynone & 3        &  ~    & *(red)\\
\mynone  & \mynone  & \mynone  & \mynone & \mynone &  2     & ~\\
\mynone  & \mynone  & \mynone  & \mynone & \mynone & \mynone & 0
\end{ytableau}}, \ \
\fs^{(4)}_{-1012211} = {\Tiny \begin{ytableau}
-1        & ~      & ~         &   ~      &~         & ~        & *(red) \\
\mynone & 0        & ~         & ~        &~         & ~         & *(red)\\
\mynone & \mynone  & 1         & ~        &  ~       & ~         & *(red) \\
\mynone  & \mynone  & \mynone  & 2        & ~        & ~          & *(red)  \\
\mynone  & \mynone  & \mynone  & \mynone  & 2        & ~          & ~\\
\mynone  & \mynone  & \mynone  & \mynone  & \mynone  & 1          & ~\\
\mynone  & \mynone  & \mynone  & \mynone  & \mynone  &   \mynone & 1\\
\end{ytableau}}.
\]
The first was computed explicitly in Examples \ref{ex k Schur to Schur} and \ref{ex 3321}.
\end{example}

The following is essentially a restatement of the definition of $\Delta^k$, which we will reference  frequently.
\begin{proposition}
\label{p Delta k basics}
For $\mu \in \tpar^k_\ell$, the western border of the root ideal $\Delta^k(\mu)$ consists of the roots $(i, k+1-\mu_i+i)$ for $i$ such that $k+1-\mu_i+i \le \ell$.
Moreover, $(i,k+1-\mu_i+i)$ is a removable root of $\Delta^k(\mu)$ if and only if $\mu_{i} \ge \mu_{i+1}$ and $k+ 1-\mu_i + i  \le \ell$.
\end{proposition}

\subsection{$k$-Schur straightening}

\begin{lemma}[$k$-Schur straightening I]
\label{l k schur straightening 1spin}
Let $\mu \in  \tpar^k_\ell$,  $\Psi  = \Delta^k(\mu)$, and
$z \in [\ell-1]$.
Suppose
\begin{align}
\label{el straighten 1spin 1}
&\text{$y+1 = \upp_\Psi(z+1)$ is defined;}\\
\label{el straighten 1spin 2}
&\text{$\mu_z = \mu_{z+1} - 1$;}\\
\label{el straighten 1spin 3}
&\text{$\mu_{z+1} \ge \mu_{z+2}$ and $\mu_y \ge \mu_{y+1}$.}
\end{align}
Then
\begin{align}
\label{el k schur straightening 1spin}
\fs^{(k)}_\mu =
t \, \fs^{(k)}_{\mu + \epsilon_{y+1} - \epsilon_{z+1}}.
\end{align}
\end{lemma}

\begin{proof}
Expand $\HH(\Psi ; \mu)$ using \eqref{e Rlambda recurrence HH} with the removable root  $\delta = (y+1,z+1)$ to obtain
\begin{align}
\label{e l chainup one term remains}
\HH(\Psi;\mu) = \HH(\Psi \setminus \delta; \mu) + t\, \HH(\Psi; \mu+ \eroot{\delta}).
\end{align}
Using that  $\mu_z = \mu_{z+1}-1$ and $\mu_y \ge \mu_{y+1}$ with the definition of  $\Delta^k$ shows that  $\Psi \setminus \delta$ has a wall in rows  $z,z+1$ and
a ceiling in columns  $z, z+1$.
Hence  $\HH(\Psi \setminus \delta; \mu) = 0$ by Lemma~\ref{l toggle lemma zero}.

The root $\alpha = (y+1,\down_\Psi(y+1)-1)= (y+1, z)$ is addable to $\Psi$ by the assumption $\mu_y \ge \mu_{y+1}$;
also set $B = \{\beta\}$ with $\beta = (z+1,\down_\Psi(z+1))$ if this is defined and $B = \varnothing$ otherwise.
Lemma~\ref{l cascading toggle lemma}, applied to the indexed root ideal $(\Psi \cup \alpha, \mu+ \eroot{\delta})$,
yields
\[\HH(\Psi; \mu+ \eroot{\delta})= \HH(\Psi \cup \alpha ; \mu+ \eroot{\delta})  = \HH(\Psi \cup \alpha \setminus B; \mu+ \eroot{\delta}) = \fs^{(k)}_{\mu+ \eroot{\delta}}. \]
For the last equality we are using $\mu_{z+1} \ge \mu_{z+2}$ to conclude that $\down_\Psi(z+1)$ is defined if and only if the $(z+1)$-st row of $\Psi$ is nonempty.
\end{proof}

\begin{remark}
\label{r corner cases}
For conditions such as \eqref{el straighten 1spin 3} arising in this section and the next, corner cases  $z=\ell-1$ and  $y =0$ are conveniently handled by defining  $\mu_{0} = k$ and  $\mu_{\ell+1} = 0$ for  $\mu \in \Par^k_\ell$.
The latter is a standard convention, however, we must take care here since the definitions of  $\Delta^k(\mu)$ and $\fs^{(k)}_{\mu} = H(\Delta^k(\mu); \mu)$
depend on  $\ell$;
we still regard  $\Delta^k(\mu)$ as a subset of  $[\ell] \times [\ell]$ \emph{not} $[\ell+1] \times [\ell+1]$.
\end{remark}

\begin{example}
By Lemma \ref{l k schur straightening 1spin} with $z = 2, y = 0$, we have
\ytableausetup{mathmode, boxsize=1.1em,centertableaux}
\[
\fs^{(4)}_{32321}
\hspace{1.4mm}=
\hspace{1.4mm}
{\tiny
\begin{ytableau}
3   &    ~    &  *(red)& *(red)   &*(red)\\
    &     2   &    ~   &  ~       &*(red)\\
     &         &    3   &  ~       & *(red)\\
     &         &        &    2     & ~\\
   &         &        &          &  1
\end{ytableau}}
\hspace{1.4mm} = \hspace{1.4mm}
t\,
{\tiny
\begin{ytableau}
 4   &    *(red)    &  *(red)& *(red)   &*(red)\\
     &     2   &    ~   &  ~       &*(red)\\
     &         &    2   &  ~       & ~\\
   &         &        &    2     & ~\\
     &         &        &          &  1
\end{ytableau}}
\hspace{1.4mm}
= \hspace{1.4mm}
t \, \fs^{(4)}_{42221}.
\]
\end{example}

\begin{remark}
It is not difficult to show using Lemma \ref{l toggle lemma zero} that
for $\mu \in \tpar^k_\ell$,
$\fs^{(k)}_\mu = 0$ if \eqref{el straighten 1spin 2} holds and $\upp_{\Delta^k(\mu)}(z+1)$ is undefined.
Using this in combination with Lemma~\ref{l k schur straightening 1spin},
one can show that for any $\mu \in \tpar^k_\ell$, $\fs^{(k)}_\mu$ is equal to 0 or a power of $t$ times $\fs^{(k)}_\nu$ for $\nu \in \Par^{k}_\ell$.
However, we will not need this general statement.
Instead, we focus on the special case $\mu = \lambda - \epsilon_z$ for $\lambda \in \Par^k_\ell$,
where we can give an explicit combinatorial description of when $\fs^{(k)}_\mu$ is zero and the power of $t$ when it is nonzero.
\end{remark}

\begin{lemma}
\label{l k schur straightening 0}
Let $\mu \in \tpar^k_\ell$  and $\Psi  = \Delta^k(\mu)$.
Let $h \in \ZZ_{\ge 1}$ and $z \in [\ell-h]$.
Suppose that
\begin{align}
\label{el straighten0 1}
&\text{$y+1 = \upp_\Psi(z+1)$ is defined};\\
\label{el straighten0 2}
&\text{$\mu_z = \mu_{z+1} - 1$}; \\
\label{el straighten0 3}
&\text{$\mu_{z+1} = \dots = \mu_{z+h}$ and $\mu_{y+1} = \dots = \mu_{y+h}$}; \\
&\text{$\mu_{z+h} \ge \mu_{z+h+1}$ and $\mu_y \ge \mu_{y+1}$}.
\label{el straighten0 4}
\end{align}
Then
\begin{align}
\label{el k schur straightening 0}
\fs^{(k)}_{\mu} = t^{h}\fs^{(k)}_{\mu + \epsilon_{[y+1,y+h]} - \epsilon_{[z+1,z+h]}}.
\end{align}
\end{lemma}
\begin{proof}
Set $\mu^i = \mu + \epsilon_{[y+1,y+i]} - \epsilon_{[z+1,z+i]}$ for  $i \in [0,h]$.
By $h$ applications of Lemma \ref{l k schur straightening 1spin}, we obtain
\[
\fs^{(k)}_{\mu} = t\, \fs^{(k)}_{\mu^1} = \dots = t^{h}\fs^{(k)}_{\mu^{h}}.
\]
We verify that the hypotheses of Lemma \ref{l k schur straightening 1spin} are satisfied at each step:
it follows from Proposition \ref{p Delta k basics} that $(y+i+1, k-\mu_{y+i+1} + y+i+2) = (y+i+1, z+i+1)$
(using \eqref{el straighten0 3} and $k+1-\mu_{y+1}+y+1 = z+1$ for the equality)
is a removable root of  $\Delta^k(\mu^i)$ for each $i \in [0,h-1]$; this ensures that
\eqref{el straighten 1spin 1} holds for each application of the lemma.
The assumptions \eqref{el straighten0 3}--\eqref{el straighten0 4} ensure that \eqref{el straighten 1spin 2}--\eqref{el straighten 1spin 3} hold for each application of the lemma.
\end{proof}

The following definition is forced on us by the combinatorics arising in  $k$-Schur straightening.  Examples of this definition and Lemma \ref{l k schur straightening 0}
are given at the end of the subsection.

\begin{definition}
\label{d cvr}
Let $\lambda \in \Par^k_\ell$ and $z \in [\ell]$.
Set $\mu = \lambda - \epsilon_z$
and $\Psi= \Delta^k(\mu)$.
Let  $c = |\uppath_\Psi(z)|$.
If $z = \ell$ or $\lambda_z > \lambda_{z+1}$ or $\upp^{c}_\Psi(z+1)$ is undefined, then set  $h=0$;
otherwise, set $y + 1 = \upp^{c}_\Psi(z+1)$ and
let $h \in [\ell-z]$ be as large as possible such that
\begin{align}
\label{d cvr 0}
\parbox{12.4cm}{$\mu$ is constant on each of the intervals $[z+1, z+h]$, $[\upp_\Psi(z), \upp_\Psi(z)+h]$, \\[1mm]
$[\upp^2_\Psi(z), \upp^2_\Psi(z)+h]$, $\dots$, $[\chainup_\Psi(z), \chainup_\Psi(z)+h]$, and  $[y+1, y+h]$.}
\end{align}
Define
\begin{align*}
\cvr_z(\lambda) =
\lambda + \epsilon_{[y+1,y+h]} - \epsilon_{[z,z+h]}.
\end{align*}
If $y$ is undefined or, equivalently, $h = 0$ then $\cvr_z(\lambda) = \mu$ (we interpret $[y+1,y+h] = \varnothing$ in this case).
We also define the \emph{bounce} from  $\lambda$ to  $\cvr_z(\lambda)$ by
\begin{align}
\bounce(\cvr_z(\lambda), \lambda) = h \cdot c.
\end{align}
\end{definition}

For the proof of  $k$-Schur straightening II below and the results of the next subsection,
we also define integers $c'$ and  $h_x$ as follows:
if  $\lambda_{z+h}> \lambda_{z+h+1}$,
then $c' := -1$;  otherwise,
\begin{align}
\label{e cprime def}
\text{$c' := \max\big\{ i \in [0,c-1] \, \mid \, \mu_{x+h} = \mu_{x+h+1}$
for all $x \in \bpath_\Psi(\upp_\Psi^{i}(z), \upp_\Psi(z)) \big\}$. }
\end{align}
And for $x \in \uppath_\Psi(z)$, let
\begin{align}
\label{et hx def ed}
h_x &= \begin{cases}
h+1 &\text{if  $x= \upp_\Psi^{s}(z)$ with $s \le c'$},\\
h & \text{otherwise}.
\end{cases}
\end{align}

\begin{lemma}
\label{l more about d cvr}
The following facts clarify Definition \ref{d cvr}.
\begin{list}{\emph{(\roman{ctr})}} {\usecounter{ctr} \setlength{\itemsep}{1pt} \setlength{\topsep}{2pt}}
\item If  $y$ is defined (equivalently,  $h >0$), then $\mu_y >  \mu_{y+1}$.
\item $\cvr_z(\lambda) \in \Par^k_\ell$ if and only if
$\lambda_{z+h} > \lambda_{z+h+1}$.
\item If $\upp^{c}_\Psi(z+1)$ is defined and  $z < \ell$, then \eqref{d cvr 0} holds with  $h=1$ (so there does exist a largest  $h \in  [\ell-z]$ satisfying \eqref{d cvr 0}).
\end{list}
\end{lemma}

Note that we have used Remark \ref{r corner cases} to handle corner cases in (i) and (ii).

\begin{proof}
Statements (i) and (iii) are straightforward consequences of Proposition~\ref{p Delta k basics} and (ii) is immediate from (i).
\end{proof}

\begin{lemma}
\label{l straightening many no overlap}
The intervals in \eqref{d cvr 0} are pairwise disjoint.
\end{lemma}
\begin{proof}
It suffices to show that  $x+h_x < \down_\Psi(x)$ for all  $x \in \uppath_\Psi(z)\setminus\{z\}$
and $y+h <  \chainup_\Psi(z)$.
We begin by proving the former (this is stronger than what we need, but we will use this version later).
Suppose for a contradiction that this fails; then there is a largest $x \in \uppath_\Psi(z)\setminus\{z\}$
such that $x+h_x \ge \down_\Psi(x)$.
By definition of  $h_x$, we have $\mu_x = \dots = \mu_{\down_\Psi(x)}= \dots = \mu_{x+h_x}$.
We thus cannot have $x = \upp_\Psi(z)$ since this would contradict $\mu_{z-1} > \mu_z$.
So $x = \upp_\Psi^a(z)$ for some $a \ge 2$.
We then have
\[
h_{\down_\Psi(x)} \ge h_x \ge \down_\Psi(x)-x = k+1-\mu_x =k+1-\mu_{\down_\Psi(x)} = \down_\Psi^2(x)-\down_\Psi(x),
\]
where we have used Proposition \ref{p Delta k basics} for the first and third equalities.
This contradicts our choice of  $x$.

Now to prove $y+h <  \chainup_\Psi(z)$,
suppose for a contradiction that  $y+h \ge \chainup_\Psi(z)$.
Then
by definition of  $h$, we have $\mu_{y+1} = \dots = \mu_{\chainup_\Psi(z)} = \dots = \mu_{y+h}$.
If  $z= \chainup_\Psi(z)$, this contradicts $\mu_{z-1} > \mu_z$ and we are done, so assume  $\chainup_\Psi(z)< z$.
We have
\begin{multline*}
h \ge \chainup_\Psi(z)-y = \down_\Psi(y+1)-(y+1)  \\
= k+1-\mu_{y+1}
=k+1-\mu_{\chainup_\Psi(z)} = \down_\Psi(\chainup_\Psi(z))-\chainup_\Psi(z),
\end{multline*}
where the the first, second, and fourth equalities follow from Proposition \ref{p Delta k basics}.
This contradicts the result of the previous paragraph $x+h_x < \down_\Psi(x)$ for  $x = \chainup_\Psi(z)$.
\end{proof}

\begin{theorem}[$k$-Schur straightening II]
\label{t k schur straightening many}
Maintain the notation of Definition~\ref{d cvr}.
Then
\begin{align}
\label{et k schur straightening many}
\fs^{(k)}_\mu = t^{hc}\fs^{(k)}_{\cvr_z(\lambda)}  = t^{\bounce(\cvr_z(\lambda), \lambda)} \fs^{(k)}_{\cvr_z(\lambda)}
\end{align}
and this is equal to 0 if $\cvr_z(\lambda) \notin \Par^k_\ell$.
\end{theorem}
\begin{proof}
We first prove \eqref{et k schur straightening many}.  This is trivial if  $h=0$, so assume $h > 0$ for the remainder of this paragraph.
Let  $\bpath_\Psi(y+1,z+1) = (b_c+1,b_{c-1}+1,\dots,b_0+1)$ (thus $b_0 =z$ and  $b_c = y$).
Set $\mu^i = \mu + \epsilon_{[b_i+1,b_i+h]} - \epsilon_{[z+1,z+h]}$ for $i \in [c]$ (thus  $\mu^c = \cvr_z(\lambda)$).
By  $c$ applications of Lemma~\ref{l k schur straightening 0}, we obtain
\[
\fs^{(k)}_{\mu} = t^h\, \fs^{(k)}_{\mu^1} = \dots = t^{hc} \, \fs^{(k)}_{\mu^{c}}.
\]
The hypotheses of the lemma are satisfied at each step:
\eqref{el straighten0 1}--\eqref{el straighten0 2} are clear from the definition of  $h$,
\eqref{el straighten0 3} follows from the definition of  $h$ together with Lemma \ref{l straightening many no overlap},
and \eqref{el straighten0 4} follows from $\mu + \epsilon_z \in \Par^k_\ell$.

Now suppose $\nu := \cvr_z(\lambda) \notin \Par^k_\ell$; this can only happen if $\lambda_{z+h} = \lambda_{z+h+1}$
by Lemma~\ref{l more about d cvr}~(ii).
If $z+h=\ell$, then $\lambda_z = \lambda_{z+h} = 0$, and we have  $\fs^{(k)}_{\nu} = 0$ by Proposition \ref{p H definition CHL} and the fact that $H_{\gamma} = 0$ for $\gamma_\ell < 0$.
So now assume $z+h<\ell$.
Set  ${\tilde{\Psi}} = \Delta^k(\nu)$.
We will apply Lemma~\ref{l little lemma many zero 2} to show $\fs^{(k)}_{\nu} = 0$.

First consider the case  $h >0$ and define $b_i$ as above.
Let  $c'$ be as in \eqref{e cprime def}.
Note that  $\tilde{\Psi}$ and  $\Psi$ are identical in rows $[y+h+1, z-1]$ and have removable roots in these rows.
It then follows from Proposition \ref{p Delta k basics} and the definitions of  $h$ and  $c'$ that
\begin{multline}
\label{e straighten step 1}
\text{$\down_{\tilde{\Psi}}(b_i+h) = \down_{\Psi}(b_i+h) = b_{i-1}+h$ \, and} \\ \text{$\down_{\tilde{\Psi}}(b_i+h+1) = \down_{\Psi}(b_i+h+1) = b_{i-1}+h+1$ \, for each $i \in [c']$}\,.
\end{multline}
(We need that $b_i+h+1 \le z-1$, which holds since $\mu_{b_i} = \dots = \mu_{b_i+h+1}$ and $\mu_{z-1} > \mu_z$.)
Also by the definitions of  $h$ and $c'$, we have
$\mu_{b_{c'+1}+h} > \mu_{b_{c'+1}+h+1}$
(we need to check that $\mu_{b_{c'+1}+h} < \mu_{b_{c'+1}+h+1}$ cannot occur,
but this follows from $\mu_{z-1}> \mu_z$ and $\mu_{b_{c'+1}} = \dots  = \mu_{b_{c'+1}+h}$).
This implies $\nu_{b_{c'+1}+h} > \nu_{b_{c'+1}+h+1}$.
It then follows from Proposition \ref{p Delta k basics} that
\begin{align*}
&\text{$\down_{\tilde{\Psi}}(b_{c'+1}+h) \le \down_{\Psi}(b_{c'+1}+h) = b_{c'}+h$ and $\down_{\tilde{\Psi}}(b_{c'+1}+h+1) > b_{c'}+h+1$.}
\end{align*}
This means that $\tilde{\Psi}$ has a ceiling in columns  $b_{c'}+h, b_{c'}+h+1$.
Together with \eqref{e straighten step 1}, this shows that the hypotheses of Lemma \ref{l little lemma many zero 2} are satisfied for the indexed root ideal  $(\tilde{\Psi}, \nu)$ (with  $y$,  $z$,  $w$ of the lemma equal to $b_{c'}+h$,  $z+h$,  $z+h$, respectively)
and hence $\fs^{(k)}_{\nu} = 0$.

The case $h=0$ is similar but easier.
Here we define the $b_i$ by
$\uppath_\Psi(z) = (b_0,b_1,\dots,b_{c-1})$.
The proof in the previous paragraph
still works (though some arguments simplify since $\nu = \mu$ and  $\tilde{\Psi} = \Psi$) except for the proof
that $\tilde{\Psi}$ has a ceiling in columns  $b_{c'}+h, b_{c'}+h+1$ in the case $c' = c-1$;
but this follows directly from the fact that $\upp_{\Psi}(b_{c'}+h+1)$ is undefined
(by Definition \ref{d cvr}, $\upp_\Psi^c(z+1)$ is undefined since $h=0$, $\lambda_z = \lambda_{z+1}$, and $z < \ell$).
\end{proof}

\begin{example}
\label{ex straightening 1}
We illustrate Definition \ref{d cvr} and different cases of the proof of Theorem~\ref{t k schur straightening many} with several examples, all with  $k=4$.
First, let $\lambda = 222222221$, $\mu = 222221221$,  $z = 6$.
Then  $\uppath_{\Delta^k(\mu)}(z) = (6,3)$, $\uppath_{\Delta^k(\mu)}(z+1) = (7,4,1)$, $y+1 = 1$,  $c=2$, $h=2$, and  $\cvr_z(\lambda) = 332221111$.
We illustrate the two applications of Lemma~\ref{l k schur straightening 0} used in the proof of Theorem~\ref{t k schur straightening many} to obtain $\fs^{(4)}_{222221221} = t^4\fs^{(4)}_{332221111}$\,:
\ytableausetup{mathmode, boxsize=1.02em,centertableaux}
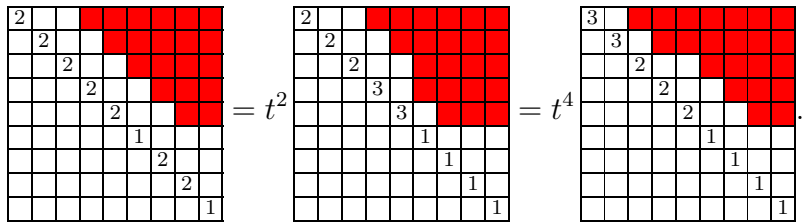
\begin{align*}
{\tiny \begin{ytableau}
   2    &        &         &*(red)   &*(red)  &*(red)   &*(red)  &*(red) &*(red)\\
        &    2   &         &        &*(red)   &*(red)  &*(red) &*(red)  &*(red)\\
        &        &     2   &        &         & *(red) &*(red)  &*(red)  &*(red)\\
        &        &         &    2     &        &        &  *(red)&*(red)  &*(red)\\
        &         &         &         &    2   &       &         &*(red)  &*(red)\\
        &         &         &         &        &   1   &         &       &     \\
        &         &         &         &        &       &    2    &        &     \\
        &         &         &         &        &       &        &   2    &   \\
        &         &         &         &        &       &        &        &  1  \\
\end{ytableau}
=
t^2\, \begin{ytableau}
   2    &        &         &*(red)   &*(red)  &*(red)   &*(red)  &*(red) &*(red)\\
        &    2   &         &        &*(red)   &*(red)  &*(red) &*(red)  &*(red)\\
        &        &     2   &        &         & *(red) &*(red)  &*(red)  &*(red)\\
        &        &         &    3    &        &  *(red) &  *(red)&*(red)  &*(red)\\
        &         &         &         &    3   &       & *(red)  &*(red)  &*(red)\\
        &         &         &         &        &   1   &         &       &       \\
        &         &         &         &        &       &   1    &        &     \\
        &         &         &         &        &       &        &   1    &   \\
        &         &         &         &        &       &        &        &  1  \\
\end{ytableau}
=
t^4\, \begin{ytableau}
   3    &        &   *(red) &*(red)   &*(red)  &*(red)   &*(red)  &*(red) &*(red)\\
        &    3   &         & *(red)  &*(red)   &*(red)  &*(red) &*(red)  &*(red)\\
        &        &     2   &        &         & *(red) &*(red)  &*(red)  &*(red)\\
        &        &         &    2    &        &        &  *(red)&*(red)  &*(red)\\
        &         &         &         &    2   &       &         &*(red)  &*(red)\\
        &         &         &         &        &   1   &         &       &       \\
        &         &         &         &        &       &   1    &        &     \\
        &         &         &         &        &       &        &   1    &   \\
        &         &         &         &        &       &        &        &  1  \\
\end{ytableau}}.
\end{align*}

For comparison, here is a similar example where $\fs^{(k)}_{\mu} = 0\,$: $\lambda = 222222222$, $\mu = 222221222$,  $z = 6$.
Then  $c=2$, $h=2$, $c' = 0$, and  $\cvr_z(\lambda) = 332221112$.
The proof of Theorem \ref{t k schur straightening many} yields
\begin{align*}
{
\tiny \begin{ytableau}
   2    &        &         &*(red)   &*(red)  &*(red)   &*(red)  &*(red) &*(red)\\
        &    2   &         &        &*(red)   &*(red)  &*(red) &*(red)  &*(red)\\
        &        &     2   &        &         & *(red) &*(red)  &*(red)  &*(red)\\
        &        &         &    2    &        &        &  *(red)&*(red)  &*(red)\\
        &         &         &         &    2   &       &         &*(red)  &*(red)\\
        &         &         &         &        &   1   &         &       &     \\
        &         &         &         &        &       &   2    &        &     \\
        &         &         &         &        &       &        &   2    &   \\
        &         &         &         &        &       &        &        &  2  \\
\end{ytableau}
=
t^4\, \begin{ytableau}
   3    &        &   *(red) &*(red)   &*(red)  &*(red)   &*(red)  &*(red) &*(red)\\
        &    3   &         & *(red)  &*(red)   &*(red)  &*(red) &*(red)  &*(red)\\
        &        &     2   &        &         & *(red) &*(red)  &*(red)  &*(red)\\
        &        &         &    2    &        &        &  *(red)&*(red)  &*(red)\\
        &         &         &         &    2   &       &         &*(red)  &*(red)\\
        &         &         &         &        &   1   &         &       &       \\
        &         &         &         &        &       &   1    &        &     \\
        &         &         &         &        &       &        &   1    &   \\
        &         &         &         &        &       &        &        &  2  \\
\end{ytableau}
}
=0.
\end{align*}

Lastly, an example where $\fs^{(k)}_{\mu} = 0$ and  $c' > 0\,$:  $\lambda = 432222222$, $\mu = 432221222$, $z = 6$.
Then $\uppath_{\Delta^k(\mu)}(z) = (6,3)$, $\uppath_{\Delta^k(\mu)}(z+1) = (7,4,2)$,  $y+1 = 2$, $c=2$, $h=1$, $c' = 1$, and  $\cvr_z(\lambda) = 442221122$.
The proof of Theorem \ref{t k schur straightening many} yields
\begin{align*}
{
\tiny \begin{ytableau}
   4    & *(red) &  *(red) &*(red)   &*(red)  &*(red)   &*(red)  &*(red) &*(red)\\
        &    3   &         & *(red) &*(red)   &*(red)  &*(red) &*(red)  &*(red)\\
        &        &     2   &        &         & *(red) &*(red)  &*(red)  &*(red)\\
        &        &         &    2    &        &        &  *(red)&*(red)  &*(red)\\
        &         &         &         &    2   &       &         &*(red)  &*(red)\\
        &         &         &         &        &   1   &         &       &     \\
        &         &         &         &        &       &   2    &        &     \\
        &         &         &         &        &       &        &   2    &   \\
        &         &         &         &        &       &        &        &  2  \\
\end{ytableau}
=
t^2\, \begin{ytableau}
   4    &*(red)  &   *(red) &*(red)   &*(red)  &*(red)   &*(red)  &*(red) &*(red)\\
        &    4   &*(red)   & *(red)  &*(red)   &*(red)  &*(red) &*(red)  &*(red)\\
        &        &     2   &        &         & *(red) &*(red)  &*(red)  &*(red)\\
        &        &         &    2    &        &        &  *(red)&*(red)  &*(red)\\
        &         &         &         &    2   &       &         &*(red)  &*(red)\\
        &         &         &         &        &   1   &         &       &       \\
        &         &         &         &        &       &   1    &        &     \\
        &         &         &         &        &       &        &   2    &   \\
        &         &         &         &        &       &        &        &  2  \\
\end{ytableau}
}
=0.
\end{align*}
\end{example}

\subsection{$k$-Schur straightening and the subset lowering operators}
Here we show that $k$-Schur straightening ``commutes'' with the subset lowering operators under some mild assumptions.

\begin{lemma}
\label{l k schur straightening 1spin ed}
Let $\mu \in  \tpar^k_\ell$, $\Psi  = \Delta^k(\mu)$, $z \in [\ell-1]$, and $V \subset [\ell]$.
Suppose that \eqref{el straighten 1spin 1}--\eqref{el straighten 1spin 3} hold
and  $V$ contains both or neither of $z, z+1$.
Recall that $y+1 := \upp_\Psi(z+1)$ by \eqref{el straighten 1spin 1}.
Then for any  $d \ge 0$,
\begin{align}
\label{el k schur straightening 1spin ed}
\beperp_{d,V} \fs^{(k)}_\mu =
t \, \beperp_{d,V} \fs^{(k)}_{\mu + \epsilon_{y+1} - \epsilon_{z+1}}
\end{align}
\end{lemma}
\begin{proof}
Repeat the proof of Lemma \ref{l k schur straightening 1spin} with
\eqref{e Rlambda recurrence HH ed} in place of \eqref{e Rlambda recurrence HH},
Lemma \ref{l toggle lemma zero ed} in place of Lemma~\ref{l toggle lemma zero},
and Lemma~\ref{l cascading toggle lemma ed} in place of Lemma~\ref{l cascading toggle lemma}.
\end{proof}

\begin{lemma}
\label{l k schur straightening 0 ed}
Let $\mu \in  \tpar^k_\ell$, $\Psi  = \Delta^k(\mu)$, $h \in \ZZ_{\ge 1}$, $z \in [\ell-h]$, and
$V \subset [\ell]$.
Suppose \eqref{el straighten0 1}--\eqref{el straighten0 4} hold and
 $V$ contains all or none of the interval  $[z, z+h]$.
Then for any  $d \ge 0$,
\begin{align}
\label{el k schur straightening 0 ed}
\beperp_{d,V} \fs^{(k)}_{\mu} = t^{h} \, \beperp_{d,V} \fs^{(k)}_{\mu + \epsilon_{[y+1,y+h]} - \epsilon_{[z+1,z+h]}}.
\end{align}
\end{lemma}
\begin{proof}
Repeat the proof of Lemma \ref{l k schur straightening 0} with
Lemma \ref{l k schur straightening 1spin ed} in place of Lemma \ref{l k schur straightening 1spin}.
\end{proof}

\begin{theorem}
\label{t k schur straightening many ed}
Maintain the notation of Definition~\ref{d cvr} and \eqref{et hx def ed}.
Let $V \subset [\ell]$ be such that  $V$ contains all or none of the interval
$[x,x+h_x]$ for all $x \in \uppath_\Psi(z)$.
Then for any $d \ge 0$,
\begin{align}
\label{et k schur straightening many ed}
\beperp_{d,V}\fs^{(k)}_{\lambda-\epsilon_{z}} = \begin{cases}
t^{\bounce(\cvr_z(\lambda),\lambda)} \beperp_{d,V}\fs^{(k)}_{\cvr_z(\lambda)} & \text{if $\cvr_z(\lambda) \in \Par^k_\ell$,}\\
0 & \text{otherwise.}
\end{cases}
\end{align}
\end{theorem}

\begin{proof}
Repeating the proof of \eqref{et k schur straightening many} with
Lemma \ref{l k schur straightening 0 ed} in place of Lemma \ref{l k schur straightening 0} gives  $\beperp_{d,V}\fs^{(k)}_{\lambda-\epsilon_{z}} =
t^{\bounce(\cvr_z(\lambda),\lambda)} \beperp_{d,V}\fs^{(k)}_{\cvr_z(\lambda)}$  (without the restriction $\cvr_z(\lambda) \in \Par^k_\ell$);
here we need that  $V$ contains all or none of the interval
$[x,x+h]$ for all $x \in \uppath_\Psi(z)$, which is certainly implied by our assumption on  $V$.

If $\cvr_z(\lambda) \in \Par^k_\ell$ we are done.
Otherwise, repeat the last three paragraphs of the
proof of Theorem \ref{t k schur straightening many} with Lemma \ref{l little lemma many zero 2 ed} in place of Lemma \ref{l little lemma many zero 2};
Lemma \ref{l little lemma many zero 2 ed}
requires that  $V$ contains all or none of $[x+h, x+h+1]$
for all  $x \in \bpath_{\tilde{\Psi}}(\upp_{\tilde{\Psi}}^{c'}(z), z) = \bpath_{\Psi}(\upp_{\Psi}^{c'}(z), z)$,
where the equality is by \eqref{e straighten step 1} and $c'$ and  $\tilde{\Psi}$  are as in the proof of Theorem \ref{t k schur straightening many};
this condition on  $V$ combined  with that of the previous paragraph is equivalent to the assumption on $V$ in the statement.
\end{proof}

\begin{example}
We continue Example~\ref{ex straightening 1}.
For the first example, $\mu = 222221221$,
we have $h_x= h = 2$ for all $x\in \uppath_{\Delta^k(\mu)}(z) = (6,3)$.
Hence the hypothesis on $V$  in Theorem~\ref{t k schur straightening many ed} is that $V$ must contain
all or none of the intervals $\{3,4,5\}$ and $\{6,7,8\}$.
For the second example, $\mu = 222221222$, we have  $\uppath_{\Delta^k(\mu)}(z) = (6,3)$, $h=2$, $h_6 = 3$, $h_3 = 2$.
Thus the intervals  $[x,x+h_x]$ are $\{3,4,5\}$ and $\{6,7,8,9\}$. So, for instance,  $V = [8]$ would work for
the previous example, but not this one.
For the third example, $\mu = 432221222$, we have  $\uppath_{\Delta^k(\mu)}(z) = (6,3)$, $h=1$, $h_6 = 2$, $h_3 = 2$.
Thus the intervals  $[x,x+h_x]$ are $\{3,4,5\}$ and $\{6,7,8\}$.
\end{example}

Theorem \ref{t k schur straightening many ed} gives the most general conditions on  $V$
for which  $k$-Schur straightening is possible, however we will only need it
for $V= [m-1]$  for each $m \in \uppath_\Psi(z)$.
Theorem~\ref{t k schur straightening many ed} does indeed apply in this case, as the following result shows.

\begin{lemma}
\label{l straightening many no overlap v2}
Let $\mu$, $\Psi = \Delta^k(\mu)$, and $z$ be as in Definition~\ref{d cvr}, and $h_x$ be as in \eqref{et hx def ed}.
Then for any $m \in \uppath_\Psi(z)$, the set $[m-1]$
contains all or none of the interval $[x,x+h_x]$ for all $x \in \uppath_{\Psi}(z)$.
\end{lemma}
\begin{proof}
This is immediate from the first paragraph of the proof of Lemma \ref{l straightening many no overlap}.
\end{proof}

\section{The bounce graph to core dictionary}
\label{s bounce graph to core dictionary}

Here we connect the combinatorics associated to bounce graphs and $k$-Schur straightening to strong (marked) covers.
In particular, we give an explicit description of strong (marked) covers in terms of the corresponding $k$-bounded partitions (Propositions \ref{p covers agree} and \ref{p bijection marked covers}).
To make this connection, we use a description of strong covers
in terms of edge sequences and offset sequences; we follow \cite{LLMSMemoirs1} for this background.

Throughout this section, fix a positive integer $k$ and let $n = k+1$.

The \emph{edge sequence} of a partition $\kappa$ is the bi-infinite binary word $p(\kappa) = p = \dots p_{-1}p_0p_1 \dots$
obtained by tracing the border of the diagram of $\kappa$ from southwest to northeast, such that every letter $1$ (resp. $0$) represents a north (resp. east) step.
We adopt the convention\footnote{This is off by 1 from the convention in \cite{LLMSMemoirs1}, but our extended offset sequences agree.} that the meeting point of the edges $p_i$ and $p_{i+1}$ is the southeast corner of a box in diagonal $i$,
where the \emph{diagonal} of a box  $(r,c)$ of a partition diagram is  $c-r$.

A partition $\kappa$ is an $n$-core if and only if each subsequence $\dots p_{i-2n}p_{i-n}p_ip_{i+n}p_{i+2n} \dots$ has the form $\dots 111000 \dots$.
Thus an  $n$-core  $\kappa$ is specified by recording, for each  $i \in \ZZ$, the
integer  $d_i$ such that $p_{i+n(d_i-1)} = 1$ and $p_{i+nd_i} = 0$.
The sequence $(d_i)_{i\in \ZZ}$ is the \emph{extended offset sequence of $\kappa$}.
Note that
\begin{align}
\label{e offset seq fact}
d_{i-n} = d_{i} + 1 \text{ \quad for all $i \in \ZZ$}.
\end{align}

The affine symmetric group $\eS_n$ can be identified with the set of permutations $w$ of $\ZZ$ such that
$w(i + n) = w(i) + n$ for all $i \in \ZZ$ and
$\sum_{i=1}^n (w(i) - i) = 0$ (see \cite{LusztigTAMS}).
For $r < s$ with $r \not\equiv s \bmod n$, the reflection $t_{r,s} \in \eS_n$ is defined by
$t_{r,s}(r+jn) = s+jn$, $t_{r,s}(s+jn) = r+jn$ for all $j \in \ZZ$ and
$t_{r,s}(i) = i$ for all $i\in \ZZ$ such that $r-i, s-i \notin n\ZZ$.
There is a natural action of  $\eS_n$
on edge sequences and extended offset sequences:
the reflection $t_{r,s}$ acts on an edge sequence $p$ by exchanging the bits $p_{r+in}$ and $p_{s+in}$ for all $i \in \ZZ$,
and $t_{r,s}d$ is obtained from $d$ in the same way.
This gives an  $\eS_n$ action on $n$-cores.

It is worth pointing (though we will only make use of this implicitly through citations) that
there is a bijection between
minimal coset representatives $\eS_{n}/\SS_{n}$ and
$n$-cores, compatible with the $\eS_n$ actions, given
by $w\SS_n \mapsto w \cdot \varnothing$, where  $\varnothing$
denotes the empty partition.
Moreover, this bijection matches strong Bruhat order with the inclusion partial order on  $n$-cores
(see Propositions 8.7, 8.8, 8.9, and 9.3 of \cite{LLMSMemoirs1}).

We also need one new definition, not given in the reference \cite{LLMSMemoirs1}.
For an $n$-core $\kappa$, the \emph{row map} of $\kappa$ is the function
\begin{align}
\label{d row map}
f: \ZZ_{\ge 1} \to \ZZ, \text{ given by } f(z) = \kappa_z-z+1.
\end{align}
With this definition, $f(z)-1$ is the diagonal of the box  $(z,\kappa_z)$ on the eastern border of  $\kappa$.
By our convention above for the sequence $p$, this means that $p_{f(z)}$ corresponds to the north step in row $z$;
in other words, $f(z)$ is equal to the index $i$ such that $p_i = 1$ and $p_i p_{i+1} \cdots$ contains $z$ 1's.

\begin{example}
Let $k = 4$, $n = 5$.  Below is the diagram of the  $n$-core $\kappa = 665443221$ with its edge sequence
 $p$ labeled, where the $\bullet$ separates  $p_0$ and  $p_1$:
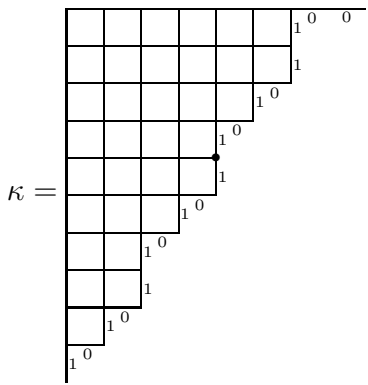
\[
\kappa = {\fontsize{7pt}{7pt}\selectfont \tableau{
~&~&~&~&~&~&\makebox[0pt][r]{\makebox[3.5pt][l]{\Tiny{1}}} \raisebox{3.5pt}[0pt][0pt]{\hspace{2pt}\Tiny{0}}\fr[t] & \raisebox{2pt}[0pt][0pt]{\Tiny{0}}\fr[t]\\
~&~&~&~&~&~ & \makebox[0pt][r]{\makebox[3.5pt][l]{\Tiny{1}}} \raisebox{3.5pt}[0pt][0pt]{\hspace{2pt}\hphantom{\Tiny{0}}} \bl \\
~&~&~&~&~& \makebox[0pt][r]{\makebox[3.5pt][l]{\Tiny{1}}} \raisebox{3.5pt}[0pt][0pt]{\hspace{2pt}\Tiny{0}}\bl \\
~&~&~&~ & \makebox[0pt][r]{\makebox[3.5pt][l]{\Tiny{1}}}\raisebox{3.5pt}[0pt][0pt]{\hspace{2pt}\Tiny{0}} \bl \\
~&~&~& \makebox[10pt][r]{\makebox[0pt][l]{\raisebox{4.5pt}[0pt][2pt]{$\bullet$}}}
     &\makebox[0pt][r]{\makebox[3.5pt][l]{\Tiny{1}}} \raisebox{3.5pt}[0pt][0pt]{\hspace{2pt}\hphantom{\Tiny{0}}} \bl \\
~&~&~& \makebox[0pt][r]{\makebox[3.5pt][l]{\Tiny{1}}} \raisebox{3.5pt}[0pt][0pt]{\hspace{2pt}\Tiny{0}}\bl \\
~&~& \makebox[0pt][r]{\makebox[3.5pt][l]{\Tiny{1}}} \raisebox{3.5pt}[0pt][0pt]{\hspace{2pt}\Tiny{0}}\bl \\
~&~&\makebox[0pt][r]{\makebox[3.5pt][l]{\Tiny{1}}} \raisebox{3.5pt}[0pt][0pt]{\hspace{2pt}\hphantom{\Tiny{0}}} \bl \\
~&\makebox[0pt][r]{\makebox[3.5pt][l]{\Tiny{1}}} \raisebox{3.5pt}[0pt][0pt]{\hspace{2pt}\Tiny{0}}\bl \\
\makebox[0pt][r]{\makebox[3.5pt][l]{\Tiny{1}}} \raisebox{3.5pt}[0pt][0pt]{\hspace{2pt}\Tiny{0}}\fr[l]}}\]
\[
\small
\begin{array}{r|rrrrrrrrrrrrrrrrrrrrrrrrrrrrrrrrrr}
i & \!\!   -9&\!\! -8&\!\! -7&\!\! -6&\!\! -5&\!\! -4&\!\! -3&\!\! -2&\!\! -1&\!\! 0&\!\! 1&\!\! 2&\!\! 3&\!\! 4&\!\! 5&\!\! 6&\!\! 7&\!\! 8 \\
\hline \\[-3.5mm]
p_i &\!\!  1&\!\! 0&\!\! 1&\!\! 0&\!\! 1&\!\! 1&\!\! 0&\!\! 1&\!\! 0&\!\! 1&\!\! 1&\!\! 0&\!\! 1&\!\! 0&\!\! 1&\!\! 1&\!\! 0&\!\! 0\\
d_i &\!\!  4&\!\! 0&\!\! 3&\!\! 0&\!\! 3&\!\! 3&-1&\!\! 2&-1&\!\! \hphantom{-}2&\!\! \hphantom{-}2&-2&\!\! \hphantom{-}1&-2&\!\! \hphantom{-}1&\!\! \hphantom{-}1&-3 &\!\! \hphantom{-}0
\end{array}
\]

\pagebreak[2]

It is convenient to depict the edge sequence  in an $\infty \times n$ array with the $r$-th row equal to the sequence $p_{1+nr}p_{2+nr}\dots p_{n+nr}$, where the horizontal line divides rows 0 and $-1$.
Then the entries  $d_1 d_2 \cdots d_n$ record the heights the 1's attain in columns $1,2,\dots, n$.
\[
{\footnotesize
\begin{array}{ccccc}
\vdots& \vdots & \vdots & \vdots & \vdots\\
0 & 0 & 0 & 0 & 0 \\
1 & 0 & 0 & 0 & 0\\
1 & 0 & 1 & 0 & 1\\ \hline
1 & 0 & 1 & 0 & 1\\
1 & 0 & 1 & 0 & 1\\
1 & 1 & 1 & 1 & 1\\[-1.3mm]
\vdots& \vdots & \vdots & \vdots & \vdots
\end{array}}\]
The row map is given by $f(1), f(2), \dots = 6, 5, 3, 1, 0, -2, -4, -5, -7, -9, -10, -11, \dots$.
\end{example}

Our first task in this section is to give the translation  between bounce paths and extended offset sequences.
We write  $\core$ for the inverse of the bijection $\p$, which is a map from  $\Par^k$ to  $k+1$-cores.

\begin{proposition}
\label{p bounce path vs offset}
Let $\lambda \in \Par^k_\ell$ and $\Phi = \Delta^k(\lambda)$. Let $\kappa = \core(\lambda)$
have edge sequence  $p$, extended offset sequence  $d$, and row map  $f$.
Let $r \in \ZZ_{\ge 1}$ and  $z \in [\ell]$.  Then
\begin{list}{\emph{(\alph{ctr})}} {\usecounter{ctr} \setlength{\itemsep}{1pt} \setlength{\topsep}{2pt}}
\item $f(r) = f(r+n-\lambda_r) +n$;
\item $f(\upp_\Phi(z))= f(z) +n$ whenever  $\upp_\Phi(z)$ is defined;
\item $d_{f(z)} = d_{f(\upp_\Phi(z))} + 1$ whenever $\upp_\Phi(z)$ is defined;
\item $d_{f(z)} > 1 \iff \upp_\Phi(z)\text{ is defined}$;
\item $d_{f(z)} = |\uppath_\Phi(z)|$;
\item $n-\lambda_z = p_{f(z)-n+1} +  p_{f(z)-n+2} + \dots + p_{f(z)}$;
\item $\lambda_z = \lambda_{z+1}  \ \iff \ $ there are no 0's in  $d_{f(z+1)} d_{f(z+1)+1} \dots d_{f(z)}$.
\end{list}
\end{proposition}
\begin{proof}
We first prove (a).  First suppose  $\kappa_r > \lambda_r$ and let $(r,c) = (r, \kappa_r-\lambda_r)$ be the easternmost box in the $r$-th row of $\kappa$ having hook length  $> n$.
For this box to have hook length  $> n$, the box  $(r+n-\lambda_r, c)$ must lie in $\kappa$.
By definition of the map $\p$ (see Section~\ref{s main results}), the box $(r, c+1)$ has hook length  $< n$ in  $\kappa$,
and thus $(r+n-\lambda_r, c+1) \notin \kappa$.   Hence $(r+n-\lambda_r, c)$ lies on the eastern border of  $\kappa$ and we have
\[ f(r) = \kappa_r-r+1 = c+\lambda_r-r+1 = c-(r+n-\lambda_r)+1  + n = f(r+n-\lambda_r) +n.\]
This argument also works in the case $\kappa_r = \lambda_r$ if we consider the column  $\{(i,0)\mid i \in \ZZ_{\ge 1}\}$ to be part of $\kappa$.

Statement (b) follows from (a) by setting  $r = \upp_\Phi(z)$ and
using that $z = \down_\Phi(r) = r+n-\lambda_r$.
Statement  (c) follows from (b) and \eqref{e offset seq fact},
(e) follows from (c) and (d),
(f) follows from (a), and (g) follows from (f).

It remains to prove (d).
The  $\Leftarrow$ direction is immediate from (c).  For the  $\Rightarrow$ direction,
suppose $d_{f(z)} > 1$. Then
$p_{f(z)+n} = 1$ and $f(r) = f(z)+n$ for some  $r < z$; hence  $z= r+n-\lambda_r$ by (a).
Note that $\down_\Phi(r)$ is defined exactly when  $r+n-\lambda_r \le \ell$ (by Proposition \ref{p Delta k basics});
since  $z \le \ell$ we then have  $z= \down_\Phi(r)$ and $\upp_\Phi(z) = r$.
\end{proof}

Recall that a strong cover $\tau \Rightarrow \kappa$ is a pair
of $n$-cores such that $\tau \subset \kappa$ and $|\p(\tau)| + 1 = |\p(\kappa)|$.

\begin{lemma}
\label{l cover by offset sequence}
Let $\kappa$ be an  $n$-core with extended offset sequence  $d$ and $t_{r,s} \in \eS_n$ a reflection.
\begin{list}{\emph{(\roman{ctr})}} {\usecounter{ctr} \setlength{\itemsep}{1pt} \setlength{\topsep}{2pt}}
\item $\kappa \Rightarrow t_{r,s} \kappa$ if and only if $d_r > d_s$ and for all $r < i < s$, $d_i \notin [d_s, d_r]$.
\item $t_{r,s} \kappa \Rightarrow \kappa$ if and only if  $s-n < r$, $d_r < d_s$, and for all $r < i < s$, $d_i \notin [d_r, d_s]$.
\item If $t_{r,s} \kappa \Rightarrow \kappa$, then the skew shape $\kappa/(t_{r,s} \kappa)$
has components $R_{d_r}, \dots, R_{d_s-1}$, where $R_{j}$ is the ribbon with $s-r$ boxes in diagonals $r+nj, r+1+nj, \dots, s-1+nj$.
\item For each $j \in [d_r, d_{s}-1]$, let $z_j$ be the smallest row index of the ribbon $R_{j}$ from (iii).
Then $f(z_j) = s+n j$ and $\uppath_{\Delta^k(\p(\kappa))}(z_{d_r}) = (z_{d_r}, z_{d_r+1}, \dots, z_{d_s-1})$.
\end{list}
\end{lemma}
\begin{proof}
Statement (i) is {\cite[Lemma 9.4 (2)]{LLMSMemoirs1}}.
Statement (ii) follows from (i) and \eqref{e offset seq fact}.
Statement (iii) is essentially \cite[Proposition 9.5]{LLMSMemoirs1};
it follows from the interpretation of $t_{r,s}\kappa \Rightarrow \kappa$ in terms of edge sequences.
We now prove (iv). By (iii),
the northeastmost box of the ribbon $R_j$ is the box in diagonal $s-1 + nj$ on the eastern border of $\kappa$.
The row $z_j$ containing this box satisfies $f(z_j)=s+nj$ by the discussion following \eqref{d row map}.
The statement about the uppath then follows from Proposition~\ref{p bounce path vs offset} (b) and (e).
\end{proof}

The next result
gives a description of  strong covers
with the focus on the row indices of the shape  $\kappa$ rather than diagonals.
This result as well as Lemmas \ref{l uppath plus one v2} and \ref{l cvr definitions} prepare us to prove the main results of this section, which connect strong covers
 to  $\cvr_z(\lambda)$ from  $k$-Schur straightening.

\begin{lemma}
\label{l covers agree prep}
Let  $\kappa$ be an $n$-core with extended offset sequence $d$ and row map $f$.
Let  $z \in [\ell(\kappa)]$ and set  $s = f(z)$.
Let  $r < s$ be as small as possible such that
\begin{align}
\label{el covers agree prep}
\text{ $\big(d_{i} > d_s$ \ or \ $d_{i} <0 \big)$ \, for all $i \in [r+1,s-1]$.}
\end{align}
There exists a strong cover $\tau\Rightarrow \kappa$ such that the southwestmost component of $\kappa/\tau$ has smallest row index $z$ if and only if $(d_{r} = 0$ and  $s-n < r)$.
Moreover, this cover is unique if it exists and $\tau = t_{r,s} \, \kappa$.
\end{lemma}

\begin{definition}
\label{d cover}
We define $\cover_z(\kappa) = \tau$ if the strong cover in Lemma \ref{l covers agree prep} exists, and otherwise we say that $\cover_z(\kappa)$ does not exist.
\end{definition}

\begin{proof}
For the ``only if'' direction, suppose $\tau\Rightarrow \kappa$ is a strong cover such that the southwestmost component of $\kappa/\tau$ has smallest row index $z$.
We can write $\tau = t_{r',s'}\kappa$ with  $r'<s'$ determined uniquely by requiring $d_{r'} =0$.
Then $s = f(z) = s'+n d_{r'}= s'$ by Lemma~\ref{l cover by offset sequence} (iii)--(iv).
By Lemma~\ref{l cover by offset sequence}~(ii),  $r= r'$, and hence $d_r=0$ and $s-n<r$.  This also establishes uniqueness.

For the ``if'' direction, suppose  $d_r =0$ and $s-n<r$.
Note that  $d_s > 0$ since  $s \in \im(f)$.
Then by Lemma~\ref{l cover by offset sequence}~(ii),  $t_{r,s}\kappa\Rightarrow \kappa$.
By Lemma \ref{l cover by offset sequence} (iii)--(iv), the southwestmost component of $\kappa/\tau$ has smallest row index $z$.
\end{proof}

\begin{example}
\label{ex strong cover and z}
For $\kappa = 665443221$ and $z=6$, we have
$\cover_z(\kappa) = 663331111$.
The strong cover  $\cover_z(\kappa) \Rightarrow \kappa$ is the same as that of Example \ref{ex strong cover and spin}.
The key relevant quantities from
Lemma~\ref{l covers agree prep} are $s = -2$, $r = -6$, and  $d_{r} d_{r+1} \cdots d_s = 0 \ 3 \ 3 \ {-1}\ 2$.
\end{example}

\begin{lemma}
\label{l uppath plus one v2}
Let $\lambda \in \Par^k_\ell$
and set $\Phi = \Delta^k(\lambda)$.
Let  $j \in [\ell]$ and $i \in \uppath_{\Phi}(j)$.
Then $|\bpath_\Phi(i,j)| \le |\uppath_\Phi(j+1)|$  if and only if
 $\lambda_x = \lambda_{x+1}$ for all $x \in \bpath_\Phi(i,j) \setminus \{j\}$;  if these conditions hold, then
$\bpath_{\Phi}(i+1,j+1) = \{x+1 \mid x \in \bpath_\Phi(i,j)\}$.
\end{lemma}
\begin{proof}
This is a direct consequence Proposition \ref{p Delta k basics}.
\end{proof}

\begin{remark}
\label{rp bijection marked covers}
Let $\lambda \in \Par^k_\ell$ and $\mu= \lambda-\epsilon_z$ ($z \in [\ell]$),
and set $\Phi = \Delta^k(\lambda)$ and  $\Psi = \Delta^k(\mu)$.  We have
$\uppath_\Phi(z) = \uppath_\Psi(z)$ and if $|\uppath_\Psi(z+1)| > 1$, then $\uppath_\Phi(z+1) = \uppath_\Psi(z+1)$.
%
%
\end{remark}

\begin{lemma}
\label{l cvr definitions}
Maintain the notation of Definition \ref{d cvr}
and set  $\Phi = \Delta^k(\lambda)$.
Then \break $h = \min(k-\lambda_z, h')$, where  $h'$ is the largest element  of
$[0,\ell-z]$ such that
\begin{align}
\label{d cvr 0 v2}
\text{$|\uppath_{\Phi}(z)| < |\uppath_{\Phi}(z+i)|$ and  $\lambda_z = \lambda_{z+i}$ for all $i \in [h']$}.
\end{align}
\end{lemma}
\begin{proof}
One checks directly that $h=0$ if and only if $h'=0$.
So now assume $h>0$ and $h'>0$.
One checks using Remark \ref{rp bijection marked covers} and Lemma \ref{l uppath plus one v2} with  $j = z+1$, then  $j=z+2$, $\dots, \, j = z+h'-1$ and  $i = \upp_\Phi^c(j)$
that $h'$ is the largest element of $[z-\ell]$ such that
\begin{align*}
\parbox{14cm}{$\lambda$ is constant on each of the intervals $[z+1, z+h']$,
$[\upp_{\Psi}(z), \upp_{\Psi}(z)+h']$, \\[1mm]
$[\upp^2_{\Psi}(z), \upp^2_{\Psi}(z)+h']$, $\dots$, $[\chainup_{\Psi}(z), \chainup_{\Psi}(z)+h']$, and  $[y+1, y+h']$.}
\end{align*}
This is the same as the definition of  $h$ except with  $\lambda$ in place of $\mu = \lambda-\epsilon_z$.  Hence showing  $h = \min(k-\lambda_z, h')$ amounts to checking
\begin{align}
\label{e min check}
& \min\big(k-\lambda_z, g(\lambda)\big)
= \min\big(z-\upp_\Psi(z+1), g(\lambda)\big) = g(\mu), \\
& \text{where $g(\nu) := \max\big\{i \mid \nu \text{ is constant on } [\upp_\Psi(z+1),\upp_\Psi(z+1)+i-1]\big\}$ for  $\nu \in \ZZ^\ell$} \notag
\end{align}
(the second equality of \eqref{e min check} holds by $\mu_{z-1}> \mu_z$
$\implies$  $g(\mu) \le z-\upp_\Psi(z+1)$).
We have
\[k-\lambda_{z} \ge k-\lambda_{\upp_\Psi(z+1)}
 = k-\mu_{\upp_\Psi(z+1)} = z-\upp_\Psi(z+1),\]
with equality if and only if  $g(\lambda) > z-\upp_\Psi(z+1)$ (the last equality is by Proposition \ref{p Delta k basics}).
This proves \eqref{e min check}.
\end{proof}

The next three results establish a dictionary between constructions on the core side (strong covers, spin)
and constructions on bounce graph side ($\cvr_z(\lambda)$, bounce).

\begin{proposition}
\label{p covers agree}
Let $\lambda \in \Par^k_\ell$, $\kappa = \core(\lambda)$, and  $z \in [\ell]$.  Then
\begin{list}{\emph{(\roman{ctr})}} {\usecounter{ctr} \setlength{\itemsep}{2pt} \setlength{\topsep}{1pt}}
\item $\cover_z(\kappa)$ exists if and only if  $\cvr_z(\lambda) \in \Par^k_\ell$;
\item if  $\cvr_z(\lambda) \in \Par^k_\ell$, then $\cvr_z(\lambda) = \p(\cover_z(\kappa))$.
\end{list}
\end{proposition}
\begin{proof}
If  $\lambda_z =0$, then both conditions in (i) fail.  So now assume  $\lambda_z > 0$.
Maintain the notation of Lemma~\ref{l covers agree prep} so that $s = f(z)$ and $r$ is defined by \eqref{el covers agree prep}.
The edge sequence, extended offset sequence, and row map of  $\kappa$ are denoted $p$, $d$, and $f$, respectively.

We first prove (i).
For an index  $r' < s$, define $\tilde{h}(r') = p_{r'+1}+p_{r'+2} +\dots+p_{s-1}$ \break $ = |\{i \in [r'+1, s-1] : d_i > 0\}|$.
One checks using Proposition~\ref{p bounce path vs offset} (e) and (g) that  $\tilde{h}(r)$ is equal to the  $h'$ defined in Lemma~\ref{l cvr definitions} and
\begin{align}
\label{e dr eq 0 1}
d_r= 0 \ \iff \ \lambda_{z+{h'}} > \lambda_{z+{h'}+1}.
\end{align}
By Lemma~\ref{l covers agree prep} and Lemma \ref{l more about d cvr} (ii), statement  (i) amounts to showing
\begin{align}
\label{e dr eq 0 2}
\big(d_r= 0 \text{ and } s-n < r \big) \ \iff \ \lambda_{z+{h}} > \lambda_{z+{h}+1},
\end{align}
where  $h$ is as in Definition \ref{d cvr}.
We treat the cases  $s-n<r$ and  $s-n \ge r$ separately.
By Proposition \ref{p bounce path vs offset} (f),  $k-\lambda_z = \tilde{h}(s-n)$;
this together with Lemma \ref{l cvr definitions} yields
\break $s-n < r \implies h=h'$, and  $s-n \ge r \implies  h = k-\lambda_z$.
Hence in the case $s-n < r$, \eqref{e dr eq 0 1} implies \eqref{e dr eq 0 2}.
In the case $s-n \ge r$, using that $d_{s-n}= d_s+1 > 0$ and $h = k-\lambda_z = \tilde{h}(s-n)$ we have $f(z+h+1)= s-n$; by definition of  $r$,
there are no 0's in  $d_{f(z+h+1)} \cdots d_{f(z+h)}$, and
thus    $\lambda_{z+{h}} = \lambda_{z+{h}+1}$ by Proposition~\ref{p bounce path vs offset} (g).

We now prove (ii).  By (i),  $\tau := \cover_z(\kappa)$ exists.
The border sequence of  $\tau$ is  $q := t_{r, s} p$.
Define the sequence $\tilde{p}$ by $\tilde{p}_i = n-\sum_{j \in [i-n+1, i]} p_i$, and define $\tilde{q}$ similarly in terms of $q$.
Since  $q=p + \sum_{i \in [0,d_s-1]}(\epsilon_{r+in} - \epsilon_{s+in})$,  we have
\begin{align}
\label{e tilde p}
\tilde{q} = \tilde{p} + \epsilon_{[r+d_s n, s+d_s n-1]} -  \epsilon_{[r, s-1]} .
\end{align}
By Proposition~\ref{p bounce path vs offset}  (f),   $\lambda$ (resp.  $\p(\tau)$) is the subsequence of $\tilde{p}$ (resp. $\tilde{q}$)
over those indices $j$ such that  $p_j=1$  (resp.  $q_j = 1$).
By the definition of  $r$, the sequence $\tilde{p}$ is constant on each interval $[r+in,s+in]$ for $i \in [0,d_s-1]$.
Using this, \eqref{e tilde p}, and $q = t_{r, s} p$, one checks that
$\p(\tau)= \lambda + \epsilon_I - \epsilon_J$,
where  $I = f^{-1}([r+d_s n, s+d_s n-1])$
and  $J = f^{-1}([r,s])$.
Let $y$ be as in Definition \ref{d cvr}.
We have  $I = [y+1, y+h]$ and  $J = [z, z+h]$ by  $h = \tilde{h}(r)$ (from the proof of (i)) and Proposition~\ref{p bounce path vs offset} (b) and (e).
Hence $\p(\tau)= \cvr_z(\lambda)$ as desired.
\end{proof}

\begin{example}
Let  $\lambda = 222222221$ and  $\kappa = 665443221$.
By Examples \ref{ex straightening 1} and \ref{ex strong cover and z},
\[\cvr_z(\lambda) = 332221111 = \p(663331111) =  \p(\cover_z(\kappa)),\]
in agreement with Proposition \ref{p covers agree} (ii).
\end{example}


\begin{proposition}
\label{p bijection marked covers}
Fix $\lambda \in \Par^k_\ell$ and set  $\Phi = \Delta^k(\lambda)$ and  $\kappa= \core(\lambda)$.
There is a bijection
\[
\big\{(\nu, z, m) \in \Par^k_\ell \times [\ell] \times [\ell]  \mid z \in \downpath_{\Phi}(m) , \, \nu = \cvr_z(\lambda) \big\} \xrightarrow{\, \cong \,}
\VSMT^k_{(1)}(\lambda)
\]
given by  $(\nu, z, m) \mapsto (\core(\nu) \xRightarrow{~~m~~} \kappa)$.
\end{proposition}

\begin{proof}
The set $\VSMT^k_{(1)}(\lambda)$  is just another notation for the set of all strong marked covers
$\tau \xRightarrow{~~m~~} \kappa$,
which, by Lemma \ref{l covers agree prep}, can be written as
the union of $\{(\cover_z(\kappa) \xRightarrow{~~m~~} \kappa)\}$
over all ${z, m \in [\ell]}$ such that $\cover_z(\kappa)$ exists and  $m$ is a possible marking of  $\cover_z(\kappa) \Rightarrow \kappa$.
By Lemma~\ref{l cover by offset sequence} (iii)--(iv),
the possible markings of
$\cover_z(\kappa) \Rightarrow \kappa$ are exactly the elements of $\uppath_\Phi(z)$.
The result then follows from
Proposition \ref{p covers agree}.
\end{proof}

Recall the definitions of spin from Section \ref{s main results} and bounce from Definition \ref{d cvr}.

\begin{proposition}
\label{p spins agree}
The bijection of Proposition \ref{p bijection marked covers} takes bounce to spin:
for \break $(\cvr_z(\lambda), z, m) \mapsto T$, we have
\[
\bounce(\cvr_z(\lambda), \lambda)+B_\Phi(m,z) = \spin(T). \]
\end{proposition}
\begin{proof}
We have  $T = (\core(\cvr_z(\lambda)) \xRightarrow{~~m~~} \kappa)= (\cover_z(\kappa) \xRightarrow{~~m~~} \kappa)$ (by Proposition \ref{p covers agree}).
Let  $c$ and $h$ be as in Definition \ref{d cvr}.
By Lemma~\ref{l cover by offset sequence} (iv), the
number of components of  $\kappa/\cover_z(\kappa)$ is $|\uppath_\Phi(z)| = c.$
By Lemma~\ref{l cover by offset sequence} (iii), the height of each component is the number of 1's (north steps) in $p_r p_{r+1} \cdots p_{s-1} p_s$,
where  $p$ is the edge sequence of $\kappa$ and $\cover_z(\kappa) = t_{r,s} \kappa$ as in Lemma \ref{l covers agree prep}.
This is equal to  $\tilde{h}(r) +1 = h+1$ by the proof of Proposition \ref{p covers agree} (i).
So  $\bounce(\cvr_z(\lambda), \lambda)$ is equal to the number of components of  $\kappa/\cover_z(\kappa)$ times one less than the height of each component.
Finally, by Lemma~\ref{l cover by offset sequence}~(iv),
the number  of components of $\kappa/\cover_z(\kappa)$
entirely contained in rows  $> m$
is equal to  $|\bpath_\Phi(\down_\Phi(m), z)| = B_\Phi(m,z)$.
\end{proof}

\begin{example}
In Example \ref{ex strong cover and spin} we computed the spins of the possible markings of
the strong cover $663331111 = \tau \Rightarrow \kappa = 665443221$ (reproduced on the right in \eqref{e bounce 1}--\eqref{e bounce 2}).
We have  $\tau = \cover_z(\kappa)$ for  $z = 6$  (see Example \ref{ex strong cover and z}) and $\lambda = \p(\kappa) = 222222221$.
We verify Proposition \ref{p spins agree} for each $m \in \uppath_{\Delta^k(\lambda)}(z) = \{6, 3\}$:
\begin{align}
\bounce(\cvr_z(\lambda), \lambda)+B_{\Delta^k(\lambda)}(6,z) = 4 + 0 = 4 = \spin(\tau \xRightarrow{~~6~~} \kappa), \label{e bounce 1}\\
\bounce(\cvr_z(\lambda), \lambda)+B_{\Delta^k(\lambda)}(3,z) = 4 + 1 = 5 = \spin(\tau \xRightarrow{~~3~~} \kappa), \label{e bounce 2}
\end{align}
where $\bounce(\cvr_z(\lambda), \lambda)= h \cdot c = 2 \cdot 2 = 4$ (by Example \ref{ex straightening 1}).
\end{example}

\section{The dual Pieri rules}


To better understand the dual Pieri rules for  $k$-Schur functions, let us first consider the
dual Pieri rule for ordinary Schur functions for  $d=1$. This rule
can be stated in the following elegant way:
\[e_1^\perp s_\lambda = \sum_{z=1}^\ell s_{\lambda - \epsilon_z}.\]
This is the same as the sum of Schur functions indexed by partitions obtained by removing
a corner from  $\lambda$
since  $s_{\lambda-\epsilon_z}$ is nonzero
if and only if  $(z, \lambda_z)$ is a removable corner.
We have the following analogous formula for the $\fs^{(k)}_\lambda$ (by summing \eqref{ep f1m k schur} below over  $m \in [\ell]$):
\begin{align}
e_1^\perp \fs^{(k)}_\lambda &= \sum_{m \in [\ell]} \sum_{z \in \downpath_\Phi(m)} t^{B_\Phi(m,z)}\fs^{(k)}_{\lambda - \epsilon_z};
\end{align}
the (vertical) dual Pieri rule
for  $d=1$ is then obtained by evaluating the right side using $k$-Schur straightening II.

To prove the vertical dual Pieri rule for general $d$, 
we need a more refined version to induct on  $d$.
This refined statement involves the subset lowering operators, recalled below.
Somewhat miraculously, the combinatorics of bounce graphs and  $k$-Schur straightening matches
exactly the combinatorics of strong covers---we obtain an algebraic meaning for the sum over strong covers
with a fixed marking; see Theorem \ref{t flm k schur}. This is the refined statement we need for  $d=1$.
The general case is then handled in Theorem~\ref{t partial restriction}.  (The former is a special case of the latter so is unnecessary but we include it as an instructive warmup.)

We will need the following variants of the notation from Definition \ref{d subset lowering} for the subset lowering operators:
for  $d \ge 0$, $m \in [\ell]$, and an indexed root ideal  $(\Phi, \lambda)$ of length  $\ell$, define
\begin{align}
\notag
\beperp_{d,m} \HH(\Phi;\lambda) &:= \beperp_{d,[m]} \HH(\Phi;\lambda) = \sum_{S \subset [m], \, |S| = d} \HH(\Phi;\lambda-\epsilon_S)\, ; \\
\tL_{d,m} \HH(\Phi;\lambda) &:= \beperp_{d,m}\HH(\Phi;\lambda) - \beperp_{d,m-1}\HH(\Phi;\lambda) = \sum_{S \subset [\ell], \, |S| = d, \, \max(S) = m} \HH(\Phi;\lambda-\epsilon_S).
\label{d fdm}
\end{align}
By Lemma~\ref{l ed perp HH}, $\beperp_{d, \ell} H(\Phi;\lambda) = e_d^\perp H(\Phi;\lambda)$ for any  $d \ge 0$, and
$e_d^\perp H(\Phi;\lambda) = 0$ when  $d > \ell$.
Note that  $\beperp_{d,m} \HH(\Phi;\lambda) = \tL_{d,m} \HH(\Phi;\lambda) = 0$ when $d > m$,
the natural generalization of this latter fact.

\begin{proposition}
\label{p f1m k schur}
Let $\lambda \in \Par^k_\ell$, $\Phi = \Delta^k(\lambda)$, and  $m \in [\ell]$.
Then
\begin{align}
\label{ep f1m k schur}
\tL_{1,m} \fs^{(k)}_\lambda &= \HH(\Phi;\lambda - \epsilon_m) = \sum_{z \in \downpath_\Phi(m)} t^{B_\Phi(m,z)}\fs^{(k)}_{\lambda - \epsilon_z}.
\end{align}
\end{proposition}
\begin{proof}
We have
\begin{align*}
\HH(\Phi;\lambda - \epsilon_m)  \ = \sum_{z \in \downpath_\Phi(m)} t^{B_\Phi(m,z)}\HH(\Psi^z ;\lambda - \epsilon_z)
\ = \sum_{z \in \downpath_\Phi(m)} t^{B_\Phi(m,z)}\fs^{(k)}_{\lambda - \epsilon_z} \, ,
\end{align*}
where  $\Psi^z := \Phi \setminus \{(z,\down_\Phi(z))\}$ for  $z \ne \chaindown_\Phi(m)$ and $\Psi^z := \Phi$ for  $z = \chaindown_\Phi(m)$.
The first equality is by Corollary~\ref{c inductive computation atom down}
and the second is by $\Delta^k(\lambda - \epsilon_z) = \Psi^z$.
\end{proof}

For a (vertical) strong marked tableau  $T = (\kappa^{(0)} \xRightarrow{~~r_1~~} \kappa^{(1)} \xRightarrow{~~r_2~~} \cdots  \xRightarrow{~~r_m~~} \kappa^{(m)})$,
define $\markz(T)$ to be the set $\{r_1, \dots, r_m\}$ of markings that appear in $T$.

\begin{theorem}
\label{t flm k schur}
For any $\lambda \in \Par^k_\ell$ and  $m \in [\ell]$,
\begin{align}
\label{ec flm k schur}
\tL_{1,m} \fs^{(k)}_\lambda &=  \sum_{T \in \VSMT^k_{(1)}(\lambda), \: \markz(T) = \{m\}} t^{\spin(T)}\fs^{(k)}_{\inside(T)}.
\end{align}
\end{theorem}
\begin{proof}
Set  $\Phi = \Delta^k(\lambda)$.
Beginning with Proposition \ref{p f1m k schur}, we obtain
\begin{align}
\tL_{1,m} \fs^{(k)}_\lambda &= \sum_{z \in \downpath_\Phi(m)} t^{B_\Phi(m,z)}\fs^{(k)}_{\lambda - \epsilon_z} \\
&= \sum_{\{z \in \downpath_\Phi(m) \, \mid \, \cvr_z(\lambda) \in \Par^k_\ell\}} t^{B_\Phi(m,z)+\bounce(\cvr_z(\lambda),\lambda)}\fs^{(k)}_{\cvr_z(\lambda)} \\
\label{e f1m k schur 2}
& = \sum_{T \in \VSMT^k_{(1)}(\lambda), \: \markz(T) = \{m\}} t^{\spin(T)}\fs^{(k)}_{\inside(T)},
\end{align}
where the second equality is by Theorem \ref{t k schur straightening many}
and the third equality is by Propositions \ref{p bijection marked covers} and \ref{p spins agree}.
\end{proof}

\subsection{The vertical dual Pieri rule}

\begin{theorem}
\label{t partial restriction}
For any $\lambda \in \Par^k_\ell$ and integers $d \ge 0$,  $m \in [\ell]$, we have
\begin{align}
\label{et partial restrict 1}
\beperp_{d,m}\fs^{(k)}_\lambda ~ &= \sum_{T \in \VSMT^k_{(d)}(\lambda), \: \markz(T) \subset [m]} t^{\spin(T)} \fs^{(k)}_{\inside(T)} \, ;\\
\label{et partial restrict 2}
\tL_{d,m}\fs^{(k)}_\lambda ~ &= \sum_{T \in \VSMT^k_{(d)}(\lambda), \: \max(\markz(T)) = m} t^{\spin(T)} \fs^{(k)}_{\inside(T)} \, .
\end{align}
\end{theorem}
In the special case  $m = \ell$, the condition $\markz(T) \subset [m]$ in \eqref{et partial restrict 1} is no restriction at all.
Hence this proves Property \eqref{et three properties 2}.

\begin{proof}
Since $\tL_{d,m}\fs^{(k)}_\lambda = \beperp_{d,m}\fs^{(k)}_\lambda - \beperp_{d,m-1}\fs^{(k)}_\lambda$,
statements \eqref{et partial restrict 1} and \eqref{et partial restrict 2} are equivalent.
We will prove them simultaneously by induction on $d$.  Specifically, we will prove \eqref{et partial restrict 2} using \eqref{et partial restrict 1} for  $d-1, m-1$.
The base cases  $d=0$ and  $m < d$ are trivial.  So now assume  $0 < d \le m$.

Set $\Phi = \Delta^k(\lambda)$.
By definition \eqref{d fdm},
$\tL_{d,m}\fs^{(k)}_\lambda = \sum_{S \subset [\ell], \, |S| = d, \, \max(S) = m} \HH(\Phi;\lambda-\epsilon_S)$.
For each term $\HH(\Phi;\lambda-\epsilon_S)$ in this sum,
expand on $\downpath_\Phi(m)$ using Corollary~\ref{c inductive computation atom down} to obtain
\begin{align}
\HH(\Phi;\lambda - \epsilon_S) &= \sum_{z \in \downpath_\Phi(m)} t^{B_\Phi(m,z)} \HH(\Psi^z ; \lambda - \epsilon_z - \epsilon_{S'}),
\label{e expand down 2}
\end{align}
where  $\Psi^z := \Phi \setminus \{(z,\down_\Phi(z))\}$ for  $z \ne \chaindown_\Phi(m)$ and $\Psi^z := \Phi$ for  $z = \chaindown_\Phi(m)$,
and $S' := S \setminus \{m\}$.

Summing \eqref{e expand down 2} over all $S \subset [\ell]$ such that $|S| = d$ and $\max(S) = m$, we obtain
\begin{align*}
\tL_{d,m}\fs^{(k)}_\lambda
&= \sum_{S' \subset [m-1], \: |S'| = d-1} \ \sum_{z \in \downpath_\Phi(m)} t^{B_\Phi(m,z)} \HH(\Psi^z ; \lambda- \epsilon_z - \epsilon_{S'})\\
&= \sum_{z \in \downpath_\Phi(m)} t^{B_\Phi(m,z)} \beperp_{d-1,m-1} \HH(\Psi^z; \lambda-\epsilon_z) \\
&= \sum_{\{z \in \downpath_\Phi(m) \, \mid \, \cvr_z(\lambda) \in \Par^k_\ell \}}  t^{B_\Phi(m,z)+\bounce(\cvr_z(\lambda),\lambda)} \beperp_{d-1,m-1}\fs^{(k)}_{\cvr_z(\lambda)} \\
&= \sum_{V \in \VSMT^k_{(1)}(\lambda), \: \markz(V) = \{m\}} t^{\spin(V)}\beperp_{d-1,m-1} \fs^{(k)}_{\inside(V)}\\
&= \sum_{\substack{V \in \VSMT^k_{(1)}(\lambda) \\ \markz(V) =\{ m\}}} \ \sum_{\substack{U \in \VSMT^k_{(d-1)}(\inside(V)) \\ \markz(U) \subset [m-1]}} t^{\spin(V)+\spin(U)} \fs^{(k)}_{\inside(U)} \\
&= \sum_{T \in \VSMT^k_{(d)}(\lambda), \: \max(\markz(T)) = m} t^{\spin(T)} \fs^{(k)}_{\inside(T)}.
\end{align*}
We need to justify the last four equalities.
The third equality is the delicate step where we apply Theorem~\ref{t k schur straightening many ed} to each
term $\beperp_{d-1,m-1} \HH(\Psi^z ; \lambda - \epsilon_z) = \beperp_{d-1,m-1} \fs^{(k)}_{\lambda-\epsilon_z}$;
the hypotheses of the theorem are satisfied by Lemma \ref{l straightening many no overlap v2} and Remark \ref{rp bijection marked covers}.
The fourth equality is by
Propositions  \ref{p bijection marked covers} and \ref{p spins agree} (just as in the proof of Theorem~\ref{t flm k schur}).
The fifth equality is by the inductive hypothesis,
and the last equality holds since
a vertical strong marked tableau  $T$ of weight $(d)$ with  $\max(\markz(T)) = m$ is the same as a
vertical strong marked tableau  $V$ of weight $(1)$ with  $\markz(V) = \{m\}$
followed by a vertical strong marked
tableau  $U$ of weight $(d-1)$ such that  $\markz(U) \subset [m-1]$ and  $\inside(V) = \outside(U)$.
\end{proof}

\begin{example}
\label{ex partial res}
Let $k = 3$, $d = 2$, and $m=4$.
According to Theorem~\ref{t partial restriction},
$\beperp_{d,m}\fs^{(k)}_{222222}$ equals the sum of $t^{\spin(T)}s^{(k)}_{\inside(T)}$ over strong marked tableaux $T \in \VSMT^k_{(d)}(222222)$ such that $\markz(T) \subset [m]$, as shown:
{
\[
\begin{array}{rccccc}
T \text{      } & \text{\fontsize{7pt}{5pt}\selectfont \tableau{
~&~&~&~&~&1\\~&~&~&~&~&2\\~&~&~&\crc{1}\\~&~&~&\crc{2}\\~&1\\~&2\\}}&
\text{\fontsize{7pt}{5pt}\selectfont\tableau{
~&~&~&~&~&\crc{1}\\~&~&~&~&~&2\\~&~&~&1\\~&~&~&\crc{2}\\~&1\\~&2\\}}&
\text{\fontsize{7pt}{5pt}\selectfont\tableau{
~&~&~&~&~&\crc{1}\\~&~&~&~&~&\crc{2}\\~&~&~&1\\~&~&~&2\\~&1\\~&2\\}}&
\text{\fontsize{7pt}{5pt}\selectfont\tableau{
~&~&~&~&~&~\\~&~&~&1&\crc{1}&2\\~&~&~&1\\~&1&1&\crc{2}\\~&1\\~&2\\}}&
\text{\fontsize{7pt}{5pt}\selectfont\tableau{
~&~&~&~&~&~\\~&~&~&~&\crc{1}&2\\~&~&~&~\\~&~&1&\crc{2}\\~&~\\1&2\\}}\\[14mm]
\kskew(\inside(T)) & \text{\fontsize{7pt}{5pt}\selectfont\tableau{
\bl&\bl&\bl&~&~&\bl\\\bl&\bl&\bl&~&~&\bl\\\bl&~&~&\bl\\\bl&~&~&\bl\\~&\bl\\~&\bl\\}}&
\text{\fontsize{7pt}{5pt}\selectfont\tableau{
\bl&\bl&\bl&~&~&\bl\\\bl&\bl&\bl&~&~&\bl\\\bl&~&~&\bl\\\bl&~&~&\bl\\~&\bl\\~&\bl\\}}&
\text{\fontsize{7pt}{5pt}\selectfont\tableau{
\bl&\bl&\bl&~&~&\bl\\\bl&\bl&\bl&~&~&\bl\\\bl&~&~&\bl\\\bl&~&~&\bl\\~&\bl\\~&\bl\\}}&
\text{\fontsize{7pt}{5pt}\selectfont\tableau{
\bl&\bl&\bl&~&~&~\\\bl&~&~&\bl&\bl&\bl\\\bl&~&~&\bl\\~&\bl&\bl&\bl\\~&\bl\\~&\bl\\}}&
\text{\fontsize{7pt}{5pt}\selectfont\tableau{
\bl&\bl&\bl&\bl&~&~\\\bl&\bl&~&~&\bl&\bl\\\bl&\bl&~&~\\~&~&\bl&\bl\\~&~\\\bl&\bl\\}}\\[12mm]
\spin(T) &  2 & 3 & 4 & 4 & 3 \\[3mm]
\inside(T) & 222211 & 222211 & 222211 & 322111 & 222220\\[3mm]
\beperp_{2,4}\fs^{(3)}_{222222} ~ = &t^2\fs^{(3)}_{222211} \quad + & t^3\fs^{(3)}_{222211} \quad +& t^4\fs^{(3)}_{222211} \quad + &t^4\fs^{(3)}_{322111} \quad + &t^3\fs^{(3)}_{222220}.
\end{array}
\]
}
\end{example}

\begin{remark}
Theorem~\ref{t partial restriction} is somewhat delicate.
By definition,
\[\beperp_{d,m}\fs^{(k)}_\lambda = \sum_{S \subset V,\, |S| = d} \HH(\Delta^k(\lambda);\lambda-\epsilon_S).\]
However, each Catalan function $\HH(\Delta^k(\lambda);\lambda-\epsilon_S)$ in the sum is not necessarily equal to $\sum_{T \in \VSMT^k_{(d)}(\lambda),\, \markz(T) = S} t^{\spin(T)} \fs^{(k)}_{\inside(T)}$.
In fact, the Catalan functions $\HH(\Delta^k(\lambda);\lambda-\epsilon_S)$ are not always Schur positive, as can be seen from
a more detailed study of the previous example.  Setting $\Phi = \Delta^3(222222)$,  we have
\begin{align*}\beperp_{2,4}\fs^{(3)}_{222222} ~
&= \HH(\Phi;112222) + \HH(\Phi;121222) + \HH(\Phi;122122) \\
&\hspace{1pt}+ \HH(\Phi;211222) +\HH(\Phi;212122) +\HH(\Phi;221122),
\end{align*}
and these terms have the following expansions into 3-Schur functions:
\[\begin{array}{lll}
\HH(\Phi;112222) = t^4\fs^{(3)}_{222211}, &
\HH(\Phi;121222) = -t^3\fs^{(3)}_{222211}, &
\HH(\Phi;122122) = t^3\fs^{(3)}_{222211},\\
\HH(\Phi;211222) = t^3\fs^{(3)}_{222211}, &
\HH(\Phi;212122) = t^4\fs^{(3)}_{322111} + t^3\fs^{(3)}_{222220},&
\HH(\Phi;221122) = t^2\fs^{(3)}_{222211}.
\end{array}\]
Summing these, we recover the positive expression in Example \ref{ex partial res}.
The proof of Theorem~\ref{t partial restriction} implicitly handles cancellation of this form to produce
the positive sum in~\eqref{et partial restrict 1}.
\end{remark}

\subsection{The horizontal dual Pieri rule}
\label{ss horizontal dual Pieri}
We prove this rule (Property \eqref{et three properties 1}) using the vertical dual Pieri rule and a trick involving symmetric functions in noncommuting variables, inspired by \cite{FG,BF}.

Fix positive integers $\ell$ and $k$.  Recall that by Theorem \ref{t kschurs are a basis},
$\{\fs^{(k)}_\mu \mid \mu \in \Par^k_\ell \}$ is a basis for $\Lambda^k_\ell \subset \Lambda$.
We define the \emph{strong marked cover operators} $u_1, \dots, u_\ell \in \End(\Lambda^k_\ell)$ by
\begin{align}
\label{e strong marked cover operator}
\fs^{(k)}_\mu \cdot u_r = \sum_{T \in \SMT^k_{(1)}(\mu), \, \markz(T) = \{r\}} t^{\spin(T)}\fs^{(k)}_{\inside(T)}.
\end{align}
We have found it more natural here to work with right operators, so for this subsection we consider all operators to act on the right.
Note that
the set of  $T$ in the sum is just another notation for the
set of strong marked covers  $\tau \xRightarrow{~~r~~} \kappa$ with  $\p(\kappa) = \mu$.

For $d \ge 0$, define the following versions of the elementary and homogeneous symmetric functions in the (noncommuting) operators $u_i$:
\begin{align}
\tilde{e}_d ~&= ~\sum_{\ell \ge i_1 > i_2 > \cdots > i_d \ge 1}u_{i_1}u_{i_2}\cdots u_{i_d}\, ,\\
\tilde{h}_d ~&= ~\sum_{1 \le i_1 \le i_2 \le \cdots \le i_d \le \ell}u_{i_1}u_{i_2}\cdots u_{i_d}\, ;
\end{align}
by convention,
$\tilde{e}_0=\tilde{h}_0=1$ and $\tilde{e}_d = 0$ for $d > \ell$.

\begin{proof}[Proof of Property \eqref{et three properties 1}]
First, we have the basic identity
\begin{align}
\label{e E(-y)H(y) 3}
h_m^\perp-h_{m-1}^\perp\,e_1^\perp+h_{m-2}^\perp\,e_2^\perp-\cdots
+(-1)^m e_m^\perp = 0 \quad \ \text{for  $m > 0$}.
\end{align}
This follows directly from the
well-known identity
$\sum_{i=0}^m (-1)^ie_ih_{m-i} \!=\! 0$
\cite[Equation~7.13]{St}.
Hence the operators $h_1^\perp, h_2^\perp, \dots$ can be recursively expressed in terms of $e_1^\perp, e_2^\perp, \dots$.

The operators  $\tilde{h}_d$ and $\tilde{e}_d$ were cooked up so that $\fs^{(k)}_\mu \cdot \tilde{h}_d$ (resp. $\fs^{(k)}_\mu \cdot \tilde{e}_d$) equals
the right side of \eqref{et three properties 1} (resp. \eqref{et three properties 2}). 
Hence by Property \eqref{et three properties 2},  $e_d^\perp$ restricts to an operator on $\Lambda^k_\ell$ which agrees with $\tilde{e}_d$.
Then by the previous paragraph,  $h_d^\perp$ also restricts to an operator on  $\Lambda^k_\ell$
and Property \eqref{et three properties 1} is equivalent to the identity $h_d^\perp = \tilde{h}_d$ of operators on  $\Lambda^k_\ell$.
It now suffices to show that \eqref{e E(-y)H(y) 3} holds with $\tilde{e}_d$ and  $\tilde{h}_d$ in place of  $e_d^\perp$ and $h_d^\perp$.

Let $y$ be a formal variable that commutes with $u_1,\dots,u_\ell$.
The elements $\tilde{e}_d$ and $\tilde{h}_d$ can be packaged into the generating functions
\begin{align*}
\tilde{E}(y) &=\sum_{d = 0}^\ell y^d \tilde{e}_d
=(1+y u_\ell)\cdots (1+y u_1)\in \End(\Lambda^k_\ell)[y],\\
\tilde{H}(y) &= \sum_{d = 0}^\infty y^d \tilde{h}_d
=(1-y u_1)^{-1}\cdots (1-y u_\ell)^{-1}\in \End(\Lambda^k_\ell)[[y]].
\end{align*}
The identity $\tilde{H}(y)\tilde{E}(-y)=1$ implies that for any~$m>0$,
\begin{align}\label{e E(-y)H(y)}
\tilde{h}_m-\tilde{h}_{m-1}\,\tilde{e}_1+\tilde{h}_{m-2}\,\tilde{e}_2-\cdots
+(-1)^m \tilde{e}_m = 0,
\end{align}
as desired.
\end{proof}

\section{The Chen-Haiman root ideal for skew-linking diagrams}
\label{s Chen root set}


Consider a skew partition $\kappa/\eta$ and denote
its rows by $\lambda=(\kappa_1-\eta_1,\ldots,\kappa_\ell-\eta_\ell)$ and the rows of its transpose $(\kappa/\eta)'$ by  $\mu$.
If  $\lambda$ and  $\mu$ are both partitions, we say  $\kappa/\eta$ is a {\it skew-linking diagram}
and write $\skl{\lambda}{\kappa/\eta}{\mu}$.
Any $k$-skew diagram (defined in \S\ref{ss example strong tab}) is a skew-linking diagram;
this follows from the fact that the transpose of a $k+1$-core is again a $k+1$-core and the existence of the bijection $\p$.



\begin{definition}[{\cite[\S5.3]{ChenThesis}}]
\label{d Chen root ideal}
The root ideal associated to a skew-linking diagram $\skl{\lambda}{\kappa/\eta}{\mu}$ is given by
\begin{align*}
\Phi(\kappa/\eta) = \{(i,j) \in \Delta^+_{\ell(\kappa)}  \mid   i\leq\ell(\eta) \text{ and } j\geq\mu_{\eta_{i}}+i\}.
\end{align*}
Equivalently,  $\Phi(\kappa/\eta)$ is the root ideal with
removable roots $\{(i,\mu_{\eta_i}+i) \mid i\leq \ell(\eta)\}$.
\end{definition}
See Example \ref{ex skew link}.

Chen and Haiman conjectured that the Catalan functions include the $k$-Schur functions as a subclass but used
the root ideal  $\Phi(\kskew(\lambda))$
rather than our $\Delta^k(\lambda)$.
We reconcile this difference in Theorem \ref{t Chen roots} below.
This conjecture fits in the broader context of a study of Catalan functions associated to skew-linking diagrams.
We briefly mention two conjectures of Chen-Haiman arising in this study:
 for any skew-linking diagram $\skl{\lambda}{\kappa/\eta}{\mu}$, the Catalan function
$H(\Phi(\kappa,\eta); \lambda)$
(1) is the graded character of an explicitly constructed module for the symmetric group
(see \cite[Corollary~5.4.6]{ChenThesis}),
and
(2) is equal to the Schur positive sum $\sum t^{\charge(T)} s_{\sh(T)}$~over semistandard tableaux  $T$ satisfying certain katabolizability conditions
\cite[Conjecture~5.4.3]{ChenThesis}.

\begin{lemma}
\label{re:equaldowns}
Let $\skl{\lambda}{\kappa/\eta}{\mu}$ be a skew-linking diagram.
If $\kappa_y=\kappa_{y+1}$, then $\eta_y=\eta_{y+1}$ and $\lambda_y=\lambda_{y+1}$.
If  $\eta_y=\eta_{y+1} > 0$, then
$\kappa_{\mu_{\eta_{y}}+y}=\kappa_{\mu_{\eta_{y}}+y+1}$.
\end{lemma}
\begin{proof}
These facts are direct consequences of the definition of a skew-linking diagram. The first uses that $\lambda$ is a partition and the second that  $\mu$ is.
\end{proof}

\begin{lemma}
\label{le:forchen}
Let $\skl{\lambda}{\kappa/\eta}{\mu}$ be a skew-linking diagram and  $\Psi$
a root ideal which agrees with  $\Phi(\kappa/\eta)$ in rows $\ge y$.
Suppose that $\Psi$ has a removable root
of the form $\alpha = (a,y)$, $\Psi$ has a ceiling in columns $y, y+1$,
 and $\kappa_y= \kappa_{y+1}$.
Then $H(\Psi; \lambda) = H(\Psi \setminus \alpha; \lambda)$.
\end{lemma}
\begin{proof}
Apply Lemma \ref{l cascading toggle lemma} with indexed root ideal  $(\Psi, \lambda)$ and $w=\chaindown(y)$;
we need to verify the hypotheses \eqref{el little lemma many zero 2 1}--\eqref{el little lemma many zero 2 3}, \eqref{el cascading toggle lemma 4}, and  $w < \ell(\kappa)$.
The first, \eqref{el little lemma many zero 2 1}, holds by assumption.
Next, by Lemma \ref{re:equaldowns}, $\kappa_y = \kappa_{y+1}$ implies $\lambda_{y} = \lambda_{y+1}$ and $\eta_y = \eta_{y+1}$.
By  definition of $\Phi(\kappa/\eta)$,  if  $\eta_y > 0$, then $\down_{\Psi}(y) = \mu_{\eta_y}+y$ and thus $\kappa_{\down_\Psi(y)} = \kappa_{\down_\Psi(y)+1}$ by Lemma~\ref{re:equaldowns}.
Iterating this argument gives $\kappa_x = \kappa_{x+1}$, $\lambda_x = \lambda_{x+1}$, and  $\eta_x = \eta_{x+1}$ for all $x \in \downpath_\Psi(y)$,
which verifies \eqref{el little lemma many zero 2 2}, \eqref{el cascading toggle lemma 4}, and also
$w < \ell(\kappa)$ (since  $\kappa_{w} = \kappa_{w+1} > 0$).
Lastly, since  $\Psi$ has a removable root in row $i$ for $z \le i \le \ell(\eta)$ and no roots in rows $> \ell(\eta)$,
we have $w > \ell(\eta)$ and $\Psi$ has a wall in rows  $w, w+1$, i.e., \eqref{el little lemma many zero 2 3} holds.
\end{proof}

\begin{theorem}
\label{t Chen roots}
Let $\lambda \in \Par^k$ and $\kappa/\eta$ be the $k$-skew diagram of  $\lambda$. Then
\[
\fs^{(k)}_\lambda = H(\Delta^k(\lambda);\lambda) = H(\Phi(\kappa/\eta);\lambda).\]
\end{theorem}
It is most natural to regard $(\Delta^k(\lambda), \lambda)$ as an indexed root ideal of length $\ell(\lambda) = \ell(\kappa)$, and we do this in the proof below, however this is not strictly necessary by Proposition~\ref{p trailing zeros}.
\begin{proof}
Let $\mu$ denote the rows of $(\kappa/\eta)'$.
Let
\[\Psi^s = \{(i,j) \in \Delta^k(\lambda) \mid i < s\} \sqcup \{(i,j) \in \Phi(\kappa/\eta) \mid i \ge s\}.\]
We have  $\Psi^1 = \Phi(\kappa/\eta)$ and  $\Psi^{\ell(\eta)+1} = \Delta^k(\lambda)$ since
both $\Psi$ and $\Delta^k(\lambda)$ have no roots $(i,j)$ for $i>\ell(\eta)$;
namely, when $i>\ell(\eta)$, $\kappa/\eta$ is a $k$-skew diagram
implies that $k+1-\lambda_i+i\geq\mu_1+i>\ell(\eta)+\mu_1=\ell(\kappa)$.
It thus suffices to show $H(\Psi^{s+1};\lambda)=H(\Psi^s;\lambda)$ for all  $s \in [\ell(\eta)]$.

Let  $s \in [\ell(\eta)]$.
Since  $\kappa/\eta$ is a $k$-skew diagram, $\lambda_s+\mu_{\eta_s}>k$ and thus
\begin{align}\label{e psiy1}
\Psi^{s+1} = \Psi^s \sqcup \{(s,j)  \mid  j \in J\}, \text{ where $J = [k+1-\lambda_s+s,\mu_{\eta_s}+s-1]$}.
\end{align}
We have
\[H(\Psi^{s+1};\lambda)= H(\Psi^{s+1} \setminus (s,k+1-\lambda_s+s) ;\lambda) = \dots = H(\Psi^s;\lambda)\]
by repeated application of Lemma \ref{le:forchen}.
The hypotheses hold at each step by \eqref{e psiy1} and the facts
(1) $\Psi^{s+1}$ has a ceiling in columns $j,j+1$ for $j \in J$, and (2) $\kappa_j = \kappa_{j+1}$ for $j \in J$.
Fact (1) holds by \eqref{e psiy1} and
 $\down_{\Psi^{s+1}}(s+1) =  \mu_{\eta_{s+1}}+s+1 > \mu_{\eta_s}+s$.
For (2), note that the box $(s,\eta_s+1)$ has hook length $\le k$ in $\kappa$, and hence
the box  $(k+1-\lambda_s+s, \eta_s+1) \notin \kappa$.
Together with $(\mu_{\eta_s}+s,\eta_s) \in \kappa$, this implies
$\kappa_{k+1-\lambda_s+s} = \eta_s = \kappa_{\mu_{\eta_s}+s}$.
\end{proof}

\begin{example}
\label{ex skew link}
For $k= 7$ and $\lambda = 6643321111$,
here are the $k$-skew diagram $\kappa/\eta$ of $\lambda$ and the two associated Catalan functions from Theorem \ref{t Chen roots}:
\vspace{2mm}
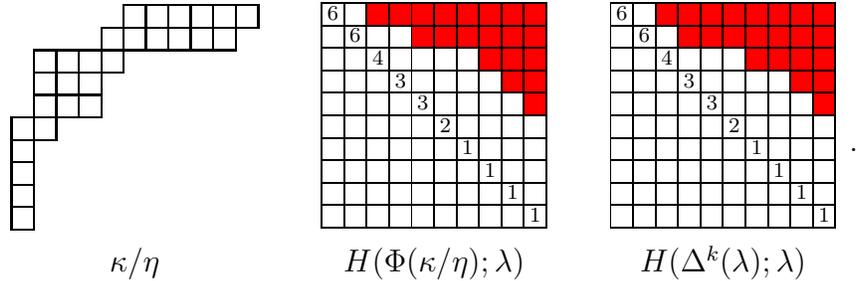
\[\begin{array}{ccccc}
{\fontsize{5pt}{4.2pt}\selectfont
\tableau{
\bl&\bl&\bl&\bl&\bl& ~ & ~ & ~ & ~ & ~ & ~ \\
\bl&\bl&\bl&\bl& ~ & ~ & ~ & ~ & ~ & ~ \\
\bl& ~ & ~ & ~ & ~ \\
\bl& ~ & ~ & ~  \\
\bl& ~ & ~ & ~  \\
 ~ & ~ \\
 ~ \\
 ~ \\
 ~ \\
 ~
}}
\ytableausetup{mathmode, boxsize=.96em,centertableaux}
\ & \ \   &{\tiny
\begin{ytableau}
   6    &        & *(red)  &*(red)   &*(red)  &*(red)   &*(red)  &*(red) &*(red)&*(red)\\
        &    6   &         &        &*(red)   &*(red)  &*(red) &*(red)  &*(red)&*(red)\\
        &        &    4   &        &         &          &       &*(red)  &*(red)&*(red)\\
        &        &         &    3    &        &        &        &        &*(red)&*(red)\\
        &         &         &         &    3   &       &         &       &      &*(red)\\
        &         &         &         &        &   2   &         &       &      & \\
        &         &         &         &        &       &   1    &        &      &\\
        &         &         &         &        &       &        &   1    &      &\\
        &         &         &         &        &       &        &        &  1   &\\
        &         &         &         &        &       &        &        &      &   1  \\
\end{ytableau}}
&\ \ &{\tiny
\begin{ytableau}
   6    &        & *(red)  &*(red)   &*(red)  &*(red)   &*(red)  &*(red) &*(red)&*(red)\\
        &    6   &         &  *(red) &*(red)   &*(red)  &*(red) &*(red)  &*(red)&*(red)\\
        &        &    4   &        &         &          &*(red) &*(red)  &*(red)&*(red)\\
        &        &         &    3    &        &        &        &        &*(red)&*(red)\\
        &         &         &         &    3   &       &         &       &      &*(red)\\
        &         &         &         &        &   2   &         &       &      & \\
        &         &         &         &        &       &   1    &        &      &\\
        &         &         &         &        &       &        &   1    &      &\\
        &         &         &         &        &       &        &        &  1   &\\
        &         &         &         &        &       &        &        &      &   1  \\
\end{ytableau}}\\[15mm]
\kappa/\eta && H(\Phi(\kappa/\eta); \lambda) && H(\Delta^k(\lambda); \lambda)
\end{array}.\]
\end{example}

\vspace{2mm}
\noindent
\textbf{Acknowledgments.}
We thank Elaine~So for help typing and typesetting figures.

\bibliographystyle{plain}
\bibliography{mycitations}
\def\cprime{$'$} \def\cprime{$'$} \def\cprime{$'$}
\end{document}